\documentclass[11pt,preprint,english]{amsart}

\usepackage[T1]{fontenc}
\usepackage[utf8]{inputenc} 
\usepackage{amsmath}
\usepackage{amsthm}
\usepackage{amssymb} 
\usepackage{mathrsfs}
\usepackage{xcolor}
\usepackage{mathtools}
\usepackage{enumitem}
\usepackage[bookmarks]{hyperref}
\usepackage{stmaryrd}
\usepackage{comment}
\usepackage{xcolor}

\theoremstyle{plain}
\newtheorem{thm}{\protect\theoremname}[section]
  \theoremstyle{plain}
  \newtheorem{fact}[thm]{\protect\factname}
  \theoremstyle{definition}
  \newtheorem{defn}[thm]{\protect\definitionname}
  \theoremstyle{plain}
  \newtheorem{lem}[thm]{\protect\lemmaname}
  \theoremstyle{plain}
  \newtheorem{prop}[thm]{\protect\propositionname}
  \theoremstyle{plain}
  
  \theoremstyle{plain}
  \newtheorem{ex}[thm]{\protect\exname}
  \theoremstyle{plain}
  \newtheorem{cor}[thm]{\protect\corname}
  \theoremstyle{definition}
  \newtheorem{nota}[thm]{\protect\notaname}
  \newtheorem{conv}[thm]{\protect\convname}
  \theoremstyle{plain}
  \newtheorem{quest}[thm]{\protect\questname}
  \theoremstyle{remark}
  \newtheorem{rem}[thm]{\protect\remname}

\usepackage{babel}
  \providecommand{\definitionname}{Definition}
  \providecommand{\factname}{Fact}
  \providecommand{\lemmaname}{Lemma}
  \providecommand{\propositionname}{Proposition}
\providecommand{\theoremname}{Theorem}
\providecommand{\exname}{Example}
\providecommand{\examplename}{Example}
\providecommand{\corname}{Corollary}
\providecommand{\notaname}{Notation}
\providecommand{\convname}{Convention}
\providecommand{\questname}{Question}
\providecommand{\remname}{Remark}

\newcommand{\frk}{\mathfrak}
\newcommand{\Lc}{\mathcal{L}}
\newcommand{\tp}{\mathrm{tp}}
\newcommand{\dc}{\#^{\mathrm{dc}}}
\newcommand{\cspan}{\overline{\mathrm{span}}}
\newcommand{\e}{\varepsilon}

\newcommand{\rectangle}{\mathrlap{\sqsubset}\sqsupset}

\newcommand{\brackconv}{= 0}

\newcommand{\dinf}{d}

\providecommand{\dotdiv}{
  \mathbin{
    \vphantom{+}
    \mathrm{
      \mathsurround=0pt 
      \ooalign{
        \noalign{\kern-.35ex}
        \hidewidth$\smash{\cdot}$\hidewidth\cr 
        \noalign{\kern.35ex}
        $-$\cr 
      }%
    }%
  }%
}

\newcommand{\vertiii}[1]{{\left\vert\kern-0.25ex\left\vert\kern-0.25ex\left\vert #1 
    \right\vert\kern-0.25ex\right\vert\kern-0.25ex\right\vert}}

\newcommand{\BLThm}{characterization}

\DeclareMathOperator{\cl}{cl}
\DeclareMathOperator{\tint}{int}

\makeatletter
\DeclareRobustCommand{\cset}{\@ifstar\star@cset\normal@cset}
\newcommand{\star@cset}[1]{\left\llbracket#1\right\rrbracket}
\newcommand{\normal@cset}[2][]{\mathopen{#1\llbracket}#2\mathclose{#1\rrbracket}}
\makeatother

\makeatletter
\DeclareRobustCommand{\mset}{\@ifstar\star@mset\normal@mset}
\newcommand{\star@mset}[1]{\left\{#1\right\}}
\newcommand{\normal@mset}[2][]{\mathopen{#1\{}#2\mathclose{#1\}}}
\makeatother

\makeatletter
\newcommand{\@imaxbig}[2]{\max\left\{#1,\allowbreak#2\right\}}
\newcommand{\@imaxsmall}[2]{\max\{#1,\allowbreak#2\}}
\newcommand{\@iminbig}[2]{\min\left\{#1,\allowbreak#2\right\}}
\newcommand{\@iminsmall}[2]{\min\{#1,\allowbreak#2\}}

\newcommand{\imax}{\@ifstar\@imaxsmall\@imaxbig}
\newcommand{\imin}{\@ifstar\@iminsmall\@iminbig}
\makeatother

\def\insep{inseparably}
\def\Insep{Inseparably}

\def\anesscont{an inherently non-discrete}
\def\esscont{inherently non-discrete}

\def\editcom#1{}

\begin{document}


\title{Strongly Minimal Sets and Categoricity in Continuous Logic}

\author{James Hanson}
\email{jehanson2@wisc.edu}
\address{Department of Mathematics, University of Wisconsin, Madison, WI 53706-1388, USA}

\date{\today}

\subjclass{Primary 03C66, 03C35, 03C45, Secondary 46B04}
\keywords{continuous logic, categoricity, strongly minimal sets}
\maketitle

\begin{abstract}
  The classical Baldwin-Lachlan characterization of uncountably categorical theories is known to fail in continuous logic in that not every \insep\ categorical theory has a strongly minimal set. Here we investigate these issues by developing the theory of strongly minimal sets in continuous logic and by examining \insep\ categorical expansions of Banach space.

  To this end we introduce and characterize `dictionaric theories,' theories in which definable sets are prevalent enough that many constructions familiar in discrete logic can be carried out. We also introduce, in the context of Banach theories, the notion of an `indiscernible subspace,' which we use to improve a result of Shelah and Usvyatsov \cite{SHELAH2019106738}. Both of these notions are applicable to continuous logic outside of the context of \insep\ categorical theories.

  Finally, we construct or present a slew of counterexamples, including an $\omega$-stable theory with no Vaughtian pairs which fails to be \insep\ categorical and an \insep\ categorical theory with strongly minimal sets in its home sort only over models of sufficiently high dimension. 




\end{abstract}

\section{Introduction}

The classical Baldwin-Lachlan characterization of uncountably categorical theories gives detailed structural information about such theories. In particular, each such theory has a strongly minimal set definable over its prime model, and every model of the theory is `controlled' by any such strongly minimal set. This structural picture fails in continuous logic in the context of \insep\ categorical theories. For example the theory of infinite dimensional Hilbert spaces ($\mathsf{IHS}$) does not contain anything resembling a strongly minimal set. 

Nevertheless, there is a meaningful notion of strongly minimal sets in continuous logic which is a non-trivial generalization of the notion in discrete logic. In this paper we will develop this machinery and examine \insep\ categorical theories which do contain strongly minimal sets, as well as the known class of structures---expansions of Banach space---which do not contain strongly minimal sets.\editcom{ Sentences too repetitive?}

\subsection{Overview}


After preliminary material in Section \ref{sec:prelim}, Section \ref{sec:VP} deals with a few mild variations on the notion of Vaughtian pairs in continuous logic. We show that none of them can occur in an \insep\ categorical theory.

Section \ref{sec:dict} deals with the issue of handling definable sets in continuous logic.  We introduce and characterize a niceness condition that ensures the prevalence of definable sets, called `dictionaricness,' All discrete theories are easily seen to be dictionaric, but a continuous theory may fail to be. We show that in dictionaric theories many of the manipulations which are trivial in discrete logic can be performed, or at least approximately performed.\editcom{ Expanded this to two sentences.} For example, the intersection of two definable sets can in general fail to be definable, but in a dictionaric theory given definable sets $D$ and $E$ there are arbitrarily small perturbations $E'$ of $E$ such that $D\cap E'$ is definable (Corollary~\ref{cor:strong_min_and_cat_in_cont:1}).
To end the section we show that all $\omega$-stable theories (and more generally totally transcendental theories) are dictionaric, despite the fact that superstable theories may fail to be. 

Section \ref{sec:SMS} deals with (strongly) minimal sets in continuous logic. 
The general development of strongly minimal sets in continuous logic resembles
their development in discrete logic, so much so that some of our proofs
follow Tent and Ziegler \cite{tent_ziegler_2012} closely, but there are
a few important differences:
\begin{itemize}
\item The definition needs to be slightly\editcom{} stronger than the most obvious
direct translation---``Every definable subset is either
compact or co-pre-compact.''---which does not work. There are
two problems with this definition:
\begin{itemize}
\item It might be that the set in question has many distinct non-algebraic
types but not enough definable subsets to distinguish between them.
\item A strongly minimal set can have a definable subset that is neither compact nor co-pre-compact. 
\end{itemize}
\item It is not clear whether or not the property of having no Vaughtian pairs is sufficient
to ensure that any minimal set is strongly minimal in an arbitrary
theory. There are two stronger hypotheses, both satisfied by $\aleph_1$-categorical
theories and both of which are strong enough to ensure that any minimal
set is strongly minimal. In particular, if $T$ is dictionaric and
has no Vaughtian pairs or if $T$ has no open-in-definable Vaughtian
pairs, then any minimal set is strongly minimal.
\item The proof that minimal sets over $\aleph_0$-saturated structures are
strongly minimal works without much modification, but in continuous logic sometimes we
need to work with approximately $\aleph_0$-saturated structures. We show that
minimal sets over approximately $\aleph_0$-saturated structures are
strongly minimal as well. (See \cite[Def.\ 1.3]{Yaacov2007} for the definition of approximately $\aleph_0$-saturated.)
\item In general, algebraic closure in continuous logic does not have finite
character---only countable character and approximate finite character---but in strongly minimal sets algebraic closure does have finite character.
This is notable because there are $\aleph_1$-categorical theories
in which $\mathrm{acl}$ is a pregeometry that does not have finite
character, namely $\mathsf{IHS}$.
\item In discrete logic, if $p$ is a strongly minimal type based on some
set $A$, then a corresponding strongly minimal set is definable over
$A$. In continuous logic, if $A$ is not a model, we need to invoke an extra assumption and
work with a weaker notion than strongly minimal set. 
\begin{itemize}
\item If $T$ is dictionaric, then for any strongly minimal $p$ based on some set $A$, there
is an $A$-definable `approximately strongly minimal set' corresponding
to $p$. 
\item Conversely (with no assumptions about $T$), if $D$ is an approximately
strongly minimal set definable over $A$, then it corresponds to a
strongly minimal type based on $A$. 
\item If $D$ is an approximately strongly minimal set definable
over $A$ and $\mathfrak{M}\supseteq A$ is any model, then there
is a $(D(\mathfrak{M})\cup A)$-definable strongly minimal set $E\subseteq D$ corresponding
to the same strongly minimal type (again with no assumptions about
$T$).
\editcom{ Removed `fortunately' from beginning of bullet point.}
\item Every known example of an $A$-definable (strongly) minimal set has
an $A$-definable (strongly) minimal imaginary quotient, but it is
not clear that this is always true.\editcom{ Changed `it's' to 'it is' and `if' to `that.'} 
\end{itemize}
\item In discrete logic if $S_{n}(\mathfrak{M})$ is topologically scattered
(if $T$ is totally transcendental, for instance), then every non-algebraic
open set contains a type that is minimal over $\mathfrak{M}$. In
continuous logic the correct topometric analog of topologically scattered (i.e.\
CB-analyzable \cite{BenYaacov2008}) is not strong enough to guarantee the existence of
any minimal or strongly minimal types. In particular, there is an $\aleph_1$-categorical
theory that does not even interpret a strongly minimal set, namely
$\mathsf{IHS}$.
\end{itemize}
Each of these differences with the exception of  the last one is either a mild technical
issue or the fortuitous lack of a potential mild technical issue.\editcom{ Changed `All' to `Each' and `except' to `with the exception of.'}  
The last difference is the most important one in that not every \insep\ categorical theory is in some way `based on' a strongly minimal set. 

Section \ref{sec:Bald-Lach} starts with a partial generalization of the Baldwin-Lachlan \BLThm\ to the context of continuous logic.
Subsections \ref{sec:theor-with-locally} and \ref{sec:ultr-theor-theor} 
presents a couple of common conditions that ensure the presence of minimal sets in an $\omega$-stable theory. Namely, theories with a locally compact model have minimal sets, and ultrametric theories have minimal imaginaries. We also characterize the relationship between ultrametric theories and theories with totally disconnected type spaces.\editcom{ Changed `In particular' to `Namely.'} 

We continue Section~\ref{sec:Bald-Lach} by presenting a counterexample to the most direct translation of the classical Baldwin-Lachlan classification, as well as  examples of \insep\ categorical theories that do have strongly minimal sets in their home sort, but only over models of sufficiently high dimension.\editcom{ Changed beginning of sentence. Added `in their home sort.'} 

Section~\ref{sec:ind-sub-abs} focuses on \insep\ categorical expansions of Banach space (or Banach theories). 
We present a novel proof that infinite dimensional Hilbert spaces do not interpret any non-trivial locally compact theories and in particular any strongly minimal theories, a fact which is implicit in \cite{Ben-Yaacov2004}. We also introduce the notion of an `indiscernible subspace' to present an improvement of the results of \cite{SHELAH2019106738}, after developing the relevant machinery. Following an observation regarding an application of dictionaricness to \insep\ categorical Banach theories, we present some relevant counterexamples, including one which resolves a question of \cite{SHELAH2019106738}.\editcom{ Changed `After' to `Following.'}


\section{ Preliminaries} \label{sec:prelim}


For the general formalism of continuous logic and the majority of the notation used here, see \cite{MTFMS}. A divergence that we should remark upon is that in the present work we will not bother with the distinction between formulas and definable predicates.

\begin{conv}
  A \emph{(real valued) formula in the variables $\bar{x}$ and over the parameters $A$} is a continuous function from $S_{\bar{x}}(A)$ to $\mathbb{R}$. We may say that a formula is \emph{$I$-valued} for some interval $I \subseteq \mathbb{R}$ (such as $[0,1]$ or $[-1,1]$) if we wish to highlight an intended restriction on the range of the formula. For any formula $\varphi(\bar{x})$ over the parameters $A$ and any tuple of elements $\bar{b}$ of the same sorts as $\bar{x}$, we understand $\varphi(\bar{a})$ or $\varphi^{\frk{M}}(\bar{a})$ to be $\varphi(\tp(\bar{b}/A))$, where $\frk{M}$ is some model containing $A\bar{b}$.
\end{conv}
Note that in general a formula may depend on up to countably many parameters (in the sense that if $\varphi: S_n(A)\rightarrow \mathbb{R}$ is a formula, then there is a countable $A_0\subseteq A$ such that $\varphi$ factors through the natural restriction map from $S_{\bar{x}}(A)$ to $S_{\bar{x}}(A_0)$). 

\begin{conv}
  In this paper when we refer to \emph{a discrete} metric or metric space, we mean a uniformly discrete metric or metric space (i.e.\ one for which there is an $\e > 0$ such that for any distinct $x$ and $y$, $d(x,y) > \e$).

  \emph{The discrete metric} refers to the metric on a given set for which any two distinct $x$ and $y$ have $d(x,y)=1$.
\end{conv}

\subsection{Notation}
\label{sec:nota}

Here are the notational conventions used in this paper that are not found in \cite{MTFMS}.

\begin{nota}
Let $(X,d)$ be a metric space. Let $x\in X$, $A \subseteq X$, and $\varepsilon > 0$.
\begin{enumerate}[label=(\roman*)]
\item $B_{\leq \varepsilon}(x) = \{y \in X : d(x,y) \leq \varepsilon \}$

\item $B_{< \varepsilon}(x) = \{ y \in X : d(x,y) < \varepsilon \}$

\item $d(x,A) = \inf\{d(x,y):y \in A\}$

\item Given $B \subseteq X$, the \emph{Hausdorff distance between $A$ and $B$}, written $d_H(A,B)$, is $\max\{\sup_{a\in A}\dinf(a,B),\sup_{b\in B}\dinf(b,A)\}$. 

\item $A^{\leq \varepsilon} = \{y \in X : d(x,A) \leq \varepsilon\}$

\item $A^{< \varepsilon} = \{y \in X : d(x,A) < \varepsilon \} = \bigcup_{y\in A}B_{<\varepsilon}(y) $

\item $\#^\mathrm{dc} A$, the \emph{metric density character of $A$}, is the minimum cardinality of a metrically dense subset of $A$.

\item We say that $B\subseteq X$ is \emph{$({>}\varepsilon)$-separated} if $d(x,y) > \varepsilon$ for every pair of distinct points $x,y\in B$.

\item  $\#^\mathrm{ent}_{>\varepsilon} A  = \sup\{|B|\,:\,B \subseteq A\text{ and }B\text{ }({>}\varepsilon)\text{-separated}\}$. This quantity is called the \emph{$(>\varepsilon)$-metric entropy of $A$}.\footnote{There is some inconsistent usage of the term `metric entropy' in the literature. Sometimes it is the logarithm of our $\#^{\mathrm{ent}}_{>\e} A$ or some similarly defined quantity.} (Note that this supremum is not always attained.)

\item $\overline{A}=\bigcap_{\e > 0}A^{<\e}$ is the \emph{metric closure} of $A$.
\end{enumerate}
\end{nota}

$d_H(A,B)$, which can be equivalently defined as $\inf\{\e > 0 : A \subseteq B^{<\e}, B \subseteq A^{<\e}\}$, is also called the \emph{Hausdorff metric}, although note that in general it is only a pseudo-metric unless restricted to closed sets.

\begin{nota}
  In situations in which both a topology and a metric are relevant (such as in a type space), we will use $\cl A$ to represent the topological closure, as opposed to $\overline{A}$, the metric closure. The intended meaning will also be reinforced by context and occasionally explicit reminders. We will use $\tint A$ to represent the topological interior. We do not need notation for metric interior, as we never use the concept.

  If we need to specify a closure or interior relative to some subspace, we will use subscripts. $\cl_B A$ is the closure of $A$ in the subspace $B$, for instance.
\end{nota}

We take the $^{<\e}$ and $^{\leq \e}$ operators to bind more tightly than infix and prefix operators, so $\tint A^{<\e}$ is $\tint (A^{<\e})$ and $B\cup C^{\leq \delta}$ is $B \cup (C^{\leq \delta})$.\editcom{New sentence. }

\begin{nota}
If $S_n(A)$ is a type space and $\varphi$ is an $A$-formula with free variables amongst $x_0,\dots,x_{n-1}$, we will let $\cset{\varphi < \varepsilon}$ denote the set of types $p \in S_n(A)$ satisfying $p \models \varphi < \varepsilon$. The type space in question will be clear from context. Likewise for $\cset{\varphi \leq \varepsilon}$, $\cset{\varphi > \varepsilon}$, $\cset{\varphi \geq \varepsilon}$, $\cset{\varphi = \varepsilon}$, and $\cset{\varphi \neq \varepsilon}$. 
If it is necessary to denote parameters we may write something like $\cset{\varphi(-,\overline{a}) < \varepsilon}$ to avoid specifying free variables.  We may also write something like $\cset{\varepsilon \leq \varphi < \delta}$ to mean $\cset{\varphi \geq \varepsilon} \cap \cset{\varphi < \delta}$.

If $\mathfrak{M}$ is a structure and $\rectangle\; \in \{<,>,\leq,\geq,=,\neq\}$, then $\cset{\varphi(\mathfrak{M})\;\rectangle \varepsilon}$ is the set of tuples $\overline{a} \in \mathfrak{M}$ such that $\mathfrak{M} \models \varphi\;\rectangle \varepsilon$. Expressions such as $\cset{\varepsilon \leq \varphi(\mathfrak{M}) < \delta}$ have the obvious interpretation. If $F$ is an arbitrary (typically closed) subset of the type space $S_n(A)$, then we will use $F(\frk{M})$ similarly.
\end{nota}
To avoid confusion with the established logical roles of $\wedge$ and $\vee$, we will avoid using this notation for $\min$ and $\max$.

\subsection{Basic Facts and Definitions}

\label{sec:basic-facts}

\begin{defn}
  A \emph{zeroset} is a set (of types) of the form $\cset{\varphi = 0}$.
\end{defn}
Sets of the form $\cset{\varphi\leq \psi}$, $\cset{\varphi\geq\psi}$, and $\cset{\varphi=\psi}$ can be equivalently expressed as zerosets and will be freely referred to as such. Note that a subset of a given type space is a zeroset if and only if it is a closed $G_\delta$ set, by Urysohn's lemma.

\begin{defn}
A zeroset $F\subseteq S_n(A)$ is \emph{algebraic} if for every $\mathfrak{M} \supseteq A$, $F(\mathfrak{M})$ is metrically compact.
\end{defn}

The following lemma is a way to verify the algebraicity of a given zeroset in a single model.

\begin{lem} \label{lem:basic-alg}
If $F$ and $G$ are $M$-zerosets such that $F\subseteq \tint G$
and $G(\mathfrak{M})$ is metrically compact, then $F$ is algebraic.
\end{lem}
\begin{proof}
Find a formula $\varphi(x;\overline{a})$ with $\overline{a}\in \mathfrak{M}$ such that $F \subseteq \cset{\varphi(-;\overline{a})<\frac{1}{3}} \subseteq \cset{\varphi(-;\overline{a}) \leq \frac{2}{3}} \subseteq \tint G$. For each $\varepsilon>0$ there is an $n_\varepsilon < \omega$ such that $\cset{\varphi(\mathfrak{M};\overline{a}) \leq \frac{2}{3}}$ contains a $({>}\varepsilon)$-separated set of cardinality $n_\varepsilon$, but does not contain one of cardinality $n_\varepsilon + 1$, so in particular if we consider the sentence
\[\chi = \inf_{x_0,\dots,x_{n_\varepsilon}} \max_{i<j\leq n_\varepsilon} \imax{\frac{1}{\varepsilon}\left(\varepsilon \dotdiv d(x_i,x_j)\right) }{ \frac{3}{2}\imin{\varphi(x_i;\overline{a}) }{ \frac{2}{3}}},\]
we have $\mathfrak{M} \models \chi = 1$. This will be true of any elementary extension of $\mathfrak{M}$ as well, so if we assume that in some $\mathfrak{N} \succ \mathfrak{M}$ we have that $F(\mathfrak{N})$ is not metrically compact, then for some $\varepsilon>0$ there is an infinite $({>}\varepsilon)$-separated set of elements in $F(\mathfrak{N})$, but this would imply that $\mathfrak{N} \models \chi < 1$ since any element of $F(\mathfrak{N})$ is also an element of $\cset{\varphi(\mathfrak{N};\overline{a}) < \frac{1}{3}}$. This is a contradiction, so $F$ must be algebraic. 
\end{proof}

\begin{defn}\label{defn:strong_min_and_cat_in_cont:1}
\leavevmode
\begin{enumerate}[label=(\roman*)]
\item Given a closed set $F\subseteq S_{n}(T)$, a closed set $Q\subseteq F$
is \emph{definable relative to $F$} (or relatively definable in $F$, or definable-in-$F$) if for every
$\varepsilon>0$ there is an open-in-$F$ set $U$ such that $Q\subseteq U\subseteq Q^{<\varepsilon}$. 

\item A point $p\in S_{n}(T)$ is \emph{$d$-atomic} if $\{p\}$ is a definable
set. It is relatively $d$-atomic in $F$ if $\{p\}$ is a relatively definable
subset of $F$.\footnote{We have chosen the term `$d$-atomic' instead of `$d$-isolated' (as in \cite{BenYaacov2008})  to avoid confusion with `isolated with respect to $d$' and to emphasize that this notion plays a role analogous to that of atomic (as in principal) types in discrete logic.}    
\end{enumerate}
\end{defn}

Note that we have defined `relatively definable' sets in a way that is distinct from a common understanding of the notion in discrete logic, namely those subsets of a given type-definable set (i.e.\ a given closed set) which are the intersection of it with a definable set (i.e.\ a clopen set). These two definitions are equivalent in discrete logic, but are not equivalent in continuous logic. We will show in Proposition~\ref{prop:ext} that, under some assumptions, one direction of the equivalence holds, i.e.\ every relatively definable $Q \subseteq F$ is actually of the form $D\cap F$ for some definable set $D$.

\begin{lem}\label{lem:strong_min_and_cat_in_cont:4}
Given $Q\subseteq F\subseteq S_{n}(T)$, with $Q$ and $F$ closed,
$Q$ is definable relative to $F$ if and only if there is a formula
$\varphi:S_{n}(T)\rightarrow[0,1]$ such that $F\cap\cset{\varphi\brackconv}=Q$ 
and for all $x\in F$, $d(x,Q)\leq \varphi(x)$.
\end{lem}

\begin{proof}
$(\Leftarrow)$. Assuming that such a function exists, for every $\varepsilon>0$
the set $F\cap\cset{\varphi<\frac{\varepsilon}{2}}$ is relatively open in $F$ and by construction
we have that $Q\subseteq F\cap\cset{\varphi<\frac{\varepsilon}{2}}\subseteq Q^{<\varepsilon}$.

$(\Rightarrow)$. Assume without loss of generality that the metric diameter of $F$
is at most $1$. Assume that $Q$ is definable relative to $F$. Let
$U_{0}=F$ and $\varepsilon_{0}=1$. For each $n<\omega$, find $\varepsilon_{n}>0$
small enough that $\varepsilon_{n}<2^{-n}$, $\varepsilon_{n}<\varepsilon_{n-1}$,
and $Q^{\leq\varepsilon_{n}}\subseteq U_{n-1}$ (which must exist
by compactness). Then let $U_{n}=\tint Q^{\leq\varepsilon_{n}}$,
which by assumption contains $Q$. Note that by construction $\bigcap_{n<\omega}U_{n}=\bigcap_{n<\omega}Q^{\leq\varepsilon_{n}}=Q$.

Since $F$ is a compact Hausdorff space, and therefore normal, we
can find for each $n<\omega$ a continuous functions $\varphi_{n}:F\rightarrow[0,1]$
such that $\varphi_{n}\upharpoonright Q^{\leq\varepsilon_{n}}=0$ and $\varphi_{n}\upharpoonright(F\smallsetminus U_{n-1})=1$.
Now let $\varphi(x)=\sum_{n<\omega}(\varepsilon_{n}-\varepsilon_{n+1})\varphi_{n+3}(x)$.
Note that this is a uniformly convergent series of continuous functions
(since $\varepsilon_{n}$ is a monotonically decreasing sequence of
positive numbers and since the $\varphi_{n}$ are uniformly bounded), so
$\varphi(x)$ is a continuous function. Furthermore note that the zeroset
$\cset{\varphi\brackconv}$ is precisely $Q$. Finally note that if $\varepsilon_{n}<d(x,Q)\leq\varepsilon_{n-1}$,
then we have that $x\notin Q^{\leq\varepsilon_{n}}$ and in particular
$x\notin U_{n}$, so $x\in F\smallsetminus U_{n}$. This implies that
$\varphi_{k}(x)=1$ for all $k\geq n+1$, so we have that $\varphi(x)$ is at
least $\varepsilon_{n-2}>\varepsilon_{n-1}$ (or $1$ if $n-2<0$).
If $x\in Q$ then $d(x,Q)=\varphi(x)=0$, so we have that for any $x\in F$,
$d(x,Q)\leq \varphi(x)$, as required.

Finally since $S_{n}(T)$ is a normal space we can use the Tietze
extension theorem to extend $\varphi$ to a continuous function on all of
$S_{n}(T)$.
\end{proof}

In type spaces the function in Lemma~\ref{lem:strong_min_and_cat_in_cont:4} can be taken to be the actual distance predicate of the set, but in arbitrary subspaces this is not always possible. 

\begin{lem} \label{lem:def-rel-def}
For $D$, $E$, and $F$ closed sets, if $D\subseteq F$ is definable
relative to $F$ and $E\subseteq D$ is definable relative to $D$,
then $E$ is definable relative to $F$. In particular if $D$ is
a definable set and $E$ is definable relative to $D$ then $E$ is
definable.
\end{lem}

\begin{proof}
Using the previous lemma and the Tietze extension theorem, let $P_{0}$
be a continuous function on $F$ such that $\cset{P_{0}\brackconv}=D$ and $d(x,D)\leq P_{0}(x)$
for all $x\in F$. Furthermore let $P_{1}$ be a continuous function
on $F$ such that $D\cap\cset{P_{1}\brackconv}=E$ and such that $d(x,E)\leq P_{1}(x)$
for all $x\in D$. Since $P_{1}$ is continuous and $F$ is a compact
topometric space, $P_1$ is uniformly continuous with regards to the metric.
Let $\alpha:[0,1]\rightarrow[0,1]$ be a continuous, non-decreasing
modulus of uniform continuity of $P_{1}$ (i.e.\ a function satisfying $\alpha(0)=0$ and $|P_{1}(x)-P_{1}(y)|\leq\alpha(d(x,y))$
for all $x$ and $y$). Finally let $P(x)=\imin{P_{0}(x)+\alpha(P_{0}(x))+P_{1}(x)}{1}$.\editcom{ Move $\alpha(0) = 0$.} 

First note that the zeroset of $P(x)$ is precisely $E$. Furthermore,
for any $x\in F$ we can find $y\in D$ such that $d(x,y)=d(x,D)$
(by compactness), so we have that $d(x,E)\leq d(x,y)+d(y,E)\leq d(x,D)+P_{1}(y)\leq P_{0}(x)+P_{1}(x)+\alpha(d(x,y))\leq P_{0}(x)+P_{1}(x)+\alpha(P_{0}(x))$.
So since $d(x,E)\leq1$, we have that $d(x,E)\leq P(x)$, as required.
\end{proof}

In discrete logic it is a basic fact that if $\tp(ab)$ is atomic, then $\tp(a/b)$ is atomic as well. This may fail in continuous logic \cite[Rem.\ 12.14]{MTFMS}, but some approximation of it is still true, which is summarized in the following lemma.

\begin{lem}
If $\mathrm{tp}(a\overline{b})$ is atomic and $\mathfrak{M}$
is any model containing $\bar{b}$, then for every $\varepsilon>0$ there exists $a^{\prime}\overline{b}^{\prime}\in\mathfrak{M}$
such that $a\overline{b}\equiv a^{\prime}\overline{b}^{\prime}$ and
$d(\overline{b},\overline{b}^{\prime})<\varepsilon$.
\end{lem}

\begin{proof}
Let $D(x,\overline{y})$ be a distance predicate for the definable
set $\mathrm{tp}(a\overline{b})$. We have that $\mathfrak{M}\models\inf_{x}D(x,\overline{b})=0$,
since this is part of the type of $\overline{b}$. Let $a_{0}$ be
such that $D(a_{0},\overline{b})<\varepsilon$. Since $D$ is a distance
predicate this implies that there is $a^{\prime}\overline{b}^{\prime}\in D(\mathfrak{M})$
with $d(a_{0}\overline{b},a^{\prime}\overline{b}^{\prime})<\varepsilon$.
$a^{\prime}\overline{b}^{\prime}$ is the required sequence of elements. 
\end{proof}

The following proposition and corollary are almost certainly known but we could not find them in the literature. They are very similar to Fact 1.5 in \cite{Yaacov2007}. 

\begin{prop} \label{Prop:Prime-Hom}
If a countable theory $T$ has a prime model, then it has a unique
prime model. Moreover if $\overline{a}\equiv\overline{b}$ with $\overline{a}\in\mathfrak{M}$
and $\overline{b}\in\mathfrak{N}$, with $\mathfrak{M}$ and $\mathfrak{N}$
prime, then for every $\varepsilon>0$ there is an isomorphism $f:\mathfrak{M}\rightarrow\mathfrak{N}$
such that $d(f(\overline{a}),\overline{b})<\varepsilon$.
\end{prop}

\begin{proof}
Let $\mathfrak{M}$ and $\mathfrak{N}$ be prime. Let $\overline{a}\in\mathfrak{M}$
and $\overline{b}\in\mathfrak{N}$ with $\overline{a}\equiv\overline{b}$.
Fix $\varepsilon>0$. By the omitting types theorem, every type realized
in $\mathfrak{M}$ or $\mathfrak{N}$ is atomic. Let $\{a_{i}\}_{i<\omega}$
be an enumeration of a tail dense sequence in $\mathfrak{M}$ (i.e.\ a sequence in which
every final segment is dense), starting with $\overline{a},$ and
let $\{b_{i}\}_{i<\omega}$ be an enumeration of a tail dense sequence
in $\mathfrak{N}$ , starting with $\overline{b}$. We are going to
build an array $\{c_{j}^i\}_{i<\omega,j\leq i}\subseteq\mathfrak{M}$
and an array $\{e_{j}^i\}_{i<\omega,j\leq i}\subseteq\mathfrak{N}$
with the following properties:
\begin{itemize}
\item
For every $i<\omega$ and $j\leq i$, $d(c_{j}^i,c_{j}^{i+1})<2^{-i-2}\varepsilon$
and $d(e_{j}^i,e_{j}^{i+1})<2^{-i-2}\varepsilon$.
\item
At each stage $i$, $c_{0}^i c_{1}^i \dots c_{i}^i \equiv e_{0}^i e_{1}^i \dots e_{i}^i$.
\item
For every $k<\omega$, $c_{2k}^{2k}=a_{k}$ and $e_{2k+1}^{2k+1}=b_{k}$.
\end{itemize}
For the initial step, let $c_{0}^0=a_{0}$. Since $\mathrm{tp}(a_{0})$
is atomic, we can find $e_{0}^0\in\mathfrak{M}$ such that $c_{0}^0 \equiv e_{0}^0$.

On odd step $2k+1$, given $c_{\leq 2k}^{2k}$ and
$e_{\leq 2k}^{2k}$ satisfying $c_{\leq 2k}^{2k}\equiv e_{\leq 2k}^{2k}$,
let $e_{j}^{2k+1}=e_{j}^{2k}$ for $j\leq2k$ and let $e_{2k+1}^{2k+1}=b_{k}$.
In particular note that $d(e_{j}^{2k},e_{j}^{2k+1})=0<2^{-2k-2}\varepsilon$
for each $j\leq2k$. Now by the lemma we can find $c_{\leq 2k+1}^{2k+1}\in\mathfrak{M}$
such that $d(c_{\leq 2k}^{2k},c_{\leq 2k}^{2k+1})<2^{-2k-3}\varepsilon$
and such that $c_{\leq 2k+1}^{2k+1}\equiv e_{\leq 2k+1}^{2k+1}.$

On even step $2k+2$, given $c_{\leq 2k+1}^{2k+1}$
and $e_{\leq 2k+1}^{2k+1}$ satisfying $c_{\leq 2k+1}^{2k+1}\equiv e_{\leq 2k+1}^{2k+1}$,
let $c_{j}^{2k+2}=c_{j}^{2k}$ for $j\leq2k+1$ and let $c_{2k+2}^{2k+2}=a_{k}$.
In particular note that $d(c_{j}^{2k+1},c_{j}^{2k+1})=0<2^{-2k-3}\varepsilon$
for each $j\leq2k+1$. Now by the lemma we can find $e_{\leq 2k+2}^{2k+2}\in\mathfrak{N}$
such that $d(e_{\leq 2k+1}^{2k+1},e_{\leq 2k+1}^{2k+2})<2^{-2k-4}\varepsilon$
and such that $e_{\leq 2k+2}^{2k+2}\equiv c_{\leq 2k+2}^{2k+2}$.

For each $j<\omega,$ let $c_{j}=\lim_{i\rightarrow\omega}c_{j}^{i}$
and $e_{j}=\lim_{i\rightarrow\omega}e_{j}^{i}$. By construction these
are all Cauchy sequences and thus have limits in $\mathfrak{M}$ and
$\mathfrak{N}$ respectively. Also by uniform continuity of formulas
we have that $c_{0}c_{1}\dots\equiv e_{0}e_{1}\dots$. 

We need to show that the sequences $c_{<\omega}$ and $e_{<\omega}$
are metrically dense in $\mathfrak{M}$ and $\mathfrak{N}$ respectively.
This follows by the fact that the sequences $a_{<\omega}$ and $b_{<\omega}$
were tail dense. For any $x\in\mathfrak{M}$ and any $\delta>0$ there
is an $a_{i}$ such that $d(x,a_{i})<\frac{1}{2}\delta$ and $d(a_{i},c_{j})<\frac{1}{2}\delta$
for some $c_{j}$, implying that $d(x,c_{j})<\delta$. Therefore $c_{<\omega}$
is dense in $\mathfrak{M}$. Likewise for $e_{<\omega}$ in $\mathfrak{N}$,
so we have that the elementary map $f_{0}:c_{<\omega}\rightarrow e_{<\omega}$
extends to an isomorphism $f:\mathfrak{M}\rightarrow\mathfrak{N}$. 

Now note that if $\overline{c}$ is the initial segment of $c_{<\omega}$
of the same length as $\overline{a}$, then we have by construction
that $d(\overline{c},\overline{a})<\sum_{i<\omega}2^{-i-2}\varepsilon=\frac{1}{2}\varepsilon$.
Likewise if $\overline{e}$ is the initial segment of $e_{<\omega}$
of the same length as $\overline{b}$, then by construction we have
that $d(\overline{e},\overline{b})<\frac{1}{2}\varepsilon$. Therefore
we have that $d(f(\overline{a}),\overline{b})\leq d(f(\overline{a}),f(\overline{c}))+d(\overline{e},\overline{b})<\frac{1}{2}\varepsilon+\frac{1}{2}\varepsilon$,
as required.
\end{proof}

\begin{cor}
If $\mathfrak{M}$ is a prime model of a countable theory, then
it is approximately $\aleph_0$-homogeneous (i.e.\ satisfies that for any finite tuples $\overline{a}\equiv \overline{b}$ and for any $\varepsilon>0$ there is an automorphism $\sigma$ of $\mathfrak{M}$ such that $d(\sigma(\overline{a}),\overline{b}) < \varepsilon$).
\end{cor}

Unlike in discrete logic, in continuous \insep\ categorical theories it may be necessary to look in imaginary sorts for strongly minimal sets (see Example \ref{ex:ult-example}, the unit ball of a $p$-adic $L^\infty$ space). \cite[Sec.\ 11]{MTFMS} alludes to the fact that the notion of imaginary sort in continuous logic ought to be more complicated than the corresponding notion in discrete logic, and \cite{10.2307/25747120} sketches this construction. For the sake of precision we will give an exact definition of what we mean by an imaginary sort here, as it is perhaps a little more complicated than someone familiar with discrete logic might suppose. To summarize the difference informally, we need to allow $\omega$-products and passing to definable subsets in addition to the familiar product and quotient sorts (where quotients are with regards to definable pseudo-metrics, not just definable equivalence relations). Note that this notion is more permissive than the traditional notion for discrete theories interpreted as continuous theories.\editcom{ Moved `than the traditional notion.'} 

\begin{defn}
  Given two pseudo-metrics $d_0$ and $d_1$ on some fixed set $X$, we say that $d_0$ \emph{uniformly dominates} $d_1$ if for every $\e > 0$, there is a $\delta>0$ such that for all $x,y \in X$, if $d_0(x,y)<\delta$, then $d_1(x,y)<\e$.

  We say that $d_0$ and $d_1$ are \emph{uniformly equivalent} if $d_0$ uniformly dominates $d_1$ and $d_1$ uniformly dominates $d_0$.
\end{defn}

\begin{defn} \label{defn:imag}
\emph{Imaginary sorts} are defined inductively:
\begin{itemize}
\item
The home sort is an imaginary sort.
\item
Any product of at most countably many imaginary sorts is an imaginary sort. A finite product sort is given the $\max$ metric, unless otherwise stated, and an $\omega$-product of sorts with metrics $\{d_k\}_{k<\omega}$ is given a metric of the form $\max_{k<\omega} a_k d_k$, where $a_k$ is a sequence of positive real numbers such that $\lim_{k\rightarrow \infty} a_k D_k = 0$ and $D_k$ is the diameter bound of the sort corresponding to $d_k$. We will typically take $a_k = 2^{-k}$.\footnote{Note that any such metric is always definable on the sort in question and any two such metrics are always uniformly equivalent.}
\item
Any quotient of an imaginary sort by a definable pseudo-metric is an imaginary sort.
\item
Any definable subset of an imaginary sort is an imaginary sort.
\end{itemize}
\end{defn}


One thing that should be noted is that, when talking about models, a quotient sort may contain more elements than the actual image of the corresponding quotient map, as we implicitly pass to the metric completion.

Despite the apparent iterative complexity of the collection of imaginaries, all imaginaries can be definably re-expressed in a simple normal form:

\begin{lem}[Imaginary Normal Form]\label{lem:imaginary-norm-form}
  For any parameter set $A$, any $A$-definable imaginary has an $A$-definable bijection with an $A$-definable subset of a $\varnothing$-definable quotient of $H^\omega$, the sort of $\omega$-tuples from the home sort.
\end{lem}
\begin{proof}
First we will (a) prove the lemma for $A = \varnothing$, then we will (b) prove that any $A$-definable quotient and a $\varnothing$-definable imaginary has an $A$-definable bijection with an $A$-definable subset of a $\varnothing$-definable imaginary. Then, finally, we will use this to (c) prove the full lemma.

(a). We'll say that an imaginary that is a definable subset of a quotient of $H^\omega$  is in `normal form.' We'll show by structural induction that every imaginary is in definable bijection with one in normal form. The home sort is in definable bijection with a quotient of $H^\omega$, namely the one that ignores all of the $\omega$-tuple except for the first element.

Assume that $I$ is an imaginary that is a definable subset of an imaginary in normal form. Since a relatively definable subset of a definable set is definable (by Lemma \ref{lem:def-rel-def}), we have that $I$ can be re-expressed in normal form.

Assume that $I$ is an imaginary that is a definable quotient of an imaginary in normal form. First we need to show that if $D$ is a definable set in some sort $S$ and $\rho$ is a pseudo-metric on $D$, then there is a pseudo-metric $\rho^\prime$ on $S$ such that $\rho^\prime \upharpoonright D = \rho$. Assume without loss of generality that $d$ and $\rho$ are $[0,1]$-valued. Since $\rho$ is a definable pseudo-metric, it is uniformly dominated by $d$, so in particular we can find a continuous function $\alpha:[0,1]\rightarrow [0,1]$ such that
\begin{itemize}
\item  $\alpha(x)=0$ if and only if $x=0$,
\item  $\alpha$ is concave down and non-decreasing,
\item  and $\rho(x,y) \leq \alpha(d(x,y))$ for all $x,y$.
\end{itemize}
Note that by the first two conditions we have that $\alpha(d)$ is a metric uniformly equivalent to $d$. Consider the formula
\[
  \rho^\prime (x,y) = \inf_{z,w\in D} \alpha(d(x,z)) + \rho(z,w) + \alpha(d(w,y)).
\]
  First to see that $\rho^\prime \upharpoonright D = \rho$, for any $x,y,z,w \in D$ we have $\rho(x,y) \leq \rho(x,z) + \rho(z,w) + \rho(w,y) \leq \alpha(d(x,z)) + \rho(z,w) + \alpha(d(w,y))$, so $\rho(x,y) \leq \rho^\prime (x,y)$. But clearly $\rho^\prime (x,y) \leq \rho(x,y)$ so we have $\rho(x,y) = \rho^\prime (x,y)$ for $x,y\in D$ as required. $\rho^\prime$ clearly satisfies $\rho^\prime(x,x) = 0$, $\rho^\prime(x,y) \geq 0$, and $\rho^\prime (x,y) = \rho^\prime (y,x)$, so we only need to verify that $\rho^\prime$ satisfies the triangle inequality. Pick $x,y,z$. For any $u,v,s,t \in D$ clearly we have
  \begin{align*}
\rho(u,t) &\leq \rho(u,v)+\rho(v,s)+\rho(s,t) \leq \rho(u,v)+\alpha(d(v,s))+\rho(s,t) \\
 &\leq \rho(u,v) + \alpha(d(v,y))+\alpha(d(y,s)) +\rho(s,t).
  \end{align*}
So we have
\begin{align*}
  \rho^\prime(x,z) &\leq \alpha(d(x,u)) + \rho(u,t) + \alpha(d(t,z)) \\
&\leq \alpha(d(x,u))+\rho(u,v)+\alpha(d(v,y)) + \alpha(d(y,s))+\rho(s,t) + \alpha(d(t,z)) .
\end{align*}
So since this is true for any $u,v,s,t \in D$, we must have 
\[\rho^\prime(x,z) \leq \rho^\prime(x,y) + \rho^\prime(y,z),\]
 as required.

So we have a definable pseudo-metric $\rho^\prime$ extending $\rho$ to all of $S$. Now we need to argue that the (metric closure of the) image $D/\rho^\prime$ is a definable subset of $S/\rho^\prime$, 
 but this is trivial since the distance predicate of $D/\rho^\prime$ is just $\rho^\prime(x,D/\rho^\prime) =\inf_{y\in D} \rho^\prime(x,y)$, which is a formula because $D$ is a definable set.

We have shown that if we have a quotient of an imaginary in normal form we can re-express it as a definable subset of a quotient of a quotient of $H^\omega$. So then all we need to do is note that a quotient by some pseudo-metric $\rho^\prime$ of a quotient by some pseudo-metric $\rho$ is equivalent to a quotient by $\rho^\prime$ in the first place, since $\rho^\prime$ is still a definable pseudo-metric on the original sort. From this we get that $I$ can be expressed in normal form.\editcom{ Changed word `quotienting.' Split off last clause into sentence.} 

Let $I$ be an imaginary that is a (possibly $\omega$-)product of imaginaries in normal form. Let $\{I_k\}_{k<\ell}$ be the family of imaginaries and let $I_k$ be the set $D_k$ that is a definable subset of $H^\omega / \rho_k$, where $\rho_k$ is some definable pseudo-metric.  Our first claim is that $\prod_{k<\ell} D_k$ is a definable subset of $X = \prod_{k<\ell} H^\omega / \rho_k$. If $\ell$ is finite, then the metric on $X$ is given by $\mathrm{max}(\rho_0,\dots,\rho_{\ell - 1})$ and so the distance predicate of $\prod_{k<\ell} D_k$ is just the maximum of the distance predicates of the $D_k$. If $\ell$ is infinite, then the metric on $X$ is given by $\max_{k<\omega}{2^{-k} \rho_k}$, so again we can define the distance predicate of  $\prod_{k<\omega} D_k$ in $X$ by $\max_{k<\omega} 2^{-k} \rho_k(x_k, D_k)$, 
so again  $\prod_{k<\ell} D_k$  is a definable subset of $X$. All that's left is to argue that $X$ is equivalent to a quotient of $H^\omega$, but this is easy since the metric on $X$ is a formula on $((H^\omega)^\ell)^2$ and $(H^\omega)^\ell$ and $H^\omega$ clearly have a $\varnothing$-definable bijection between them. 

So by structural induction on imaginaries, we have that every $\varnothing$-definable imaginary has a $\varnothing$-definable bijection with an imaginary that is a $\varnothing$-definable subset of a $\varnothing$-definable quotient of $H^\omega$.

(b). Let $\rho(x,y;\overline{a})$ be an $A$-definable pseudo-metric on some $\varnothing$-definable imaginary sort $I$. Now consider the formula
\[
  \rho^\prime(x,\overline{z},y,\overline{w})=\imax{ d(\overline{z},\overline{w})} { \sup_{t}|\rho(t,x,\overline{z})-\rho(t,y,\overline{w})|}.
\]
This is clearly a pseudo-metric on $I\times H^\ell$ for some $\ell \leq \omega$ (since it is the supremum of a family of pseudo-metrics). The set $U \times \{\overline{a}\}$ is an $A$-definable subset of $U \times H^\ell$, so by the argument in part (a) its image in $(U \times H^\ell ) / \rho^\prime$ is a definable subset, which we'll call $D$. Furthermore, if $d(\overline{z},\overline{w}) > 0$, then $\rho^\prime (x,\overline{z},y,\overline{w})>0$ as well, so if some $x\overline{z} \in I \times H^\omega$ maps to $D$, then $\overline{z}= \overline{a}$. Finally,  the pseudo-metric $\rho^\prime$ restricted to the set $I \times {\overline{a}}$ is clearly identical to the pseudo-metric $\rho$ under the obvious bijection between $I$ and $I \times \{\overline{a}\}$. This implies that this bijection respects both of the quotients $I / \rho$ and $(I\times\{\overline{a}\}) / \rho^\prime$, but the second one is identically the definable set $D$, so $I / \rho$ has an $A$-definable bijection with an $A$-definable subset of the $\varnothing$-definable imaginary $(I\times H^\ell)/\rho^\prime$.

(c). We have that every $A$-definable imaginary has an $A$-definable bijection with an imaginary that is an $A$-definable subset $D$ of an $A$-definable quotient $H^\omega / \rho$. We also have that $H^\omega / \rho$ has an $A$-definable bijection with an $A$-definable subset $E$ of some $\varnothing$-definable imaginary (specifically a different quotient of $H^\omega$). Thinking of $D$ as an $A$-definable subset of $E$ we get that $D$ is an $A$-definable subset of some $\varnothing$-definable quotient of $H^\omega$.
\end{proof}

\begin{nota}
  Given an imaginary sort $I$ and a set of parameters $A$, we will write the type space of elements of $I$ over the parameters $A$ as $S_I(A)$.
\end{nota}

\section{Vaughtian Pairs} \label{sec:VP}

 There are two strengthenings of the no Vaughtian pairs condition that we will need to use later on. Here we develop familiar facts regarding these notions and, in particular, prove that \insep\ categorical theories satisfy these stronger conditions.

\begin{defn}
If $X\subseteq S_{1}(A)$ is a definable (resp.\ open or open-in-definable)
set containing a non-algebraic type, with $A$ countable, and if $\mathfrak{M}\succ\mathfrak{N}\supseteq A$
is a proper elementary pair such that $X(\mathfrak{M})=X(\mathfrak{N})$,
then we say that $(\mathfrak{M},\mathfrak{N})$ is a \emph{definable} (resp.\
\emph{open} or \emph{open-in-definable}) \emph{Vaughtian pair} (with regards to $X$).
A Vaughtian pair with no qualifier is a definable Vaughtian pair.
\end{defn}

Note that if a theory has no open-in-definable Vaughtian pairs then
it has no open Vaughtian pairs and no definable Vaughtian pairs, since open sets and definable sets are special cases of open-in-definable sets. In discrete logic these three notions are essentially the same;
in continuous logic, however, they are all distinct:\editcom{ Removed `of course.'} 

\begin{ex} \label{ex:VP-examples}
\leavevmode
\begin{enumerate}[label=(\roman*)]
\item If $\mathfrak{M}$ is a model of any discrete theory with a Vaughtian
pair, then the structure $\mathfrak{M}^{\omega}$ with the metric
$d(\alpha,\beta)=2^{-i}$ where $i$ is the smallest index such that
$\alpha(i)\neq\beta(i)$ and with the function $f(\alpha)(i)=\alpha(i+1)$,
has a definable Vaughtian pair but no open Vaughtian pairs. If $\mathrm{Th}(\mathfrak{M})$
is $\omega$-stable but is not \insep\ categorical, then this
is an example of an $\omega$-stable theory with no open Vaughtian
pairs that is not \insep\ categorical.

\item (Example~\ref{ex:main-bad-example}) If $\mathfrak{N}$ is a discrete metric space with a single unary
$[0,1]$-valued predicate whose values are dense in $[0,1]$, then
$\mathrm{Th}(\mathfrak{N})$ has no definable Vaughtian pairs (because
every definable set is either finite or cofinite in every model) but
does have an open Vaughtian pair. 

\item The structure $\mathfrak{N}^{\omega}$---where $\mathfrak{N}$
is the structure in part \emph{(ii)} with the truncation map
$f(\alpha)(i)=\alpha(i+1)$ and with a $[0,1]$-valued unary predicate that is the predicate from part \emph{(ii)} evaluated on $\alpha(0)$---is an example of a structure whose theory has no definable Vaughtian pairs and no open Vaughtian pairs but which does have an open-in-definable Vaughtian pair.
\end{enumerate}
\end{ex}

\begin{proof} [Verification]
  \emph{(i)} is clear. \emph{(ii)} is verified in Example \ref{ex:main-bad-example}. For \emph{(iii)}, $\mathrm{Th}(\mathfrak{N}^{\omega})$
has no open Vaughtian pairs, by the same argument as in \emph{(i)}. This is an imaginary sort of $\mathfrak{N}$,
so in particular if $D$ is a definable subset of $\mathfrak{N}^{\omega}$
and for some $\sigma\in\mathfrak{N}^{<\omega}$, we look at the ball
$B_{\sigma}=\{\alpha\in\mathfrak{N}^{<\omega}:\sigma\prec\alpha\}$
and the function $f:B_{\sigma}\rightarrow\mathfrak{N}$ which maps
$f(\alpha)=\alpha(|\sigma|)$ (i.e.\ the first element of $\alpha$ not
in $\sigma$), then $D\cap B_{\sigma}$ is definable (since $B_{\sigma}$
is logically clopen) and the image $f(D\cap B_{\sigma})$ is a definable
subset of $\mathfrak{N}$. This implies that any definable subset
of $\mathfrak{N}^{\omega}$ is actually definable in the reduct where
the $[0,1]$-valued predicate is removed. This reduct is $\aleph_1$-categorical,
so it has no definable Vaughtian pairs. Therefore $\mathrm{Th}(\mathfrak{N}^{\omega})$
has no definable Vaughtian pairs. However if we fix some element $b\in\mathfrak{N}$
and note that if $D=\{\alpha:(\forall i>0)\alpha(i)=b\}$, then $D$
is a definable set, since $d(x,D)=\frac{1}{2}d(f(x),b^{\omega})$.
We then get an open-in-definable Vaughtian pair by considering the
same open set as in part \emph{(ii)}, even though this theory does not have any definable Vaughtian pairs or any open Vaughtian pairs.\editcom{ Moved `then.'} 
\end{proof}


There is a strictly stable discrete theory which has no Vaughtian pairs
but which does have an imaginary Vaughtian pair, although the same cannot happen in a discrete superstable theory \cite{BOUSCAREN1989129}.

\begin{prop}
\label{prop:no-VP}Suppose that $T$ is a countable theory, $U\subseteq D\subseteq S_{n}(B)$
is an open-in-definable set, and $B\subseteq\mathfrak{N}\prec\mathfrak{M}$
is a Vaughtian pair over $U$, then there exists a model $\mathfrak{A}$
such that for some open $V\subseteq D$ containing a non-algebraic
type, $\#^{\mathrm{dc}}\mathfrak{A}=\aleph_1$ but $\#^{\mathrm{dc}}V(\mathfrak{A})=\aleph_0$.
\end{prop}

\begin{proof}
Find $p\in U$ that is non-algebraic and find a formula $\varphi$ such that $p\in\cset{\varphi<\frac{1}{2}}\subseteq U$. Note that $\varphi$ and $D$ are definable over some
countable set $B_{0}\subseteq B$. Now it is clear that $(\mathfrak{M},\mathfrak{N})$
is still an open-in-definable Vaughtian pair with regards to $V=\cset{\varphi<\frac{1}{2}}$.\editcom{ Removed restricted formula.} 

Let $(\mathfrak{D}_{1},\mathfrak{D}_{0})$ be a countable elementary
sub-pre-structure\footnote{Recall that a pre-structure is a metric structure in which the metric is allowed to be a possibly incomplete pseudo-metric \cite{MTFMS}.\editcom{Removed `as in.'}} of $(\mathfrak{M},\mathfrak{N})$ such that $\mathfrak{D}_{0}\supseteq B_{0}$.
By the same argument as in discrete logic, we can find countable pre-structures
$(\mathfrak{A}_{1},\mathfrak{A}_{0})\succeq(\mathfrak{D}_{1},\mathfrak{D}_{0})$
such that $\mathfrak{A}_{1}$ and $\mathfrak{A}_{0}$ realize
the same types, and are both (exactly) $\aleph_0$-homogeneous (as pre-structures, i.e.\ only over parameters that are actually in them),\editcom{Changed `in the pre-structures' to `in them.'} so that
in particular $\mathfrak{A}_{1}\cong\mathfrak{A}_{0}$. We can also
ensure that $\mathfrak{A}_{1}$ and $\mathfrak{A}_{0}$ both realize
dense subsets of $D$, which implies that $\mathfrak{A}_{0}$ realizes
a metrically dense subset of $U(\mathfrak{A}_{0})=U(\mathfrak{A}_{1})$,
since $U$ is relatively open in $D$.

We can run the same elementary chain argument as in discrete logic,
the argument that $\mathfrak{A}_{i}$ for limit $i$ is still isomorphic
(as a countable pre-structure) to $\mathfrak{A}_{0}$ still works,
but we need to argue that at each stage $i$ the set $U(\mathfrak{A}_{0})$
is still metrically dense in $U(\mathfrak{A}_{i})$. Obviously, by
construction, $U(\mathfrak{A}_{0})$ is metrically dense in $U(\mathfrak{A}_{1})$.
Suppose that for some $i<\omega_{1}$, $U(\mathfrak{A}_{0})$ is metrically dense in $U(\mathfrak{A}_{i})$ and
 $(\mathfrak{A}_{i+1},\mathfrak{A}_{i})\cong(\mathfrak{A}_{1},\mathfrak{A}_{0})$.
We get from this that $U(\mathfrak{A}_{i})$ is metrically dense in $U(\mathfrak{A}_{i+1})$,
implying that $U(\mathfrak{A}_{0})$ is metrically dense in $U(\mathfrak{A}_{i+1})$
as well. Suppose that for some limit $i<\omega_{1}$ we have that
for all $j<i$, $U(\mathfrak{A}_{0})$ is metrically dense in $U(\mathfrak{A}_{j})$.
Then any $a\in U(\mathfrak{A}_{i})$ is actually in $U(\mathfrak{A}_{j})$
for some $j<i$, so it is still in the metric closure of $U(\mathfrak{A}_{0})$,
therefore $U(\mathfrak{A}_{0})$ is still metrically dense in $U(\mathfrak{A}_{i})$.

After taking the union and metric completion to get $\mathfrak{A}=\overline{\bigcup_{i<\omega_1}\mathfrak{A}_i}$,
we need to argue that $U(\mathfrak{A}_{0})$ is still metrically dense
in $U(\mathfrak{A})$. Suppose that $a\in U(\mathfrak{A})$. Since
it is a relatively open set in a definable set, it must be the metric
limit of some sequence $a_{k}\in\bigcup_{i<\omega_{1}}U(\mathfrak{A}_{i})$.
Therefore since $U(\mathfrak{A}_{0})$ is dense in every $U(\frk{A}_i)$,\editcom{Changed `all of those.'} it
must be in the metric closure of $U(\mathfrak{A}_{0})$ as well, so
in particular $\#^{\mathrm{dc}}U(\mathfrak{A})=\aleph_0$.

Finally we need to show that $\#^{\mathrm{dc}}\mathfrak{A}=\aleph_1$.
Let $\varepsilon>0$ be such that for some $a\in\mathfrak{A}_{1}$,
$d(a,\mathfrak{A}_{0})>\varepsilon$. By isomorphism, for each $i<\omega_1$,
we can find an $a_{i}\in\mathfrak{A}_{i+1}\smallsetminus\mathfrak{A}_{i}^{<\varepsilon}$,
this is a $({>}\varepsilon)$-separated set of size $\aleph_1$, so
$\#^{\mathrm{dc}}\mathfrak{A}\geq\aleph_1$. On the other hand $\mathfrak{A}$
clearly has a dense subset of size $\aleph_1$, so $\#^{\mathrm{dc}}\mathfrak{A}=\aleph_1$,
as required.
\end{proof}
\begin{cor}
If $T$ is a countable $\aleph_1$-categorical theory, then $T$
has no open-in-definable Vaughtian pairs. 
\end{cor}

\begin{proof}
Suppose that $T$ is a countable theory with an open-in-definable
Vaughtian pair, then by Proposition \ref{prop:no-VP}, $T$ has a model
$\mathfrak{A}$ with $\#^{\mathrm{dc}}\mathfrak{A}=\aleph_1$, but
for some non-algebraic type over a countable set $p$, $\#^{\mathrm{dc}}p(\mathfrak{A})\leq\omega$.
This implies that $\mathfrak{A}$ is not $\aleph_1$-saturated,
implying that $T$ is not $\aleph_1$-categorical.
\end{proof}

\begin{cor}
If $T$ is a countable $\aleph_1$-categorical theory, then $T$ has no open Vaughtian pairs in any imaginaries. In particular it has no imaginary Vaughtian pairs.
\end{cor}
\begin{proof}
An imaginary expansion of a $\kappa$-categorical theory is still $\kappa$-categorical.
\end{proof}

\section{Dictionaric Type Spaces} \label{sec:dict}

It is well known that definable sets are poorly behaved in continuous
logic. As it will turn out, $\omega$-stable theories (and, more generally, totally transcendental theories) have better
behavior with regards to definable sets than arbitrary continuous
first-order theories, even relative to strictly superstable theories. The specific property they have 
seems to be important enough to have its own adjectival name. A similar, slightly weaker concept was  
discussed in the introduction of \cite{benyaacov2010} but was not developed. 

\subsection{Definition and Characterization}
\label{sec:defn-and-char}

\begin{defn}
\leavevmode
\begin{enumerate}[label=(\roman*)]
\item Let $X$ be a type space or a definable set. We say that $X$
is \emph{dictionaric} if for every $p\in X$ and closed $F\subseteq X$ with
$p\notin F$, there is a definable set $D\subseteq X$ such that $p\in\tint _{X}D$
and $D\cap F=\varnothing$ (i.e.\ $X$ has a basis of definable neighborhoods).

\item A theory $T$ is \emph{dictionaric} if for every $n<\omega$ and parameter
  set $A$, $S_{n}(A)$ is dictionaric.
  
\item A theory $T$ is \emph{dictionaric over models} if for every $n<\omega$ and model $\frk{M}\models T$, $S_n(\frk{M})$ is dictionaric.
\end{enumerate}
\end{defn}

Note that in these definitions when we talk about a definable set being dictionaric,
we are talking about it relative to some specific type space. It may
be the case that a definable set $D$ is dictionaric relative to $S_{n}(A)$,
but that there also is some larger parameter set $B\supset A$ such that $D$
fails to be dictionaric relative to $S_{n}(B)$. 
It is easy to construct such an example in which $D$ is a singleton in $S_N(A)$. This perspective is justified by thinking of $D$ as the type space of the imaginary corresponding to $D$ (over the parameters in question).\editcom{Removed `But of course....'} 

Every discrete theory is clearly dictionaric. With some work it is possible to show that randomizations of discrete theories (as defined in \cite{BYKRandom})
are dictionaric as well.\editcom{Changed `Obviously.'}

The following is a simple example of a theory which fails to be dictionaric in a drastic way.\editcom{Changed wording.}
\begin{ex}
\label{ex:main-bad-example}
A strictly superstable theory $T$ such that $S_1(T)$ is not a singleton but such that the only definable subsets of $S_1(T)$ are $\varnothing$ and $\cset{d(x,x)\brackconv}$ (i.e.\ the whole home sort). (Note that this is the same as Example \ref{ex:VP-examples} (ii).)
\end{ex}

\begin{proof}[Verification]
Let $\mathfrak{M}$ be a structure whose universe is $[0,1]$ with a discrete metric and which has a single unary predicate $P$ such that $P(x)=x$ (note that the first $x$ is an element of the structure but the second $x$ is a truth value). To see that $T=\mathrm{Th}(\mathfrak{M})$ is weakly minimal and, in particular, superstable, note that it can be realized as a continuous reduct of a discrete weakly minimal theory: the theory with an infinite sequence of unary predicates $Q_n$ which says that every finite quantifier free type occurs (i.e.\ the standard example of a strictly superstable weakly minimal theory with trivial pregeometry). We can define $P=\sum_{n<\omega}2^{-n-1}Q_n$ and take the reduct where we forget the $Q_n$'s and this will be the theory $T$.\editcom{Removed `then.'} The type space $S_1(\varnothing)$ is homeomorphic to $[0,1]$ with the standard topology. For any $0<\varepsilon<1$ and any closed set $F\subseteq S_1(\varnothing)$, $F^{<\varepsilon}=F$, so the only definable sets are clopen sets, which are $\varnothing$ and $\cset{d(x,x)\brackconv}=S_1(\varnothing)$. $T$ cannot be $\omega$-stable because $\#^\mathrm{dc} S_1(\varnothing) = 2^{\aleph_0}$.
\end{proof}

We will now present a characterization of dictionaric type spaces and definable sets.\editcom{New sentence.}

\begin{prop}
\label{prop:dict-equivalents}Let $X$ be a type space or a definable
set. The following are equivalent:
\begin{enumerate}[label=(\roman*)]
\item $X$ is dictionaric.

\item For every $\varepsilon>0$, $X$ has a basis of open sets $U$
satisfying $\overline{U}\subseteq U^{<\varepsilon}$.

\item Definable sets separate disjoint closed subsets of $X$.

\item For every closed $F,G\subseteq X$ with $F\cap G=\varnothing$,
there is a definable set $D$ such that either $F\subseteq D$ and
$D\cap G=\varnothing$ or $G\subseteq D$ and $D\cap F=\varnothing$.

\item $X$ has a network of
definable sets (i.e.\ for every $p\in U\subseteq X$, with $U$ open, there is a definable
set $D$ such that $p\in D\subseteq U$).

\item For every $p\in U\subseteq X$, with $U$ open, and every $\varepsilon>0$,
there is an open set $V\subseteq U$ and a closed set $F\subseteq X$
such that $p\in V$ and $d_{H}(V,F)<\varepsilon$ (where $d_{H}$
is the Hausdorff metric on sets).
\end{enumerate}
\end{prop}

\begin{proof}

We will prove \emph{$\text{(i)}\Rightarrow \text{(v)}\Rightarrow \text{(ii)} \Rightarrow \text{(i)}$}, \emph{$\text{(i)}\Rightarrow \text{(iv)}\Rightarrow \text{(iii)} \Rightarrow \text{(i)}$}, and \emph{$\text{(ii)}\Leftrightarrow \text{(vi)}$}.\editcom{Added sentence.}
  
\emph{$\text{(i)}\Rightarrow\text{(v)}$.} This is immediate.

\emph{$\text{(v)}\Rightarrow\text{(ii)}$.} For any definable set $D$ and $\varepsilon>0$,
the set $U=D^{<\varepsilon/2}$ satisfies $\overline{U}\subseteq D^{\leq\varepsilon/2}\subseteq D^{<\varepsilon}$.
So since given $x\in D\subseteq U$ we can always find $\varepsilon>0$
small enough that $D^{\leq\varepsilon/2}\subseteq U$ (by compactness),
we have that \emph{(ii)} holds.

\emph{$\text{(ii)}\Rightarrow\text{(i)}$.} Let $x\in V\subseteq X$ be a point and an
open neighborhood. Find $O$ open such that $x\in O$ and $\overline{O}\subseteq V$.
Find $\varepsilon_{0}>0$ small enough that $\overline{O}^{\leq2\varepsilon_{0}}\subseteq V$.
Using the assumption of \emph{(ii)}, find an open neighborhood $U_{0}$
of $x$ such that $\overline{U}_{0}\subseteq O$ and such that $\overline{U}_{0}\subseteq U_{0}^{<\varepsilon_{0}}$.
For each $n<\omega$, given $U_{n}$, find $V_{n}$ open such that
$\overline{U}_{n}\subseteq V_{n}$ and $\overline{V}_{n}\subseteq U_{n}^{<\varepsilon_{n}}$.
Then find $\varepsilon_{n+1}>0$ small enough that $\overline{V}_{n}^{\leq\varepsilon_{n+1}}\subseteq U_{n}^{<\varepsilon_{n}}$
and $\varepsilon_{n+1}<\imin{\varepsilon_{n}}{2^{-n}}$. Using the assumption
of \emph{(ii)}, for each $y\in\overline{U}_{n}$, find an open neighborhood
$O_{n,y}$ such that $\overline{O}_{n,y}\subseteq V_{n}$ and such
that $\overline{O}_{n,y}\subseteq O_{n,y}^{<\varepsilon_{n+1}}$.
By compactness this cover has a finite subcover. Let $U_{n+1}$ be
its union. Note that we have that $\overline{U}_{n}\subseteq\overline{U}_{n+1}\subseteq U_{n+1}^{<\varepsilon_{n+1}}\subseteq\overline{U}_{n+1}^{\leq\varepsilon_{n+1}}\subseteq U_{n}^{<\varepsilon_{n}}\subseteq\overline{O}^{\leq2\varepsilon_{0}}\subseteq V.$

Now finally let $D=\bigcap_{n<\omega}\overline{U}_{n}^{\leq\varepsilon_{n}}=\bigcap_{n<\omega}U_{n}^{<\varepsilon_{n}}$.
Clearly $U_{0}\subseteq D\subseteq V$, so $D$ is a neighborhood
of $x$ that is a sub-neighborhood of $V$. Furthermore $D$ is clearly
closed. To see that $D$ is definable, note that for any $\delta>0$
there is an $n$ such that $\varepsilon_{n}<\delta$, so we have that
$D^{<\delta}\supseteq U_{n}^{<\varepsilon_{n}}\supseteq D$. Therefore
$D\subseteq\tint D^{<\delta}$ for every $\delta>0$ and $D$
is definable. Therefore $X$ is dictionaric.

\emph{$\text{(i)}\Rightarrow\text{(iv)}$.} Let $F$ and $G$ be disjoint closed sets. For each
point $x\in F$, find a definable neighborhood $D_{x}$ disjoint from
$G$. By compactness finitely many of these cover $F$. Since the
union of finitely many definable sets is definable we have that their
union $D$ is a neighborhood of $F$ disjoint from $G$.

\emph{$\text{(iv)}\Rightarrow\text{(iii)}$.} Given $F$ and $G$, disjoint closed sets, let $U\supseteq F$
and $V\supseteq G$ be disjoint open neighborhoods and let $L=X\smallsetminus(U\cup V)$
be a closed separator between $F$ and $G$. Now let $D$ be the guaranteed
definable set separating $F\cup G$ and $L$. If $D\supseteq F\cup G$,
then $D\cap U$ and $D\cap V$ are the required definable sets (because
if the intersection of a definable set and an open set is closed, then
it is definable). If $D\supseteq L$, then consider the disjoint closed
sets $F$ and $D\cup V$. Find $U^{\prime}\supseteq F$ and $V^{\prime}\supseteq D\cup V$
disjoint open neighborhoods, and let $L^{\prime}=X\smallsetminus(U^{\prime}\cup V^{\prime})$
be a closed separator between $F$ and $D\cup V$. Let $E$ be the
guaranteed definable set separating $F\cup D\cup V$ and $L^{\prime}$.
If $E\supseteq F\cup D\cup V$, then $E\cap U^{\prime}$ is the required
definable set for $F$. If $E\supseteq L^{\prime}$, then $E\cup U^{\prime}$
is the required definable set for $F$. So in any case we have that
there are definable, disjoint $E$ and $D$ such that $F\subseteq E$
and $G\subseteq D$.

\emph{$\text{(iii)}\Rightarrow\text{(i)}$,} If $x\in U\subseteq X$ with $U$ open, we
can find an open set $V$ such that $x\in V$ and $\overline{V}\subseteq U$.
A definable set separating $\overline{V}$ and $X\smallsetminus U$ is
the required definable neighborhood of $x$.

\emph{$\text{(ii)}\Rightarrow\text{(vi)}$.} This is immediate.

\emph{$\text{(vi)}\Rightarrow\text{(ii)}$.} Given $x\in V\subseteq X$, with $V$ open,
find open $V^{\prime}$ such that $x\in V^{\prime}$ and $\overline{V}^{\prime}\subseteq V$.
Fix $\varepsilon>0$, find $U\subseteq V^{\prime}$ with $x\in U$
and such that for some closed set $F$, $d_{H}(U,F)<\frac{\varepsilon}{5}$.
By the definition of the Hausdorff metric this implies that $U\subseteq\overline{U}\subseteq F^{\leq\varepsilon/3}\subseteq U^{\leq2\varepsilon/3}\subseteq U^{<\varepsilon}$,
so we get the required basis of open sets.
\end{proof}

\subsection{Properties of Dictionaric Type Spaces}
\label{sec:prop-of-dict}

In this section we will show that dictionaric type spaces enjoy a few nice properties, but first we will need the following lemma.\editcom{Extended sentence.}
\begin{lem}
\label{lem:open-in-definable-locatable}If $D\subseteq X$ is a definable
set and $U\subseteq D$ is open in $D$, then $U^{<\varepsilon}$ is open (in $X$) for every $\varepsilon>0$.
\end{lem}

\begin{proof}
For each $x\in U$, find $\varepsilon>0$ small enough that $x\notin(D\smallsetminus U)^{\leq2\varepsilon}$.
Now consider $D^{<\varepsilon}=(D\smallsetminus U)^{<\varepsilon}\cup U^{<\varepsilon}$.
Note that $(D\smallsetminus U)^{<\varepsilon}\cup U^{<\varepsilon}\smallsetminus(D\smallsetminus U)^{\leq2\varepsilon}=U^{<\varepsilon}\smallsetminus(D\smallsetminus U)^{\leq2\varepsilon}$
is an open set containing $x$ and contained in $U^{<\varepsilon}$.
Since this remains true for any smaller $\varepsilon>0$ we have that
$U\subseteq\tint U^{<\delta}$ for every $\delta>0$.

To see that $U^{<\varepsilon}$ is actually open, first note that $U^{<\varepsilon}\supseteq\bigcup_{0<\delta<\varepsilon}(\tint U^{<\delta})^{<\varepsilon-\delta}$
by the triangle inequality and $U^{<\varepsilon}\subseteq\bigcup_{0<\delta<\varepsilon}(\tint U^{<\delta})^{<\varepsilon-\delta}$
since 
\[\bigcup_{0<\delta<\varepsilon}(\tint U^{<\delta})^{<\varepsilon-\delta}\supseteq\bigcup_{0<\delta<\varepsilon}U^{<\varepsilon-\delta}=U^{<\varepsilon}.\]
So we have that $U^{<\varepsilon}=\bigcup_{0<\delta<\varepsilon}(\tint U^{<\delta})^{<\varepsilon-\delta}$,
implying that $U^{<\varepsilon}$ is an open set.
\end{proof}
\begin{prop}[Hereditariness to Definable Subsets] \label{prop:heredit}
If $X$ is a dictionaric type space or definable set, and $D$ is
a definable subset of $X$, then $D$ is dictionaric.
\end{prop}

\begin{proof}
Let $D$ be a definable set in $X$. Let $U$ satisfying $x\in U\subseteq D$
be an open neighborhood of $x$ in $D$. Find $V$ an open set in
$D$ such that $x\in V$ and $\overline{V}\subseteq U$. Then let
$\varepsilon>0$ be small enough that $\overline{V}^{<3\varepsilon}\cap(D\smallsetminus U)=\varnothing$.
By the preceding lemma we have that $V^{<\varepsilon}$ is open as
a subset of $X$, so by dictionaricness we can find a definable set
$E$ such that $x\in\tint E$ and $E\subseteq V^{<\varepsilon}$.
Now consider the set $F=E^{\leq\varepsilon}\cap D$. By the triangle
inequality this is disjoint from $D\smallsetminus U$ and so is a subset
of $U$. Furthermore it is a neighborhood of $x$ in $D$. Also, since
$E\subseteq V^{<\varepsilon}$ and $V\subseteq D$, we have that $E\subseteq F^{<2\varepsilon}$,
so $F\subseteq\tint _{D}F^{<4\varepsilon}$. Let $O$ be an open
in $D$ set such that $F\subseteq O$ , $\overline{O}\subseteq U$,
and $\overline{O}\subseteq\tint _{D}F^{<4\varepsilon}$. Then
we have that $O$ is an open neighborhood of $x$, smaller than $U$
such that $\overline{O}\subseteq O^{<4\varepsilon}$. Since we can
do this for arbitrary $x$, $U$, and $\varepsilon$, we have by Proposition
\ref{prop:dict-equivalents} that $X$ is dictionaric. 
\end{proof}
The analog of the previous result fails for arbitrary closed $F\subseteq X$; there is a dictionaric type space with a closed subset that fails
to be `relatively dictionaric' (compare this to the fact that any subspace of a totally disconnected compact Hausdorff space is totally disconnected\footnote{Or has fewer than $2$ points, if your definition of totally disconnected excludes spaces with $1$ or $0$ points.}).\editcom{removed `i.e.'}

Compare the following Proposition~\ref{prop:ext} to the fact that if $X$ is a totally disconnected compact Hausdorff space, $F \subseteq X$ is a closed subset, and $Q \subseteq F$ is relatively clopen, then there is a clopen set $D \subseteq X$ such that $D\cap F = Q$.\editcom{Added sentence.}


\begin{prop}[Extension] \label{prop:ext} If $X$ is a dictionaric type space or
definable set and we have $Q\subseteq F\subseteq X$ with $Q$ and
$F$ closed and $Q$ relatively definable in $F$, then there is a definable
set $D\subseteq X$ such that $D\cap F=Q$.
\end{prop}

\begin{proof} Use the following Lemma~\ref{lem:strong-ext} with $F_{i}=F$ for all $i<\omega$ to get some
definable set $D$ and formula $\varphi$ such that $D\cap F\cap\cset{\varphi\le1}=Q$, but $\cset{\varphi\leq1}=X$,
so $D\cap F=Q$, as required.
\end{proof}

Although Proposition \ref{prop:ext} is the more attractive statement, we will occasionally need this technical strengthening.

\begin{lem}  \label{lem:strong-ext}
Let
$\{F_{i}\}_{i<\omega}$ be a family of closed sets in $X$, a type space or definable set (not necessarily dictionaric), and let $Q$ be
a closed set such that $Q\subseteq F_{i}$ and $Q$ is relatively definable in $F_{i}$
for each $i<\omega$.
\begin{enumerate}[label=(\roman*)]
\item There is a closed set $C\subseteq X$ and a formula $\varphi:X\rightarrow[0,1]$ with $Q\subseteq\cset{\varphi\brackconv}$,  such that
\begin{itemize}
\item$Q\subseteq C$ and for each $\varepsilon>0$,
$Q\subseteq\tint _{X}C^{<\varepsilon}$, and 
\item for each $i<\omega$, $C\cap F_{i}\cap\cset{\varphi\leq2^{-i}}=Q$.
\end{itemize}
\item\emph{(Strong Extension)}  If $X$ is dictionaric, then $C$ can be taken to be definable.
\end{enumerate}
\end{lem}

\begin{proof}  \emph{(i).} Without loss of generality we may assume that $F_{i}\subseteq F_{i+1}$, since we have that 
  if $Q$ is relatively definable in two closed
  sets $F$ and $G$, then $Q$ is relatively definable in $F\cup G$ as well. So we can replace $F_{i}$ by $\bigcup_{j\leq i}F_{j}$ if necessary. 
  To see that such a $Q$ is relatively definable in $F\cup G$, fix $\e > 0$ and find $U \subseteq F$ and $V \subseteq G$, each relatively open, such that $Q \subseteq U$ and $Q \subseteq V$ and such that if $a \in U$ or if $a \in V$, then $d(a,Q) < \e$. $U\cup V$ is a relatively open neighborhood of $Q$ in $F\cup G$ with the same property, so $Q$ is relatively definable in $F\cup G$.\editcom{Redid this paragraph.}



Assume without loss of generality that the metric diameter of $X$ is at most $1$.\editcom{Added this sentence.} For each $i<\omega$, let $f_{i}:X\rightarrow[0,1]$ be a continuous
function witnessing that $Q$ is relatively definable in $F_{i}$,
i.e.\ $Q\subseteq\cset{f_{i}\brackconv}$ and for every $x\in F_{i}$, $\dinf(x,Q)\leq f_{i}(x)$.
Let $f(x)=\sum_{i<\omega}2^{-i}f_{i}$. Note that $\cset{f\brackconv}\cap F_{i}=Q$
for every $i<\omega$.

For each $k<\omega$, let $G_{k}=F_{k}\cap\cset{4^{-k-1}\leq f\leq4^{-k}}$,
let $U_{k}=\cset{f<4^{-k}}$, and let $V_{k}=\cset{4^{-k-2}<f<4^{-k+1}}\cap U_{k+3}^{<2^{-k+1}}$;
note that this $V_k$ is an open set. Also note that $G_{k}\subseteq\cset{4^{-k-2}<f<4^{-k+1}}$
and that\editcom{Added `that.'} if $p\in G_{k}$, then $f(p)\leq4^{-k}$. This implies in particular that $2^{-k}f_{k}(p)\leq4^{-k}$, and
so also $f_{k}(p)\leq2^{-k}$ and $d(p,Q)\leq2^{-k}$. This in turn implies that
$p\in Q^{\leq2^{-k}}\subseteq U_{k+3}^{<2^{-k+1}}$, and thus $G_{k}\subseteq V_{k}$.

Let $A=X\smallsetminus\bigcup_{k<\omega}V_{k}$. $A$ is clearly closed. 
Note that $\cset{f\brackconv}\cap V_{k}=\varnothing$ for each $k<\omega$, so we
have that $Q\subseteq\cset{f\brackconv}\subseteq A$. We want to show that for any
$\varepsilon>0$, $Q\subseteq\tint _{X}A^{<\varepsilon}$. 

Assume that $q\notin A$. It must be in $V_{k}$ for some $k<\omega$.
Let $q_{0}=q$ and $k(0)=k$. Assume we are given $q_{\ell}$ and $k(\ell)$
such that $q_{\ell}\in V_{k(\ell)}$. By construction there exists
$q_{\ell+1}\in U_{k(\ell)+3}$ such that $d(q_{\ell},q_{\ell+1})<2^{-k(\ell)+1}$.
Stop if $q_{\ell+1}\in A$,\editcom{Moved `Stop.'} otherwise $q_{\ell+1}\in V_{k(\ell+1)}$
for some $k(\ell+1)$ strictly larger than $k(\ell)$.

After the construction, in either case, the total distance traversed
along the sequence $q_{0},q_{1},\dots$ is $\leq\sum_{k\leq m<\omega}2^{-m+1}=2^{-k+2}$.
If we stopped, then the final point is in $A$, so $d(q,A)<2^{-k+3}$.
If the sequence never stopped, then it is a Cauchy sequence whose
limit, $q_{\omega}$, by continuity has $f(q_{\omega})=0$, so $q_{\omega}\in A$
as well. So we have that $\cset{f\brackconv}\subseteq\cset{f<2^{-k-1}}\subseteq A^{<2^{-k+3}}$,
hence $Q\subseteq\cset{f\brackconv}\subseteq\tint _{X}A^{<\varepsilon}$ for every
$\varepsilon>0$. Hence\editcom{Removed `In other words.'} we have verified the first bullet point of the lemma for the set $A$.

Let $\varphi(x)=\imin*{4\sqrt{f(x)}}{ 1}$. Clearly this is a
continuous $[0,1]$-valued function on $X$.\editcom{Added paragraph break.}

We need to adjust $A$ in order to satisfy the second bullet point in the $i=0$ case; specifically, we need to ensure that $C \subseteq \cset{\varphi < 2^{-0}}=\cset{\varphi < 1}$. It is not hard to show that if $O \supseteq Q$ is any open neighborhood such that $\overline{O} \subseteq \cset{\varphi<1}$, then $B = A\cap \overline{O}$ still satisfies the property stated in the first bullet point of the lemma and clearly satisfies $B \subseteq \cset{\varphi < 1}$.\editcom{Removed some words and comma splice.}


\editcom{Removed `Now.'} To verify the second bullet point, consider $B\cap F_{i}\cap\cset{\varphi\leq2^{-i}}$.
If $i=0$, then since $B\subseteq\cset{\varphi<2^{-0}}$ we have that for any $x \in B$, $\varphi(x)<1$,
and so $4\sqrt{f(x)}<1$, $\sqrt{f(x)}<\frac{1}{4}$, and $f(x)<4^{-2}$. 
If $i>0$, then $2^{-i}<1$, so $4\sqrt{f(x)}\leq2^{-i}$ and we have
$\sqrt{f(x)}\leq2^{-i-2}$ and $f(x)\leq4^{-i-2}$. So for any $i$, $\cset{f<4^{-i-1}}\cap B\cap F_{i}\subseteq\cset{f\brackconv}$.\footnote{To see this, note that if $x\in \left( \cset{f<4^{-i-1}}\cap A\cap F_{i} \right) \smallsetminus \cset{f\brackconv}$, then $x \in \cset{0 < f < 4^{-i-1}}$ and for all $k \geq i$, $x \in F_k \supseteq F_i$. Therefore $x$ must be in $G_k$ for some $k > i$, but the $G_k$ are all disjoint from $A$ by construction and therefore also disjoint from $B$. Thus we have a contradiction.\editcom{Moved to footnote.}} This implies
that $\cset{\varphi\leq2^{-i-1}}\cap B\cap F_{i}\subseteq\cset{f\brackconv}$, but $F_{i}\cap\cset{f\brackconv}=Q$,
so we have $B\cap F_{i}\cap\cset{\varphi\leq2^{-i-1}}=Q$, as required. So  part \emph{(i)} is satisfied by setting $C = B$.\editcom{Changed wording of last sentence.}

\emph{(ii).} Assume that $X$ is dictionaric.  Continuing from just after the definition of $\varphi(x)$  in the proof of part \emph{(i)}, we need to cover $A\smallsetminus\cset{f\brackconv}$ by definable sets so that the
overall union will be closed (this will be enough to imply that the union is definable) without spoiling the property stated in the second bullet point of the lemma.

For each $k<\omega$, let $H_{K}=A\cap\cset{4^{-k-1}\leq f\leq4^{-k}}$.
Note that by construction $H_{k}\cap F_{k}=\varnothing$. Let $W_{k}=\cset{4^{-k-2}<f<4^{-k+1}}\smallsetminus F_{k}$,
which is an open neighborhood of $H_{k}$. Let $D_{k}$ be a definable
set such that $H_{k}\subseteq D_{k}\subseteq W_{k}$. Finally let
$E=A\cup\bigcup_{k<\omega}D_{k}$.

First to see that $E$ is closed, note that any convergent net $\{q_{i}\}_{i\in I}$
in $E$ either eventually stays within some $W_{k}$ or has $\lim_{i\in I}f(q_{i})=0$.
In the first case, $E\cap(W_{k-1}\cup W_{k}\cup W_{k+1})$ is relatively
closed in $W_{k-1}\cup W_{k}\cup W_{k+1}$, since in that set it is
a finite union of closed sets, so the net converges to a point in $E$.
In the second case, the net must be converging to a point in $\cset{f=0}\subseteq A\subseteq E$
by continuity.

To see that $E$ is definable, note that for any $\varepsilon>0$,
\begin{align*}
  \tint _{X}E^{<\varepsilon}&=\tint _{X}\cset*{A\cup\bigcup_{k<\omega}D_{k}}^{<\varepsilon}\\
                            &\supseteq\tint _{X}A^{<\varepsilon}\cup\bigcup_{k<\omega}\tint _{X}D_{k}^{<\varepsilon}\\
                            &\supseteq\cset{f\brackconv}\cup\bigcup_{k<\omega}D_{k}=E.
\end{align*}
Now we want to argue that for each $k<\omega$, $Q$ is relatively clopen in $E\cap F_{k}$.
Note that by construction $\cset{4^{-k-2}<f<4^{-k+1}}\cap E=\cset{4^{-k-2}<f<4^{-k+1}}\cap\bigcup_{k-2\leq\ell\leq k+2}D_{\ell}$,
so we have that $\cset{4^{-k-2}<f<4^{-k+1}}\cap E\cap F_{k-2}=\varnothing$,
since each $D_{\ell}$, for $k-2\leq\ell$, is disjoint from $F_{k-2}$
(because $F_{m}\subseteq F_{m+1}$ for all $m<\omega$).\editcom{Removed `are all.'} This implies
that $\cset{0<f<4^{-k+1}}\cap E\cap F_{k-2}=\varnothing$ as well. Therefore
$\cset{f<4^{-k+1}}\cap E\cap F_{k-2}\subseteq\cset{f\brackconv}$, but $\cset{f\brackconv}\cap F_{k-2}=Q$,
so we have that $Q$ is relatively clopen in $E\cap F_{k-2}$, for any $k<\omega$,
as required.

Finally since $E$ is a definable set and $X$ is dictionaric, we
have that $E$ is dictionaric as well. Let $D\subseteq E$ be a definable
set such that $Q\subseteq D\subseteq\cset{\varphi<1}$.\footnote{This is only necessary
to handle the $i=0$ case.\editcom{Moved to footnote.}} 
 Consider $D\cap F_{i}\cap\cset{\varphi\leq2^{-i}}$.
If $i=0$, then since $D\subseteq\cset{\varphi<2^{-0}}$ we have that for any $x\in D$, $\varphi(x)<1$,
so $4\sqrt{f(x)}<1$, $\sqrt{f(x)}<\frac{1}{4}$, and $f(x)<4^{-2}$. 
If $i>0$, then $2^{-i}<1$, so $4\sqrt{f(x)}\leq2^{-i}$ and we have
$\sqrt{f(x)}\leq2^{-i-2}$ and $f(x)\leq4^{-i-2}$. We have already
established that $\cset{f<4^{-i-1}}\cap D\cap F_{i}\subseteq\cset{f\brackconv}$, so we
also have $\cset{\varphi\leq2^{-i-2}}\cap D\cap F_{i}\subseteq\cset{f\brackconv}$. This implies
that $\cset{\varphi\leq2^{-i-1}}\cap D\cap F_{i}\subseteq\cset{f\brackconv}$, but $F_{i}\cap\cset{f\brackconv}=Q$,
so we have $D\cap F_{i}\cap\cset{\varphi\leq2^{-i-1}}=Q$, as required. So part \emph{(ii)} is satisfied by setting $C = D$.\editcom{Changed last sentence.}
\end{proof}

The following Proposition~\ref{prop:strong_min_and_cat_in_cont:1} and Corollary~\ref{cor:strong_min_and_cat_in_cont:1} are not used elsewhere in this paper but are useful in their own right.

\begin{prop}[Approximate Intersection]\label{prop:strong_min_and_cat_in_cont:1}
If $X$ is a dictionaric type space or definable set, $D$ is a definable
subset of $X$, $F\subseteq X$ is a closed set, and $U\supseteq F$
is an open-in-$X$ set, then there is a definable set $E$ such that
$F\subseteq E\subseteq U$, and such that $D\cap E$ is definable.
\end{prop}

\begin{proof}
First since $D$ is definable it is itself
dictionaric by Proposition \ref{prop:heredit}.\editcom{Removed `in a dictionaric type space.'} Let $G\subseteq D$ be a definable set such that $G\subseteq U$
and $G\supseteq F\cap D$.

Now all we need to do is argue that $G\cup F$ is relatively definable
in $D\cup F\cup(X\smallsetminus U)$, then the result follows from 
Proposition \ref{prop:ext}.

To see that $G \cup F$ is relatively definable in $D\cup F\cup(X\smallsetminus U)$,
first note that since $G$ is definable in $D$, $G$ is definable
in $X$. Now notice that any net $\{x_{i}\}_{i\in I}$ converging
to a point in $G\cup F$ must either be limiting to a point in $G$
or be eventually contained in $G\cup F$, so in either case we have
that $\lim_{i\in I}d(x,G\cup F)=0$, so $G\cup F$ is relatively definable.
Now let $E$ be an extension of $G\cup F$ to all of $X$. We have
that $E$ is the required definable set.\editcom{Split up last sentence.}
\end{proof}

\begin{cor}[Approximate Intersection of Definable Sets]\label{cor:strong_min_and_cat_in_cont:1}
  If $X$ is a dictionaric type space or definable set and $D$ and $E$ are definable subsets of $X$, then for every $\e > 0$ there is a definable set $E'\supseteq E$ such that $D\cap E'$ is definable, and $d_H(E,E') < \e$.
\end{cor}
\begin{proof}
  Apply Proposition~\ref{prop:strong_min_and_cat_in_cont:1} with $D=D$, $F=E$, and $U=E^{<\e/2}$.
\end{proof}

Although Corollary~\ref{cor:strong_min_and_cat_in_cont:1} is an easy consequence of Proposition~\ref{prop:strong_min_and_cat_in_cont:1}, for definable sets $D$ and $E$, $d_H(D,E) < \e$ can be a very useful condition. It implies that quantifying over elements of $D$ is `approximately the same as' quantifying over elements of $E$ to within an accuracy of $\e$. If $D$ is a uniformly discrete set, then there is an $\e>0$ such that $d_H(D,E)<\e$ implies that there is a definable equivalence relation $\sim$ on $E$ such that $E/\sim$ has a canonical bijection with $D$.\editcom{Added this paragraph.}

\begin{prop}[Preservation under Quotients] \label{prop:dict-quot}
If $X$ is a dictionaric type space or definable set and $\rho$ is
a definable pseudo-metric on $X$, then $X/\rho$ is dictionaric.
\end{prop}

\begin{proof}
Let $q\in U\subseteq X/\rho$ be a type with open neighborhood $U$,
consider $\pi^{-1}(q)\subseteq\pi^{-1}(U)$ where $\pi:X\rightarrow X/\rho$
is the natural projection map. By continuity $\pi^{-1}(q)$ is closed
and $\pi^{-1}(U)$ is open. Let $D$ be a definable set such that
$\pi^{-1}(q)\subseteq D$ and $D\cap U=\varnothing$. The projection
$\pi(D)$ is a definable set (with distance function $\inf_{y\in D}\rho(x,y)$).\editcom{Changed `its' to `with.'}
It's clearly contained in $U$ and contains $q$, so $X/\rho$ has
a network of definable sets and by Proposition \ref{prop:dict-equivalents}
$X/\rho$ is dictionaric.
\end{proof}
\begin{prop} \label{prop:dict-omega}
If $S_{n}(T)$ is dictionaric for every $n<\omega$, then $S_\omega(T)$ is dictionaric as well. 
\end{prop}

\begin{proof}
This follows from the fact that given any type $p\in S_{\omega}(T)$ and an open
neighborhood $U$, there is a restricted formula $\varphi(\overline{x})$
such that $p(\overline{x},...)\in\cset{\varphi(\overline{x})<\frac{1}{3}}$
and $\cset{\varphi(\overline{x})\leq\frac{2}{3}}\subseteq U$. By dictionaricness
of $S_{|\overline{x}|}(T)$ we can find a definable set $D$ such
that $\cset{\varphi(\overline{x})<\frac{1}{3}}\subseteq D\subseteq\cset{\varphi(\overline{x})\leq\frac{2}{3}}$
and this set is still definable in the sort of $\omega$-tuples under the metric
given in Definition \ref{defn:imag}, so $S_{\omega}(T)$ is
dictionaric as well. 
\end{proof}
\begin{cor} 
  \label{cor:dict-hered}If $T$ is dictionaric, then $T^{\mathrm{eq}}$ is dictionaric over parameters from the home sort (i.e.\ if $A$ is a set of parameters from the home sort and $I$ is some imaginary, then the type space $S_I(A)$ is dictionaric). In particular, $T^\mathrm{eq}$ is dictionaric over models.

  Furthermore, if $T$ is dictionaric over models, then $T^{\mathrm{eq}}$ is as well.\editcom{Removed `dictionaric over models.'}
\end{cor}

\begin{proof}
By Lemma~\ref{lem:imaginary-norm-form} every imaginary is equivalent to a definable subset of a definable quotient of the sort of $\omega$-tuples, so the result follows from the Propositions~\ref{prop:heredit}, \ref{prop:dict-quot}, and \ref{prop:dict-omega}.
\end{proof}

Unfortunately  we cannot, in general, conclude that $T^{\mathrm{eq}}$ is itself dictionaric. There are even discrete counterexamples, such as $\mathsf{RCF}$, where if we consider the hyperimaginary given by quotienting by the `infinitesimally close' equivalence relation (which can also be seen as a continuous imaginary sort)\editcom{Changed wording in parenthetical comment.} and if we let $a$ be any infinite element of this hyperimaginary, then $S_1(a)$ fails to be dictionaric in the connected component of $\tp(a)$. This phenomenon is related to failure of elimination of hyperimaginaries.\editcom{Removed `In particular.'} If $T$ is a discrete theory that eliminates hyperimaginaries, then the continuous $T^{\mathrm{eq}}$ is dictionaric. This implies that any stable or supersimple discrete theory will have dictionaric continuous $T^{\mathrm{eq}}$, as such theories always eliminate hyperimaginaries \cite{pillay1987, supersimple2000}. It would be nice to know if this extends to continuous dictionaric theories.

\begin{quest}
If $T$ is a continuous dictionaric theory and $T$ is also stable or supersimple, does it follow that $T^\mathrm{eq}$ is dictionaric?
\end{quest}

\subsection[Omega-Stable Theories are Dictionaric]{$\omega$-Stable Theories are Dictionaric}
\label{sec:w-stab-dict}

Now we will see that, fortunately, the technicality present in Corollary~\ref{cor:dict-hered} is not relevant in the rest of the paper, as every $\omega$-stable theory is dictionaric (and so also has dictionaric $T^\mathrm{eq}$). This is a topometric analog of the fact that scattered compact Hausdorff spaces are automatically totally disconnected.\editcom{Added `topometric analog of.'} But first we will need to following Lemma~\ref{lem:important-sep-lemma}, which is a metric space analog of the fact that functions from countable sets to $[0,1]$ are not surjective.\editcom{Added last sentence.}

\begin{lem}\label{lem:important-sep-lemma}
If $(X,d)$ is any separable metric space, and $f:X\rightarrow[0,1]$
is any function (not necessarily continuous), then for all but countably
many $r\in[0,1]$, $\{f\leq r\}\subseteq\overline{\{f<r\}},$ 
where $\{f \;\rectangle r\}$ is $\{x \in X : f(x) \; \rectangle r\}$ and $\bar{A}$ is the metric closure of $A$.\editcom{Changed some notation and the final `where' statement.}
\end{lem}

\begin{proof}
Assume that there are uncountably many $r\in[0,1]$ such that $\{f\leq r\}\not\subseteq\overline{\{f<r\}}$.
For each such $r$, let $x_{r}$ be a witnessing element of $\{f=r\}$
satisfying $\dinf(x_{r},\{f<r\})>0$.  Since there are uncountably many such
$x_{r}$, there is an $\varepsilon>0$ such that for some uncountable
$R$, $d(x_{r},\{f<r\})>\varepsilon$ for all $r\in R$. This
implies that if $r,s\in R$ and $r>s$, then $d(x_{r},x_{s})>\varepsilon$,
but this is an uncountable $(>\varepsilon)$\nobreakdash-\hspace{0pt}separated set, contradicting
that $X$ is a separable metric space. 
\end{proof}
\begin{prop}\label{prop:strong_min_and_cat_in_cont:3}
\leavevmode
\begin{enumerate}[label=(\roman*)]
\item If $X$ is a small type space or definable set (i.e.\ $X$ is metrically
separable), then $X$ is dictionaric.

\item If $T$ is  
  hereditarily small (i.e.\
$S_{n}(\overline{a})$ is metrically separable for every finite tuple of parameters
$\overline{a}$), then $T$ is dictionaric.

In particular, if
\begin{itemize}
\item $T$
is $\omega$-stable, 
\item hereditarily $\aleph_0$-categorical (i.e.\ $T_{\overline{a}}$ is
$\aleph_0$-categorical for every finite tuple of parameters\footnote{This does not automatically follow from $\aleph_0$-categoricity in continuous logic \cite[Ex.\ 17.7]{MTFMS}.}), or 
\item if
$T$ has an $\aleph_0$-saturated separable model (as opposed to just an approximately $\aleph_0$-saturated separable model),
\end{itemize}
 then $T$
is dictionaric.\footnote{Note that the second bullet point is a special case of the third bullet point.\editcom{Added footnote.}} 

Furthermore, totally transcendental theories---those in which every countable reduct is $\omega$-stable---are dictionaric.

\item\emph{(Strong Intersection Property)} If $S_n(A)$ is a small type space and $D\subseteq S_{n}(A)$ is
definable, then for any formula  $\varphi:S_{n}(A)\rightarrow[0,1]$, for all but countably many
$r\in[0,1]$, we have that $\cset{\varphi\leq r}$, $\cset{\varphi\geq r}$, $\cset{\varphi\leq r}\cap D$, and
$\cset{\varphi\geq r}\cap D$ are all definable.


\end{enumerate}
\end{prop}

\begin{proof}
(i). By the previous lemma, for any continuous function $P:S_{n}(T)\rightarrow[0,1]$
we have that for all but countably many $r\in[0,1]$, $\cset{P\leq r}=\overline{\cset{P<r}}$.
If a closed set is the metric closure of an open set then it is automatically
definable since $\cset{P\leq r}^{<\varepsilon}=\cset{P<r}^{<\varepsilon}$,
which is open.

(ii). Given a type $p$ and an open neighborhood $U$ there is a restricted
formula $\varphi(\overline{x};\overline{a})$ such that $p(\overline{x})\in\cset{\varphi(\overline{x};\overline{a})<\frac{1}{3}}$
and $\cset{\varphi(\overline{x};\overline{a})\leq\frac{2}{3}}\subseteq U$.
From which we get that there is a definable set $D$ such that $\cset{\varphi(\overline{x};\overline{a})\leq\frac{1}{3}}\subseteq D\subseteq\cset{\varphi(\overline{x};\overline{a})<\frac{2}{3}}$, namely, $\cset{\varphi(\bar{x};\bar{a}) \leq r}$ for some $r\in(\frac{1}{3},\frac{2}{3})$.
This is still definable in the full theory over the full parameter
set, so it is the required definable neighborhood of $p$.

Each of the bulleted conditions implies that $T$ is hereditarily small, and the fact that totally transcendental theories are dictionaric follows easily from the fact that each of their countable reducts are dictionaric, so the full result follows.

(iii). This follows from the fact that $D\subseteq S_{n}(A)$ is small
whenever $S_{n}(A)$ is small.
\end{proof}

It is possible to produce an example showing that in Proposition~\ref{prop:strong_min_and_cat_in_cont:3} we cannot in general guarantee that $\cset{\varphi = r}$ is a definable set for any $r \in [0,1]$, as in there is an $\omega$-stable theory and a $[0,1]$-valued formula $\varphi(x)$ such that $\cset{\varphi=r}$ fails to be definable for every $r \in [0,1]$.


Dictionaricness is not preserved under arbitrary reducts (any
non-totally transcendental discrete theory interprets Example~\ref{ex:main-bad-example}). 
Superstability does not imply dictionaricness (Example~\ref{ex:main-bad-example}) and neither
does this plausible sounding separation axiom:
\begin{itemize}
\item[$(\ast)$] For every $p,q\in X$,
with $p\neq q$, there are disjoint definable sets $D$ and $E$ such that
$p\in\tint _{X}D$, $q\in\tint _{X}E$.
\end{itemize}
Examples of $(\ast)$ are actually fairly prevalent. Every theory is bi-interpretable with a theory satisfying $(\ast)$: If $\frk{S}=(S^1,X,Y)$ is the structure whose universe is the unit circle in $\mathbb{R}^2$ with the standard metric and which has predicates $X$ and $Y$ for the $x$- and $y$-coordinates, then for any structure $\frk{M}$, $\mathrm{Th}(\frk{M} \times \frk{S})$, where we take product structures to have both of the factor structures' metrics as predicates, has the required property. The key is that if $\varphi$ is any $S^1$-valued formula in the language of $\frk{M}$, then the set $\{\left<x,y\right> \in \frk{M} \times \frk{S} : d(\varphi^{\frk{M}}(x),y)\leq \e\}$ is always definable, with similar statements for $n$-tuples. Theories whose models have more than one element always interpret $\frk{S}$ and theories whose models have one or fewer elements automatically satisfy $(\ast)$.\editcom{Restructured paragraph and made example more explicit.} 

These examples, however, still have imaginary sorts failing $(\ast)$, which raises the following question.
\editcom{Added this paragraph and question.}

\begin{quest}
  If $T$ is a theory such that every type space over $T^{\mathrm{eq}}$ satisfies $(\ast)$, does it follow that $T$ or $T^{\mathrm{eq}}$ is dictionaric?
\end{quest}

\section{\Insep\ Categorical Theories with Strongly Minimal Sets}

\subsection{(Strongly) Minimal Sets} \label{sec:SMS} There are existing definitions of minimal set and strongly minimal set in the literature given in \cite{Noquez2017}, but both of these definitions are too weak. The definition of strongly minimal set given there fails to generalize the notion of strongly minimal in discrete logic and instead corresponds to a countably type-definable set containing a unique non-algebraic type which furthermore is minimal (i.e.\ a type for which every forking extension is algebraic).\editcom{Changed wording and added `furthermore' to make this sentence a little clearer.} The definition of minimal set given there is trivial in the sense that under it every countable theory with non-compact models has a minimal zeroset over any given separable model. 
The definitions given here are equivalent to the definitions given in \cite{Noquez2017} with the extra stipulation that the sets are definable rather than just zerosets. In particular this means that the notion of a strongly minimal \emph{theory} given in \cite{Noquez2017} is equivalent to the one given here.

\begin{defn}
\leavevmode
\begin{enumerate}[label=(\roman*)]
\item A non-algebraic definable set $D$ is \emph{minimal} (over the set $A$)
if for each pair $F,G\subseteq D$ of disjoint $A$-zerosets,
at most one of $F$ or $G$ is non-algebraic.

\item A definable set is \emph{strongly minimal} if it is minimal over every
set of parameters.
\end{enumerate}
\end{defn}
The na\"ive translation of the definition of minimal set---every set
$\cset{P(\mathfrak{M})\brackconv}\cap D(\mathfrak{M})$ is either compact or co-pre-compact (i.e.\ has a complement with a compact closure)---does not work:
\begin{ex}
\label{ex:not-comp-nor-co-comp}
A strongly minimal set $D$ with a definable set $E\subseteq D$ that
is neither compact nor co-pre-compact.
\end{ex}

\begin{proof}[Verification]
Let $\mathfrak{M}$ be a structure whose universe is $\omega\times S^{1}$,
where $S^{1}\subseteq\mathbb{C}$ is the unit circle with the standard
Euclidean metric. Let the distance between any points in distinct
circles be $1$. Let $D$ be the entire structure, and let $E$ be the subset of $S_1(\frk{M})$ given by  $\{(n,e^{2\pi ki/(n+1)}):k\leq n\}$ (where we are identifying elements of $\frk{M}$ with their types in $S_1(\frk{M})$) together with the unique non-algebraic type. This set is clearly closed. 
To see that it is a definable set, pick $\varepsilon>0$ and consider $E^{<\varepsilon}$.
For any $n>\frac{4\pi}{\varepsilon}$, $E^{<\varepsilon}$ contains
all of the circle $\{n\}\times S^{1}$. There are only finitely many
$n\leq\frac{4\pi}{\varepsilon}$, and on each of these $E^{<\varepsilon}$
is an open set since the logic topology agrees with the metric topology
on each individual circle in $\mathfrak{M}$. Therefore $E^{<\varepsilon}$
is an open set, and so $E$ is definable.
\end{proof}
Another example is $(-\infty,0]\cup\{\ln(1+n):n<\omega\}$, which
is an $\mathbb{R}$-definable subset of $\mathbb{R}$ (which will be shown to be strongly minimal as a metric space with the appropriate metric in Theorem~\ref{thm:strong_min_and_cat_in_cont:2}).\editcom{Changed `is' to `will be.'}

\subsubsection{Some Characterizations}

Here we present some more traditional characterizations of minimal sets.\editcom{Changed this to actually be a sentence.}
\begin{prop}
For a definable set $D$ over the structure $\mathfrak{M}$ the following
are equivalent:
\begin{enumerate}[label=(\roman*)]
\item $D$ is minimal.

\item For each restricted $M$-formula $\varphi$, at most one of $\cset{\varphi(\mathfrak{M})\leq\frac{1}{3}}$
or $\cset{\varphi(\mathfrak{M})\geq\frac{2}{3}}$ is non-compact. (In particular
we only need to check compactness in $\mathfrak{M}$, not in arbitrary
elementary extensions of $\mathfrak{M}$.)\footnote{Note that there is nothing special about $\frac{1}{3}$ and $\frac{2}{3}$ (beyond the fact that they are distinct) or the fact that these are inequalities rather than equalities, so we get the same statement with $\cset{\varphi=0}$ and $\cset{\varphi=1}$. The proof of the statement in the proposition is conceptually clearer, however.\editcom{Added `beyond the fact....' Made into footnote.}}

\item $D$ is dictionaric (as a subset of $S_{n}(\mathfrak{M})$\/) 
and for every $M$-definable subset $E$ of $D$, either $E(\mathfrak{M})$
is compact or $D(\mathfrak{M})\cap\cset{d(\mathfrak{M},E)\geq\varepsilon}$
is compact for every $\varepsilon>0$.
\end{enumerate}
\end{prop}

\begin{proof}
In this proof we will use the notation $A^{\geq \e}$ for the set types with distance from $A$ greater than or equal to $\e$.\editcom{Added sentence.}
  
\emph{$\text{(i)}\Rightarrow\text{(ii)}$.} This is obvious.

\emph{$\text{(ii)}\Rightarrow\text{(iii)}$.} First we will show that $D$ has a network
of definable sets, which is sufficient by Proposition \ref{prop:dict-equivalents} part \emph{(v)}. Let $p\subseteq D$ be a type and let $U$ be an open-in-$D$ neighborhood of $p$.  If $p$ is algebraic, then we are done, as $\{p\}$ is
a definable set. If $p$ is non-algebraic, then find $V$ such that
$p\in V\subseteq\overline{V}\subseteq U$.\editcom{Small format and notation changes.}

We want to argue that every $q\in\partial V$ is an algebraic type.
For each $q\in\partial V$ find restricted formula $\psi$ such that
$p\in\cset{\psi<\frac{1}{3}}$ and $q\in\cset{\psi>\frac{2}{3}}$. Since $p$
is non-algebraic and $V$ is open, it must be the case that $\cset{\psi(\mathfrak{M})\leq\frac{1}{3}}$ 
is non-compact. Therefore $\cset{\psi(\mathfrak{M})\geq\frac{2}{3}}$ is
compact. Since $q$ is contained in the interior of a zeroset that
is compact in some model, it is algebraic by Lemma \ref{lem:basic-alg}.
Since every $q\in\partial V$ is algebraic, and therefore definable,
$\overline{V}$ is a union of open and definable sets that is topologically
closed, and therefore definable.\footnote{Note that the property `$A^{<\e}$ is open for every $\e >0$' is preserved under arbitrary unions.\editcom{Added footnote.}} Hence $D$ has a network of definable
sets and is dictionaric.\editcom{`Then' to `Hence.'}

\editcom{Removed `now.'} If $E\subseteq D$ is an $M$-definable set, then for any $\varepsilon>0$,
we can find a restricted formula $\varphi$ such that $E\subseteq\cset{\varphi<\frac{1}{3}}$
and $D\cap E^{\geq\varepsilon}\subseteq\cset{\varphi>\frac{2}{3}}$. Either
$\cset{\varphi(\mathfrak{M})\leq\frac{1}{3}}$ and therefore $E(\mathfrak{M})$
is compact, or $\cset{\varphi(\mathfrak{M})\geq\frac{2}{3}}$ and therefore
$D(\mathfrak{M})\cap\cset{d(\mathfrak{M},E)\geq\varepsilon}$ is compact.
This implies that either $E(\mathfrak{M})$ is compact or $D(\mathfrak{M})\cap\cset{d(\mathfrak{M},E)\geq\varepsilon}$
is compact for every $\varepsilon>0$, as required.

\emph{$\text{(iii)}\Rightarrow\text{(i)}$.} If $F,G\subseteq D$ are disjoint, $M$-zerosets then we can find $E\supseteq F$ such that $E\subseteq D$
is a definable set disjoint from $G$. $E(\mathfrak{M})$ is either
compact or $D(\mathfrak{M})\cap\cset{d(\mathfrak{M},E)\geq\varepsilon}$
is compact for every $\varepsilon>0$. If $E(\mathfrak{M})$ is compact
then $E(\mathfrak{N})$ and therefore $F(\mathfrak{N})$ is compact
in every model $\mathfrak{N}\succeq \frk{M}$, since $E$ is definable. Otherwise
there is some $\varepsilon>0$ small enough that $G\cap E^{<\varepsilon}=\varnothing$.
In that case find $H\supseteq E^{\geq \e}$\editcom{Changed notation.} such that
$H\subseteq D$ is a definable set disjoint from $E$. Since $H$
is disjoint from $E$ there is some $\delta>0$ small enough that
$H\cap E^{<\delta}=\varnothing$, so $H(\mathfrak{M})\subseteq D(\mathfrak{M})\cap\cset{d(\mathfrak{M},E)\geq\delta}$
is a compact set. Since $H$ is definable we have that $H(\mathfrak{N})$
is compact in every model $\mathfrak{N}\succeq \frk{M}$, therefore $G(\mathfrak{N})$
is as well.
\end{proof}
\editcom{Removed `Obviously.'} We will ultimately show that any minimal set contains a
unique non-algebraic type, but this will be a corollary of something
slightly more technical so we will defer this to later (Proposition \ref{prop:some-min-stuff} and Corollary \ref{cor:min-set-type-correspondence}). The condition
that $D$ be dictionaric when restricting attention to definable sets
is necessary in light of Example \ref{ex:main-bad-example}, since the theory there has more than one non-algebraic $\varnothing$-type
but appears `strongly minimal' with regards to definable sets in that every definable set is either finite or co-finite.

Compare the following Proposition~\ref{prop:vaught-minimal} to the classical facts that minimal sets in theories with no Vaughtian pairs or over $\omega$-saturated models are strongly minimal.\editcom{Added paragraph.}

\begin{prop}
\label{prop:vaught-minimal}
\leavevmode
\begin{enumerate}[label=(\roman*)]
\item If $D$ is a minimal set over a model
$\mathfrak{M}$ and either
\begin{itemize}
\item $T$ is dictionaric and has no Vaughtian
pairs or
\item $T$ has no open-in-definable Vaughtian pairs,
\end{itemize}
   then $D$
is strongly minimal.

\item If $D$ is a minimal set over $\mathfrak{M}$, an approximately
$\aleph_0$-saturated model, then $D$ is strongly minimal.
\end{enumerate}
\end{prop}

\begin{proof}
\emph{(i).} Assume without loss of generality that $T$ is countable. (If we prove this
for each countable reduct of $T$ over which $D$ is definable, then
it will be true for the whole theory.) Assume that $D$ is a minimal set containing the non-algebraic type $p$ and that $p$ is not a strongly minimal type. Let
$q_{0},q_{1}$ be distinct non-algebraic extensions of $p$ to some
parameter set $A\supseteq\mathfrak{M}$. Let $\varphi(\overline{x};\overline{a})$
be a restricted formula with $\overline{a}\in A$ such that $\varphi(q_{i};\overline{a})=i$
for both $i<2$. Let $\{\overline{b}_{i}\}_{i<\omega}$ be a sequence of
tuples from $\mathfrak{M}$ such that $\mathrm{tp}(\overline{b}_{i}/\mathfrak{M})\rightarrow\mathrm{tp}(\overline{a}/\mathfrak{M})$.
It must be the case that either infinitely many $i<\omega$ have $\varphi(p;\overline{b}_{i})<\frac{2}{3}$
or infinitely many $i<\omega$ have $\varphi(p;\overline{b}_{i})>\frac{1}{3}$.
Without loss of generality assume that infinitely many $i<\omega$ have $\varphi(p;\overline{b}_{i})>\frac{1}{3}$
and restrict to a sub-sequence on which this is always true.

Now we have that for each $i<\omega$, $D\cap\cset{\varphi(-;\overline{b}_{i})\leq\frac{1}{3}}$
is algebraic. Let $\mathfrak{N}\succ\mathfrak{M}$ be a proper elementary
extension. Note that since $D\cap\cset{\varphi(-;\overline{b}_{i})\leq\frac{1}{3}}$
is algebraic we have that $D(\mathfrak{M})\cap\cset{\varphi(\mathfrak{M};\overline{b}_{i})\leq\frac{1}{3}}=D(\mathfrak{N})\cap\cset{\varphi(\mathfrak{N};\overline{b}_{i})\leq\frac{1}{3}}$
for each $i<\omega$.

Let $\mathfrak{A}_{i}=(\mathfrak{N},\mathfrak{M},\overline{b}_{i})$
and take a non-principal ultraproduct
\(
\mathfrak{A}_{\omega}=\prod_{i<\omega}\mathfrak{A}_{i}/\mathcal{U}\allowbreak=(\mathfrak{N}_{\omega},\mathfrak{M}_{\omega},\overline{b}_{\omega}).
\)
Note that $\mathfrak{N}_{\omega}\succ\mathfrak{M}_{\omega}$,\footnote{This
is not true for arbitrary ultraproducts of proper elementary pairs;
you need a uniform `width' of the extensions, as in you need a fixed $\e>0$ such that most $\mathfrak{N}_i$ contain an $a$ with $\dinf(a,\frk{M}_i)>\e$.} that $D(\mathfrak{M}_{\omega})\cap\cset{\varphi(\mathfrak{M}_{\omega};\overline{b}_{\omega})\leq\varepsilon}=D(\mathfrak{N}_{\omega})\cap\cset{\varphi(\mathfrak{N}_{\omega};\overline{b}_{\omega})\leq\varepsilon}$
for any $0<\varepsilon<\frac{1}{3}$, and that $\mathrm{tp}(\overline{b}_{\omega}/\mathfrak{M})=\mathrm{tp}(\overline{a}/\mathfrak{M})$.
Therefore $p$ has two non-algebraic extension $r_{0},r_{1}$ to $S_{n}(\mathfrak{M}_{\omega})$
such that $\varphi(r_{i};\overline{b}_{\omega})=i$ for both $i<2$. This
implies that $D(\mathfrak{M}_{\mathfrak{\omega}})\cap\cset{\varphi(\mathfrak{M}_{\omega};\overline{b}_{\omega})\leq\varepsilon}$
is not metrically compact for any $\varepsilon>0$.

If $T$ is dictionaric we can find a definable set $E$ such that
$D\cap\cset{\varphi(-;\overline{b}_{\omega})\leq\frac{1}{6}}\subseteq E\subseteq D\cap\cset{\varphi(-;\overline{b}_{\omega})<\frac{1}{3}}$,
then we will have that $E(\mathfrak{M}_{\omega})=E(\mathfrak{N}_{\omega})$
is a non-compact definable set, witnessing that $(\mathfrak{N}_{\omega},\mathfrak{M}_{\omega})$
is a Vaughtian pair.

Otherwise the set $D\cap\cset{\varphi(-;\overline{b}_{\omega})<\frac{1}{3}}$
is open-in-definable and unbounded (it is unbounded because $D\cap\cset{\varphi(-;\overline{b}_{\omega})\leq\frac{1}{6}}$
is closed and non-algebraic), witnessing that $(\mathfrak{N}_{\omega},\mathfrak{M}_{\omega})$
is an open-in-definable Vaughtian pair.

\emph{(ii).} Let $D(x;\overline{a})$ with $\overline{a}\in\mathfrak{M}$,
our approximately $\aleph_0$-saturated model, be a definable set such
that over some set of parameters $B$ there are distinct non-algebraic
types $p_{0},p_{1}\in D$. Let $\varphi(x;\overline{b})$ be a formula
such that $\varphi(p_{i};\overline{b})=i$ for both $i<2$. Since $p_{0}$
and $p_{1}$ are non-algebraic we have that $D\cap\cset{\varphi(-;\overline{b})\leq\frac{1}{6}}$
and $D\cap\cset{\varphi(-;\overline{b})\geq\frac{5}{6}}$ are both non-algebraic.
Let $\varepsilon>0$ be small enough that $\#_{>\varepsilon}^{\mathrm{ent}}D(\mathfrak{N})\cap\cset{\varphi(\mathfrak{N};\overline{b})\leq\frac{1}{3}}$
and $\#_{>\varepsilon}^{\mathrm{ent}}D(\mathfrak{N})\cap\cset{\varphi(\mathfrak{N};\overline{b})\geq\frac{2}{3}}$
are both infinite in any model $\mathfrak{N}$ (this is always possible
because $D\cap\cset{\varphi(-;\overline{b})<\frac{1}{3}}$ and $D\cap\cset{\varphi(-;\overline{b})>\frac{2}{3}}$
are relatively open in $D$).

Find $\delta>0$ small enough that $\delta<\frac{1}{3}\varepsilon$
and if $d(x,y)<\delta$, then $|\varphi(x;\overline{b})-\varphi(y;\overline{b})|<\frac{1}{9}$.
Find $\gamma>0$ small enough that if $d(\overline{z},\overline{w})<\gamma$,
then $\sup_{x}|D(x;\overline{z})-D(x;\overline{w})|<\delta$. By approximate
$\aleph_0$-saturation we can find $\overline{c}\overline{e}\equiv\overline{a}\overline{b}$
such that $d(\overline{a},\overline{c})<\gamma$ and $\overline{c}\overline{e}\in\mathfrak{M}$.
Let $\{u_{i}\}_{i<\omega}$ be an infinite $({>}\varepsilon)$-separated
set in $D(\mathfrak{M};\overline{c})\cap\cset{\varphi(\mathfrak{M};\overline{e})\leq\frac{1}{3}}$
and let $\{v_{i}\}_{i<\omega}$ be an infinite $({>}\varepsilon)$-separated
set in $D(\mathfrak{M};\overline{c})\cap\cset{\varphi(\mathfrak{M};\overline{e})\geq\frac{2}{3}}$.
By construction we have that $d_{H}(D(\mathfrak{M};\overline{a}),\allowbreak D(\mathfrak{M};\overline{c}))<\delta$
(where $d_{H}$ is the Hausdorff metric), 
 so we can find $\{w_{i}\}_{i<\omega}\subseteq D(\mathfrak{M};\overline{a})$
and $\{t_{i}\}_{i<\omega}\subseteq D(\mathfrak{M};\overline{a})$
such that $d(v_{i},w_{i})<\delta$ and $d(u_{i},t_{i})<\delta$ for
each $i<\omega$. By construction this implies that $\varphi(w_{i};\overline{e})<\frac{4}{9}$
and $\varphi(t_{i};\overline{e})>\frac{5}{9}$ for every $i<\omega$, and hence we have that $D(\mathfrak{M};\overline{a})\cap\cset{\varphi(\mathfrak{M};\overline{e})\leq\frac{4}{9}}$
and $D(\mathfrak{M};\overline{a})\cap\cset{\varphi(\mathfrak{M};\overline{e})\geq\frac{5}{9}}$
are both not metrically compact and therefore not algebraic. Therefore
$D(x;\overline{a})$ is not minimal over $\mathfrak{M}$.

So by the converse we have that if $D$ is a minimal set over
an approximately $\aleph_0$-saturated model, then it is strongly minimal.
\end{proof}
It is unclear whether or not Proposition \ref{prop:vaught-minimal} part \emph{(i)} can be improved to theories with no Vaughtian pairs in general, in that while a minimal
set over $A$ is dictionaric over $A$, it may not be dictionaric
over some $B\supseteq A$.

\subsubsection{The Pregeometry of a Strongly Minimal Set}

\begin{prop}
If $D$ is a strongly minimal set definable over the set $A$, then
$X \mapsto \mathrm{acl}(XA)$ restricted to $D$ is the closure operator
of a pregeometry with finite character.
\end{prop}

\begin{proof}
$\mathrm{acl}$ automatically obeys reflexivity and transitivity, so
all we need to verify is finite character and exchange.

For finite character, suppose that $a\in D(\mathfrak{C})$ and $a\in\mathrm{acl}(B)$
for some set $B\supseteq A$ (not necessarily in $D$). Let $\chi(\overline{x};\overline{b})$
be a distance predicate for an algebraic subset of $D$ containing
$a$, with $\overline{b}\in B$. Let $p$ be the unique non-algebraic
global type in $D$. Find a restricted formula $\varphi$ such that
$\sup_{\overline{x},\overline{y}}|\chi(\overline{x};\overline{y})-\varphi(\overline{x};\overline{y})|<\frac{1}{4}d(p,\mathfrak{C})$.
$\varphi(\overline{x};\overline{b})$ only depends on finitely many
parameters in the tuple $\overline{b}$. Note the following: $\cset{\chi(-;\overline{b})\brackconv}\subseteq D\cap\cset{\varphi(-;\overline{b})\leq\frac{1}{4}d(p,\mathfrak{C})}\subseteq\cset{\chi(-;\overline{b})\leq\frac{1}{2}d(p,\mathfrak{C})}\not\ni p$.
Therefore $D\cap\cset{\varphi(-;\overline{b})\leq\frac{1}{4}d(p,\mathfrak{C})}$
contains $a$ and is algebraic, so $a\in\mathrm{acl}(B_{0})$ for some
finite $B_{0}\subseteq B$.

For exchange, let $b\in D(\mathfrak{C})\smallsetminus\mathrm{acl}(A)$
and let $c\in D(\mathfrak{C})\smallsetminus\mathrm{acl}(Ab)$. We want to
show that $b\notin\mathrm{acl}(Ac)$. Since $D$ has a unique non-algebraic
type over any parameter set, any such pair like $bc$ has the same
type. Let $B^{\prime}$ be a set of realizations of $\mathrm{tp}(b/A)$
of cardinality $(|\mathcal{L}|+|A|+2^{\aleph_0})^{+}$ and let $c^{\prime}$
be a realization of the unique non-algebraic type in $D$ over the
set $AB^{\prime}$. For any $b_{0}^{\prime},b_{1}^{\prime}\in B^{\prime}$
we have that $b_{0}^{\prime}c^{\prime}\equiv_{A}b_{1}^{\prime}c^{\prime}$,
but there are too many realizations of $\mathrm{tp}(b_{0}^{\prime}/Ac^{\prime})$
for it to be algebraic, so we must have $b^{\prime}\notin\mathrm{acl}(Ac^{\prime})$
for any $b^{\prime}\in B^{\prime}$. Since $b^{\prime}c^{\prime}\equiv bc$,
the same must be true for $b$ and $c$, so $b\notin\mathrm{acl}(Ac)$,
as required. 
\end{proof}
Note that in a strongly minimal set, $\mathrm{acl}$ has finite character relative to arbitrary parameters, not just those in the strongly minimal set. In fact it has finite character in a particularly strong sense in that if $c \in \mathrm{acl}(AB)$, where $c$ is in a strongly minimal set over $A$, then there is a finite tuple $\overline{c}\in C$ and an $\varepsilon>0$ such that for any $\overline{e}$ with $d(\overline{c},\overline{e})<\varepsilon$, $b \in \mathrm{acl}(A\overline{e})$.

It follows from this proposition that all of the machinery of pregeometries, such as
bases and invariant dimension number, works
in continuous logic just as it does in discrete logic. The following corollary summarizes this more precisely.\editcom{Changed this last bit into a longer sentence.} 
\begin{cor}
If $D$ is a strongly minimal set definable over the set $A$ (which
we may assume without loss of generality is countable), then for any model $\mathfrak{M}\supseteq A$,
we have that $D(\mathfrak{M})\subseteq\mathrm{acl}(AB)$ for any $B\subseteq D(\mathfrak{M})$ which is
a basis with regards to the pregeometry induced by the closure operator
$\mathrm{acl}(-\cup A)$.

Any two bases of $D(\mathfrak{M})$ have the
same cardinality, any two bases of the same cardinality are elementarily
equivalent, and any basis is an $A$-indiscernible set. So in particular
if $B_{i}\subseteq D(\mathfrak{M}_{i})$ for both $i<2$ are bases of the
same cardinality, then any bijection $f:B_{0}\rightarrow B_{1}$ extends
to an elementary map $f^{\prime}:D(\mathfrak{M}_{0})\equiv D(\mathfrak{M}_{1})$.

Finally, for any $X\subseteq D(\mathfrak{C})$, if $\#^{\mathrm{dc}}X > \#^\mathrm{dc} \mathrm{acl}(A) + |\mathcal{L}|$
then $\mathrm{dim}(X)=\#^{\mathrm{dc}}X$. So in particular, if $T$ and $A$ are countable, then for any $X\subseteq D(\mathfrak{C})$ with $\#^\mathrm{dc}X$ uncountable, $\mathrm{dim}(X) = \#^\mathrm{dc}X$.
\end{cor}

The following lemma is useful for understanding the metric properties of strongly minimal sets. Perhaps unsurprisingly, the metric in a strongly minimal set always behaves somewhat like a locally finite edge relation in a discrete strongly minimal set.\editcom{Added this sentence.} We are including it in this subsection because its proof relies on the pregeometric structure of strongly minimal sets.

\begin{lem}[Uniform Local Compactness]\label{lem:uniform-local-compactness}
  Let $D$ be a strongly minimal set definable over the set $A$. There exists an $\e > 0$ such that for every $\delta > 0$, there exists an $n<\omega$ such that for any model $\frk{M}\supseteq A$,  every closed $\e$-ball $B$ in $D(\frk{M})$ can be covered by at most $n$ open $\delta$ balls.\editcom{Added `at most.'}

  So in particular for any $a,b \in D(\frk{M})$ with $d(a,b)< \e $, we have that $a\in \mathrm{acl}(b)$.\editcom{Changed wording.}
\end{lem}
\begin{proof}
  For the purposes of this proof we will use the notations $B_{\leq \alpha}(x)$ and $B_{<\alpha}(x)$ to represent closed and open balls in $D$ (rather than in the ambient structure).
  
  Let $p$ be the generic type in $D$ over $A$. Let $b$ be a realization of $p$. Let $q$ be the generic type in $D$ over $Ab$. We have that $d(b,q) > 0$. Find $\e >0$ satisfying $\e < \frac{d(b,q)}{2}$. The set $\cset{d(x,b)\leq \e}$ does not contain $q$, so it must be algebraic. This implies that for every $\delta > 0$ there is an $m_\delta < \omega$ such that $B_{\leq \e}(b)$ can be covered by $m_\delta$ open $\delta$-balls. Consider the formula
    \[
    \chi_\delta(x)\coloneqq \sup_{y_0,\dots,y_{m_\delta-1}}\inf_{z}\max \{d(x,z)\dotdiv \e,\max_{i<m_\delta}\delta\dotdiv d(y_i,z)\}.
  \]
  Note that $\chi_\delta(x)$ cannot take on negative values.
  Let $c_0,\dots,c_{m_\delta}$ be chosen so that the open $\delta$-balls centered at these points cover $B_{\leq\e}(b)$. Since $B_{\leq \gamma}$ is compact for some $\gamma > \e$, by compactness there is a $\gamma > \e $ such that $B_{\leq \gamma}(b)$ is also covered by $\bigcup_{i < m_\delta}B_{<\delta}(c_i)$. Then by compactness again we get that there is some $\theta<\delta$ such that $\bigcup_{i<m_\delta}B_{<\theta}(c_i)$ covers $B_{\leq \gamma}(b)$. By plugging in these $c_i$'s for the $y_i$'s in $\gamma_\delta(b)$, we get a witness that $\chi_\delta(a) > \min\{\gamma-\e, \delta-\theta\}>0$. Conversely, we have that for any $b \models D$, if $\chi_\delta(b) > 0$, then $B_{\leq \delta}(b)$ can be covered by $m_\delta$ open $\delta$-balls. The set $\cset{\chi_\delta(x) > 0}$ is an open neighborhood of $p$, so we have that $D \cap \cset{\chi_\delta(x)= 0}$ is algebraic. For any $e \in D$ with $\chi_\delta(e) = 0$, we must have that $d(e,p) > \e$. To see this, let $f$ be a realization of $q$. We have that $d(b,f) \geq d(b,q) > 2\e$, by the choice of $\e$. Since $e$ is in $\mathrm{acl}(\varnothing)$, $b$ and $f$ have the same type over $Ae$, implying that $d(b,e)=d(f,e)$. By the triangle inequality, these must both be greater than $\e$. This implies that the set $F = (\cset{\chi_\delta(x)=0}\cap D)^{\leq \e}\cap D$ does not contain $p$ and is therefore also algebraic. This implies that there is some $k_\delta<\omega$ such that $F(\frk{C})$ is covered by $k_\delta$ open $\delta$-balls, which implies that for any $e \in D$ with $\chi_\delta(e)=0$, $B_{\leq \e}(e)$ is also covered by $k_\delta$ open $\delta$-balls (since it is a subset of $F(\frk{C})$). Therefore we have that for any $g \in D$, $B_{\leq \e}(g)$ is covered by no more than $n_\delta = \max\{m_\delta,k_\delta\}$ open $\delta$-balls.

  Since we can do this for any $\delta > 0$, and since this property is preserved under passing to substructures (possibly increasing $n$ as a function of $\delta$), we have the required result.\editcom{Added parenthetical comment.}
\end{proof}

\subsubsection{Strongly Minimal Types}

Now we would like to give definitions of strongly minimal types directly and relate them to the notion of strongly minimal sets. As part of this we will need a notion of `approximate algebraicity.'

\begin{defn}\label{def:strongly-minim-types}
\leavevmode
\begin{enumerate}[label=(\roman*)]
\item A type $p\in S_{n}(A)$ is said to be \emph{$(<\varepsilon)$-algebraic}
if for any $\mathfrak{M}\supseteq A$ and any $q\in S_{n}(\frk{M})$ such
that $q\upharpoonright A=p$, $d(q,\mathfrak{M})<\varepsilon$ (thinking of $\frk{M}$ as the set of types it realizes in $S_{n}(\frk{M})$).

\item A set of types is said to be \emph{$(<\varepsilon)$-algebraic} if every type in it is.

\item A type $p\in S_{n}(A)$ is \emph{pre-minimal (over $A$)} if for all sufficiently small
$\varepsilon>0$, $p$ is $d$-atomic in the set of non-$(<\varepsilon)$-algebraic
types in $S_{n}(A)$.\footnote{Pre-minimal types have the same relationship with minimal sets that strongly minimal types have with strongly minimal sets. There is a terminological issue we are inheriting from discrete logic here. In discrete logic, strongly minimal sets have strongly minimal types, and weakly minimal sets have minimal types, but minimal sets do not get a name for their special types (which are admittedly not very special). Calling such types `weakly minimal,' while arguably sensible, would invite confusion. `Locally minimal' is another option, but this is unfortunate when talking about `locally minimal global types' and erroneously suggests a direct relationship with local types.}

\item A type $p\in S_{n}(A)$ is \emph{strongly minimal} if it has a unique
pre-minimal global extension.\editcom{Removed `where $\frk{C}....$'} 
\end{enumerate}
\end{defn}

Note that if a zeroset $F$ is $(<\varepsilon)$-algebraic, then in any model $\mathfrak{M}$, by compactness $F(\mathfrak{M})$ is covered by finitely many $(<\varepsilon)$-balls with centers in $\mathfrak{M}$. Something approximating the converse is true as well, but we won't need it. Also note that a zeroset $F$ is algebraic if and only if it is $(<\e)$-algebraic for every $\e > 0$.

Note that the set of non-$(<\varepsilon)$-algebraic types is always
closed, and if $\mathfrak{M}$ is a model, then the set of non-$(<\varepsilon)$-algebraic
types in $S_{1}(\mathfrak{M})$ is precisely $S_{1}(\mathfrak{M})\smallsetminus\mathfrak{M}^{<\varepsilon}$,
so in particular a type $p\in S_{n}(A)$ is strongly minimal if for
all sufficiently small $\varepsilon>0$, $p$ has a unique global
extension in the set $S_{n}(\mathfrak{C})\smallsetminus\mathfrak{C}^{<\varepsilon}$ (where $\frk{C}$ is the monster model),
and furthermore that extension is relatively $d$-atomic in $S_{n}(\mathfrak{C})\smallsetminus\mathfrak{C}^{<\varepsilon}$.\editcom{Added `where $\frk{C}....$'}

Compare this with the following definition of strongly
minimal types in discrete logic: A type $p\in S_{n}(A)$ is strongly
minimal if it has a unique global extension in the set $S_{n}(\mathfrak{C})\smallsetminus\mathfrak{C}$,
and furthermore that extension is relatively isolated in $S_{n}(\mathfrak{C})\smallsetminus\mathfrak{C}$.

Note that the definition of strongly minimal type implies that if
$p$ is a strongly minimal type and $q$ is a global extension of
it that is not equal to $p$'s special extension, then $q\in\mathfrak{C}^{<\varepsilon}$
for all sufficiently small $\varepsilon>0$, so in particular $q$
is realized and therefore algebraic. This implies that over any parameter
set containing the domain of $p$, $p$ has a unique non-algebraic
extension.

The following is a characterization of strongly minimal types in terms of Morley rank and degree, which is developed in the context of continuous logic in \cite{BenYaacov2008}. The notion of Morley rank used here is the one corresponding to the $(f,\varepsilon)$-Cantor-Bendixson derivative defined in that paper, but the statement of this proposition is not sensitive to the particular kind of Morley rank used, since they are all asymptotically equivalent as $\varepsilon \rightarrow 0$ and since in a totally transcendental theory the statement that $Md_\varepsilon(p)=1$ for all sufficiently small $\varepsilon>0$ is precisely the same as the statement that $p$ is a stationary type.\editcom{Changed sentence structure to make it clearer.} The machinery of Morley rank\editcom{Removed `s.'} is not used anywhere else in this paper, and this proposition is included only for comparison to the classical statement that a type $p$ is strongly minimal if and only if $M(R,d)(p)=(1,1)$ (where $M(R,d)(p)=(MR(p),Md(p))$) so we will omit actually defining Morley rank here.\editcom{Not sure if I want to standardize my ordered pair notation.}
\begin{prop}
A type $p\in S_{n}(A)$ is strongly minimal if and only if $M(R,d)_{\varepsilon}(p)=(1,1)$
for all sufficiently small $\varepsilon>0$.
\end{prop}

\begin{proof}
$(\Rightarrow)$. Assume that $p\in S_{n}(A)$ is strongly minimal.
Let $q$ be the unique pre-minimal global extension, so that in particular
for all sufficiently small $\varepsilon>0$, $q$ is relatively $d$-atomic in $S_{n}(\mathfrak{C})\smallsetminus\mathfrak{C}^{<\varepsilon}$.
$(S_{n}(\mathfrak{C}))_{\varepsilon}^{\prime}=S_{n}(\mathfrak{C})\smallsetminus\mathfrak{C}^{<\varepsilon}$,
so $q$ is relatively $d$-atomic in $(S_{n}(\mathfrak{C}))_{\varepsilon}^{\prime}$.
Therefore $q\notin(S_{n}(\mathfrak{C}))_{\varepsilon}^{\prime\prime}$
and $q$ is contained in open subsets of $(S_{n}(\mathfrak{C}))_{\varepsilon}^{\prime}$
of arbitrarily small diameter, so $M(R,d)_{\varepsilon}(q)=(1,1)$.
$q$ is the unique extension of $p$ to $(S_{n}(\mathfrak{C}))_{\varepsilon}^{\prime}$,
so $M(R,d)_{\varepsilon}(p)=(1,1)$ as well.

$(\Leftarrow)$. Since $p$ has ordinal Morley ranks it has global
non-forking extensions. Since $Md_{\varepsilon}(p)=1$ for all sufficiently
small $\varepsilon>0$, it has a unique global non-forking extension.
Let $q$ be its unique global non-forking extension. For each sufficiently
small $\varepsilon>0$, we have that $p$ is contained in a relatively open subset
of $(S_{n}(\mathfrak{C}))_{\varepsilon}^{\prime}$ of diameter $\leq2\varepsilon$.
For any $0<\delta<\varepsilon$, we have that $(S_{n}(\mathfrak{C}))_{\delta}^{\prime}\supseteq(S_{n}(\mathfrak{C}))_{\varepsilon}^{\prime}$,
therefore $q$ has open neighborhoods of arbitrarily small diameter
in $(S_{n}(\mathfrak{C}))_{\varepsilon}^{\prime}=S_{n}(\mathfrak{C})\smallsetminus\mathfrak{C}^{<\varepsilon}$.
And thus $q$ is relatively $d$-atomic in $S_{n}(\mathfrak{C})\smallsetminus\mathfrak{C}^{<\varepsilon}$, and $q$ is a pre-minimal global extension of $p$.

We still need to show that $q$ is the unique pre-minimal global extension
of $p$. Assume that $p$ has another pre-minimal global extension $r$.
By the $(\Rightarrow)$ direction we would have that $MR_{\varepsilon}(r)=1$
for all sufficiently small $\varepsilon>0$, but that would contradict
that $Md_{\varepsilon}(p)=1$ for all sufficiently small $\varepsilon>0$.
\end{proof}
There is one additional hiccup in the development of strongly minimal
sets in continuous logic. Although we typically think of strongly minimal types as coming from strongly minimal sets, it is not hard to show that the following is true in discrete logic:
\begin{itemize}
\item[] If $p\in S_{1}(A)$ is a type whose unique non-algebraic extension to some parameter set $B \supseteq A$ is contained in a strongly minimal set $D$, definable over $B$, then there is a strongly minimal set $E$ definable over $A$ containing $p$ as its unique non-algebraic type.\editcom{Changed $D$ to $E$.}
\end{itemize}
This pleasant fact fails in continuous logic.
\begin{ex}\label{ex:approx-min}
An $\omega$-stable theory with a strongly
minimal type over $\varnothing$ but no $\varnothing$-definable strongly minimal set.
\end{ex}
\begin{proof}[Verification]
For each $i$ let $S_{i}$ be the sphere in an infinite dimensional
Hilbert space with radius $2^{-i}$ (i.e.\ the\editcom{Added `the.'} set of vectors of norm $2^{-i}$). Let $\mathfrak{M}$ be a metric
space whose universe is $\bigsqcup_{i<\omega}S_{i}$, and let the distance
between any points in distinct $S_{i}$ be $1$. Let $T=\mathrm{Th}(\mathfrak{M})$.

The type space $S_{1}(\varnothing)$ is topologically homeomorphic
to $\omega+1$ with the order topology. The unique non-isolated type
is strongly minimal. If we let $\{a_{i}\}_{i<\omega}$ be a sequence
of points such that $a_{i}\in S_{i}\subseteq\mathfrak{M}$, then the
set $\{a_{i}\}_{i<\omega}\cup \{p\}$ is definable, by essentially the same
argument as in Example \ref{ex:not-comp-nor-co-comp}.
\end{proof}

Assuming we are working over models, the relationship between strongly minimal types and strongly minimal sets is precisely as it is in discrete logic. Curiously, this arguably relies on the same kind of behavior that gives us the mild pathology of Example \ref{ex:not-comp-nor-co-comp}.

\begin{prop}
\label{prop:find-minimal}
If $p\in S_1(\mathfrak{M})$ is a strongly minimal (resp.\ pre-minimal) type with $\mathfrak{M}$ a model, then there is an $\mathfrak{M}$-definable strongly minimal (resp.\ minimal) set $D$ containing $p$ (with no assumptions on $S_1(\mathfrak{M})$ or $T$).
\end{prop}
\begin{proof}
For each $k<\omega$, let $F_k = S_1(\mathfrak{M}) \smallsetminus \mathfrak{M}^{<2^{-k}}$. Let $\ell<\omega$ be large enough that $p\in F_k$ for any $k\geq \ell$. By Lemma \ref{lem:strong-ext} part \emph{(i)}, we can find a closed set $A\subseteq S_1(\mathfrak{M})$ and a formula $\varphi:S_1(\mathfrak{M})\rightarrow [0,1]$ such that
\begin{itemize}
\item  $p\in \cset{\varphi\brackconv} \subseteq A$,
\item  $p \in \cset{\varphi\brackconv} \subseteq \tint A^{<\varepsilon}$ for every $\varepsilon > 0$, and
\item  $\cset{\varphi \leq 2^{-k-1}}\cap A \cap F_k = \{p\}$ for every $\ell \leq k<\omega$.
\end{itemize}
 Note that while $\cset{\varphi\brackconv}$ may contain types other than $p$, if $q\in \cset{\varphi\brackconv} \smallsetminus \{p\}$, then $q \notin F_k$ for every $\ell \leq k < \omega$, so in particular $q$ is algebraic and realized in $\mathfrak{M}$.

For each $k$ with $\ell \leq k < \omega$, let $G_k = A \cap \cset{2^{-k-2}\leq \varphi \leq 2^{-k-1}}$. Note that $G_k \cap F_k = \varnothing$, since at most it could contain $\{p\}$, but it cannot contain $p$. This implies that $G_k \subseteq \mathfrak{M}^{<2^{-k}} = \bigcup_{a\in \mathfrak{M}}B_{<2^{-k}}(a)$, so by compactness there is a finite set $M_k \subseteq \mathfrak{M}$ such that $G_k \subseteq M_k^{<2^{-k}}$ and by restricting to a subset if necessary we may assume that $M_k \subseteq G_k^{<2^{-k}}$ as well, i.e.\ $d_H(G_k, M_k) < 2^{-k}$. Note that as a finite union of elements of $\mathfrak{M}$, each $M_k$ is an $\mathfrak{M}$-definable set.

Let $D = \cset{\varphi\brackconv} \cup \bigcup_{\ell \leq k < \omega} M_k$. To see that $D$ is closed, note that $\cset{\varphi\brackconv}$ is clearly closed so we only need to argue that the accumulation points of $\bigcup_{\ell \leq k < \omega} M_k$ are all contained in $D$. Let $q$ be an accumulation point of $\bigcup_{\ell \leq k < \omega} M_k$  and assume that $\varphi(q)>0$. Find $U$, an open neighborhood of $q$, such that $\overline{U}$ is disjoint from $\cset{\varphi\brackconv}$. Find $\varepsilon > 0$ small enough that $\overline{U}^{\leq \varepsilon}$ is disjoint from $\cset{\varphi\brackconv}$.  By compactness there is a $k<\omega$ large enough that $(A\cap \cset{\varphi\leq 2^{-k}})^{\leq 2^{-k}}$ is disjoint from $\overline{U}^{\leq \varepsilon}$. This implies that $D \cap \overline{U}^{\leq \varepsilon}$ is a finite set, so $q$ must be in $D$. If $\varphi(q)=0$ then $q\in \cset{\varphi\brackconv} \subseteq D$, therefore $D$ is closed.

To see that $D$ is definable, pick $\varepsilon>0$ and note that $D^{< \varepsilon / 2} \supseteq \cset{\varphi\brackconv} \cup \bigcup_{m \leq k < \omega} G_k$ for some $m<\omega$ (since $d_H(G_m, M_m) < 2^{-m}$ for every $m < \omega$), so we have
\[ D^{< \varepsilon} \supseteq \left( \cset{\varphi\brackconv} \cup \bigcup_{m \leq k < \omega} G_k  \right) ^{< \varepsilon / 2} \cup \bigcup_{\ell \leq k < \omega} M_k^{<\varepsilon}.\]
So we just have to argue that $\cset{\varphi\brackconv} \subseteq \tint  \left( \cset{\varphi\brackconv} \cup \bigcup_{m \leq k < \omega} G_k  \right) ^{< \varepsilon / 2}$. Find $\delta >0$ small enough that $\delta < \frac{\varepsilon}{2}$ and $\cset{\varphi\leq \delta}^{\leq \delta} \cap A \subseteq  \cset{\varphi\brackconv} \cup \bigcup_{m \leq k < \omega} G_k$, which must be possible by compactness. Let $q \in \cset{\varphi<\delta} \cap \tint  A^{< \delta}$. We have that $\dinf(q,A) < \delta$, so let $r\in A $ such that $d(q,r) < \delta$. This implies that $r \in \cset{\varphi\leq \delta}^{\leq \delta}$, so in particular $r \in \cset{\varphi\brackconv} \cup \bigcup_{m \leq k < \omega} G_k$. Therefore 
\begin{align*}
  \cset{\varphi\leq \delta} \cap \tint A^{<\delta} &\subseteq \tint  \left(\cset{\varphi\brackconv} \cup \bigcup_{m \leq k < \omega} G_k \right)^{<\delta} \\
  &\subseteq \tint  \left(\cset{\varphi\brackconv} \cup \bigcup_{m \leq k < \omega} G_k \right)^{<\varepsilon / 2}, 
\end{align*}

as required. So $D \subseteq \tint D^{<\varepsilon}$ for every $\varepsilon > 0$, and thus $D$ is definable. By construction every type $q \in D \smallsetminus \{p\}$ is algebraic (and realized in $\mathfrak{M}$), so $D$ is a minimal set. If the unique non-algebraic type in $D$ is strongly minimal then $D$ is a strongly minimal set as well, by Proposition \ref{prop:some-min-stuff} (the proof of that proposition does not rely on this proposition).
\end{proof}

Really if $A$ is any parameter set, $p\in S_1(A)$ is a strongly minimal type, $D$ is any definable set containing $p$, and $\mathfrak{M} \supseteq A$ is a model, then we get an $A\cup D(\mathfrak{M})$-definable strongly minimal set containing $p$. And we also have the same for any open or open-in-definable set containing $p$.

\subsubsection{Approximately (Strongly) Minimal Pairs}\label{sec:appr-strongly-minim}
We can recover a fact analogous to Proposition~\ref{prop:find-minimal} without assuming anything about the parameter set, but we
need a dictionaric theory and a slight weakening of the notion of strongly minimal set. 

\begin{defn}\label{defn:approx-stronk-min}
\leavevmode
\begin{enumerate}[label=(\roman*)]
\item A pair $(D,\varphi)$ of a definable set $D\subseteq S_{n}(A)$ and
an $A$-formula $\varphi$ is \emph{approximately minimal} if the zeroset
$\cset{\varphi\brackconv}\subseteq D$ is non-algebraic and for every pair $F,G\subseteq D$
of disjoint $A$-zerosets, every model $\mathfrak{M}\supseteq A$,
and every $\varepsilon>0$, at least one of $F\cap\cset{\varphi\leq \varepsilon}$ and $G\cap\cset{\varphi\leq \varepsilon}$ is $(<\varepsilon)$-algebraic.

\item A pair $(D,\varphi)$ is \emph{approximately strongly minimal} if it is approximately
minimal over every set of parameters over which it is definable. 

\item A definable set $D$ is approximately (strongly) minimal if
there is some formula $\varphi$ such that $(D,\varphi)$ is approximately (strongly)
minimal.
\end{enumerate}
\end{defn}

Obviously a (strongly) minimal set is approximately (strongly) minimal
if we just let $\varphi$ be the distance predicate of the set. Note that
for any kind of minimality it is sufficient to check sets of the form $\cset{\varphi \leq r}$ with $\varphi$ restricted and $r$ rational.\editcom{Changed last bit to be more precise.} 

\begin{prop}
\label{prop:some-min-stuff}
\leavevmode
\begin{enumerate}[label=(\roman*)]
\item If $(D,\varphi)$ is approximately minimal (over the set $A$), then
there is a unique non-algebraic $A$-type $p\in\cset{\varphi\brackconv}\subseteq D$. We
say that $p$ is the \emph{generic type} of $(D,\varphi)$. 

\item If $(D,\varphi)$ is approximately strongly minimal (resp.\ approximately minimal), then its generic type is strongly minimal (resp.\ pre-minimal).

\item If $(D,\varphi)$ is (approximately) minimal and its generic type is strongly minimal, then $(D,\varphi)$ is (approximately) strongly minimal (where a pair $(D,\varphi)$ is \emph{strongly minimal} if $D$ is strongly minimal).
\end{enumerate}
\end{prop}

\begin{proof}
  \emph{(i).} Assume that there are two distinct non-algebraic types, $q_0$ and $q_1$  contained in $\cset{\varphi = 0}\subseteq S_n(A)$. 
Let $F_{0},F_{1}\subseteq\cset{\varphi=0}$ be disjoint $A$-zerosets such that
$p_{i}\in F_{i}$ for both $i<2$. Since for each $\varepsilon>0$,
at least one of $F_0$ and $F_1$ must be $(<\varepsilon)$-algebraic, it must be that for some $i<2$, $F_i(\mathfrak{C})\subseteq \mathfrak{C}^{<\varepsilon}$ for arbitrarily small $\varepsilon>0$, so in particular $F_i(\mathfrak{C})\subseteq \mathfrak{C}$, implying that $F_i$ is algebraic.
This is a contradiction, therefore there cannot be two non-algebraic types in $\cset{\varphi=0}\subseteq S_{n}(A)$,
but $\cset{\varphi=0}$ is non-algebraic so there must be at least one.

\emph{(ii).} Let $X = A$ if $(D,\varphi)$ is approximately minimal over $A$, and let $X=\mathfrak{C}$ if $(D,\varphi)$ is approximately strongly minimal. We need to show that $p$ is $d$-atomic in the set of non-$(<\varepsilon)$-algebraic types in $S_n(X)$. For any $\varepsilon>0$, let $F_\varepsilon$ denote the (closed) set of non-$(<\varepsilon)$-algebraic types in $S_n(X)$. Fix $\varepsilon>0$ small enough that  $p \in F_\varepsilon$. Find $\delta > 0$ small enough that $\delta < \frac{1}{2} \varepsilon$ and $p\notin (D \cap \cset{\varphi\geq \frac{1}{2}\varepsilon})^{\leq \delta}$, and then consider $U = \left( D^{<\delta}\smallsetminus  (D\cap\cset{\varphi\geq \frac{1}{2}\varepsilon}) ^{\leq \delta} \right) \cap  F_\varepsilon$.  Note that $p\in U$ and that $U$ is relatively open in $F_\varepsilon$. We want to show that $U \subseteq B_{\leq \delta}(p)$. Let $q\in U$. By construction, this implies that $d(q,D) < \delta$ so there is some $r\in D$ such that $d(q,r) < \delta$. Also\editcom{Added `also.'} by construction, $r$ cannot be in $D \cap \cset{\varphi\geq \frac{1}{2}\varepsilon}$. Assume that $r\in D \cap \cset{0<\varphi<\frac{1}{2}\varepsilon}$. There must be some $0<\gamma < \sigma < \frac{1}{2}\varepsilon$ such that $r\in D \cap \cset{\gamma \leq \varphi \leq \sigma}$, but $D\cap\cset{\gamma \leq \varphi \leq \sigma}$ is $(<\frac{1}{2}\varepsilon)$-algebraic, so in particular $r \notin F_{\varepsilon/2}$. But this is a contradiction since for any model $\mathfrak{M} \supseteq X$ (i.e.\ $\mathfrak{M}\supseteq A$ or $\mathfrak{M}=\mathfrak{C}$) and extension $q^\prime$ to $S_n(\mathfrak{M})$, there is an extension $r^\prime$ to $S_n(\mathfrak{M})$ such that $d(q^\prime,\mathfrak{M})\leq d(q^\prime, r^\prime) + d(r^\prime,\mathfrak{M}) < \delta + \frac{1}{2}\varepsilon < \varepsilon$. But $q \in F_\varepsilon$, so $q$ has an extension $q^{\prime\prime}$ with $d(q^{\prime\prime},\mathfrak{M})\geq \varepsilon$.

Therefore $r$ must be in $\cset{\varphi=0}$. Assume that $r\neq p$. This implies that $r$ is algebraic so that in particular $r\in \mathfrak{M}$ for any $\mathfrak{M}\supseteq X$, which is again a contradiction since for any extension $q^\prime$ of $q$ there is an extension $r^\prime$ of $r$ such that $d(q^\prime,\mathfrak{M})\leq d(q^\prime,r^\prime) + d(r^\prime,\mathfrak{M}) < \delta + 0 < \frac{1}{2}\varepsilon < \varepsilon$. Hence\editcom{Changed to `Hence.'} it must be the case that $r=p$. Since this is true for any $q\in U$, this implies that $U \subseteq B_{<\delta}(p) \subseteq B_{\leq \delta}(p)$. Since we can do this for any sufficiently small $\delta > 0$, we have that $p$ is relatively $d$-atomic in $F_\varepsilon$, and the same is true for any sufficiently small $\varepsilon>0$.

\emph{(iii).} If $D$ is approximately minimal but not approximately strongly
minimal then its generic type has two distinct non-algebraic
extensions to some set of parameters, so it is not strongly minimal.
The only thing to prove is that if $D$ is minimal and its generic type is strongly minimal then $D$ is strongly minimal (and not
just approximately strongly minimal). This follows from the fact that if $p$
is the global strongly minimal type in $D\subseteq S_{n}(\mathfrak{C})$,
then any $q\in D\smallsetminus\{p\}$ must be an extension of some type
in $D\subseteq S_{n}(A)$. If it is an extension of $p\upharpoonright A$,
then it is algebraic and if it is an extension of some $r\neq p\upharpoonright A$
then it also must be algebraic. 
\end{proof}
It is easy to come up with an example of a definable set $D$ such
that $(D,\varphi)$ and $(D,\psi)$ are approximately strongly minimal but
have different generic types. The union of two disjoint 
strongly minimal sets for instance.\editcom{Changed `works' to `for instance.'}
\begin{cor}\label{cor:min-set-type-correspondence}
\leavevmode
\begin{enumerate}[label=(\roman*)]
\item If $D\subseteq S_{n}(A)$ is minimal, then there is a unique non-algebraic
type $p\in D$.

\item If $D$ is strongly minimal, then there is a unique non-algebraic
type $p\in D\subseteq S_{n}(A)$ for any set of parameters $A$ over
which $D$ is definable.
\end{enumerate}
\end{cor}

The strongly minimal type whose existence is guaranteed by Corollary~\ref{cor:min-set-type-correspondence} is also referred to as the \emph{generic type} of the corresponding set.

Finally we come to the advantage we gain by passing to this weaker notion, as promised at the beginning of this section. 
Given a strongly minimal type in a dictionaric theory, we can always find an approximately strongly minimal pair pointing to that type and\editcom{Added `and.'} definable over the same set that the type is over, and likewise with pre-minimal types and minimal sets.

\begin{prop}\label{prop:strong_min_and_cat_in_cont:2}
\leavevmode
\begin{enumerate}[label=(\roman*)]
\item If $S_{n}(A)$ is dictionaric and $p\in S_{n}(A)$ is a pre-minimal type,
then there is an $A$-definable approximately minimal pair $(D,P)$
pointing to $p$.


\item If $D\subseteq S_{n}(A)$ is approximately minimal over a model $\mathfrak{M}$,
then there is a $(D(\mathfrak{M})\cup A)$-definable minimal set $E\subseteq D$
pointing to the same type.

\item If $D\subseteq S_{n}(A)$ is approximately strongly minimal (as part of the pair $(D,P)$),
then for any model $\mathfrak{M}\supseteq A$, there is a $(D(\mathfrak{M})\cup A)$-definable
strongly minimal set $E\subseteq D$ pointing to the same type. Furthermore, for any model $\mathfrak{N} \succ \mathfrak{M}$, $E(\mathfrak{N}) \smallsetminus \cset{P(\mathfrak{N})\brackconv} \subseteq D(\mathfrak{M})$.
\end{enumerate}
\end{prop}

\begin{proof}
\emph{(i).} This follows from applying Lemma~\ref{lem:strong-ext} to the closed set $\{p\}$ which is relatively definable in the set $F_i \subseteq S_{n}(A)$, where $F_i$ is the set of all non-$(<2^{-i-k})$-algebraic types and $k>0$ is chosen so that $p \in F_i$ for every $i<\omega$. The lemma gives us a definable set $D$ and a formula $\varphi$ such that for each $i<\omega$, $D\cap F_i \cap \cset{\varphi\leq 2^{-i}} = {p}$. To see that $(D,\varphi)$ is an approximately minimal pair, pick $\varepsilon>0$. Find $i<\omega$ such that $2^{-i-1}< \varepsilon \leq 2^{-i}$ and let $G,H \subseteq D$ be disjoint zerosets. At most one of $G$ or $H$ can contain $p$, so assume without loss of generality that $p \notin G$. We have that $G \cap F_i \cap \cset{\varphi \leq 2^{-i}} \subseteq {p}$, so $G \cap F_i \cap \cset{\varphi \leq \varepsilon} \subseteq {p}$ as well, but $p\notin G$, so $G \cap F_i \cap \cset{\varphi \leq \varepsilon}$. This implies that $G \cap \cset{\varphi \leq \varepsilon}$ is contained in the set of $(<2^{-i-k})$-algebraic types so in particular since $2^{-i-k}\leq 2^{-i-1}< \varepsilon$, we have that every type in $G\cap\cset{\varphi\leq \varepsilon}$ is $(<\varepsilon)$-algebraic. Therefore $(D,\varphi)$ is an approximately minimal pair.

\emph{(ii).} This follows from the comment after Proposition \ref{prop:find-minimal}.

\emph{(iii).} Most of this follows from part \emph{(ii)} and Proposition \ref{prop:some-min-stuff}. The only thing we need to verify is the last sentence, which follows from the fact that anything in $D(\mathfrak{N})$ not realizing the strongly minimal type over $\mathfrak{M}$ must be algebraic over $\mathfrak{M}$ in the first place and so realized in it. The strongly minimal type is contained in $\cset{\varphi\brackconv}$, so the result follows.
\end{proof}

Note that parts \emph{(ii)} and \emph{(iii)} do not require any assumptions about the type space. The following example shows\editcom{Added `s.'} that the dictionaricness stipulation in part \emph{(i)} cannot be removed.
\begin{ex}


A non-dictionaric superstable theory with a strongly minimal type over $\varnothing$ but no 
 approximately strongly minimal sets over $\varnothing$.
\end{ex}

\begin{proof} [Verification]
Let $\mathcal{L}=\{P_{0},P_{1}\}$ be a language with two unary $1$-Lipschitz
$[0,1]$-valued predicates and let $\mathfrak{M}$ be an $\mathcal{L}$-structure
whose universe is $\omega\times[0,1]$ and whose metric is given
by $d((n,x),(m,y))=1$ if $n\neq m$ and $d((n,x),(n,y))=2^{-n}+d(x,y)$
if $x\neq y$. Let $P_{0}((n,x))=2^{-n}$ and $P_{1}((n,x))=x$. Finally
let $T=\mathrm{Th}(\mathfrak{M})$.

The type space $S_{1}(\varnothing)$ is homeomorphic to $(\omega+1)\times[0,1]$.\editcom{Changed `topologically' to `homeomorphic to.'}


Note that if a definable set has non-empty intersection with one of the sets of types of the form $\{n\}\times [0,1]$ for $n<\omega$, then it must contain all of it, because this set is metrically isolated from the rest of the type space and is topologically connected but uniformly metrically discrete. So to show that none of the types in $\{\omega \}\times [0,1]$ are pointed to by an approximately strongly minimal set, all we need to do is show that if a definable set contains one such type then it must contain some type in $\{n\}\times [0,1]$ for some $n<\omega$. This follows immediately because if $p \in \{\omega \}\times [0,1]$, then it is the limit of types in $\{n\}\times [0,1]$ for $n < \omega$ that are uniformly metrically separated. So if $F$ is a closed set whose intersection with $\{n\}\times [0,1]$ for $n<\omega$ is empty and whose intersection with $\{\omega\}\times [0,1]$ is precisely ${p}$, then $p\notin \tint  F^{<\varepsilon}$ for any $0 < \varepsilon < 1$ and so $F$ is not definable.\editcom{Added `so.'}

Thus if $D$ is a definable set containing some $p\in \{\omega \} \times [0,1]$, then $D$ must contain all of $\{ n \} \times [0,1]$ for sufficiently large $n<\omega$. And so since $D$ is closed it must contain all of $\{ \omega \} \times [0,1]$.  If we let $\varphi$ be a formula such that $\varphi(p)=0$, then for any $\varepsilon > 0$, $D\cap \cset{\varphi\leq \varepsilon}$ contains some $q\in \{\omega \} \times [0,1]$ with $q\neq p$.
\end{proof}
 

In all of the examples of approximately minimal sets we know of there 
is an obvious\editcom{Removed quotes.} pseudo-metric $\rho$ on $D$ such that $D/\rho$ is
a minimal set whose generic type corresponds exactly to the generic type in the original set. For instance, in Example~\ref{ex:approx-min} there is a definable pseudo-metric $\rho$ such that $\rho(x,y)=0 $ if and only if either $x$ and $y$ are both in the same Hilbert space sphere or they are not in a Hilbert space sphere and $x=y$. It is not clear if this is always possible. In the case of discrete strongly minimal types, however, it is so.\editcom{Fixed comma splice, moved `however,' and added `so.'} 
\begin{defn}
  A type $p$ is \emph{discrete} if there is an $\varepsilon>0$ such that if $a,b\models p$ and $d(a,b)<\varepsilon$, then $a=b$.
\end{defn}
\begin{prop}
If $(D,\varphi)$ is an approximately minimal pair over some parameter set $A$ pointing to a discrete pre-minimal type $p$ in a dictionaric type space, then there is an $A$-definable set $E\subseteq D$ containing $p$ and an $A$-definable equivalence relation $\rho$ on $E$ such that $E/\rho$ is minimal and such that the quotient map is a bijection when restricted to the pre-minimal type. (In particular this means that if $p$ is strongly minimal, then the corresponding type in $E/\rho$ is strongly minimal as well.)
\end{prop}

\begin{proof}

Let $\varepsilon>0$ be such that if $a,b\models p$ and $d(a,b) < \varepsilon$ then $d(a,b)=0$.

By compactness there must be a $\delta>0$ such that for any $a,b \in D(\mathfrak{C})\cap \cset{\varphi(\mathfrak{C})\leq \delta}$, if $d(a,b) < \frac{3}{4}\varepsilon$, then $d(a,b) < \frac{1}{4} \varepsilon$. This implies that in the zeroset $D\cap \cset{\varphi \leq \delta}$, the formula $\rho(x,y)=\imin{\frac{4}{3\varepsilon}(d(x,y)\dotdiv \frac{1}{4}\varepsilon)}{1}$  is a $\{0,1\}$-valued equivalence relation which is equality on the set of realizations of $p$. Let $E$ be a definable subset of $D$ such that $\cset{\varphi\brackconv}\subseteq E \subseteq \cset{\varphi<\delta}$. Clearly $\rho$ is the required equivalence relation on $E$.
\end{proof}

The leaves the question in general.\editcom{Added sentence.}

\begin{quest}
If $D\subseteq S_{n}(A)$ is approximately minimal with generic type $p$ then, does there always exist an $A$-definable pseudo-metric
$\rho$ on $D$ such that $D/\rho$ is minimal with generic type $q$ and the quotient map $D\rightarrow D/\rho$ restricted to the
set of realizations of $p$ is a bijection with the set of realizations
of $q$?
\end{quest}

\subsubsection{Strongly Minimal Theories That Do Not Interpret Infinite Discrete Theories}

Obviously if $\mathfrak{M}$ is a discrete strongly minimal structure
and $X$ is some compact metric space with a transitive automorphism
group, then $\mathfrak{M}\times X$ is strongly minimal, but this
example is in some sense trivial, as $\mathfrak{M}$ and
$\mathfrak{M}\times X$ are bi-interpretable.\editcom{Took words out of parenthetical comment.} In \cite{Noquez2017}, Noquez raised the question of whether or not there are any non-trivially\editcom{Removed quotes.} continuous examples of strongly minimal theories. 
 Here is a minimally
non-trivial continuous example, giving a positive answer:

\begin{ex}
\label{exa:min-non-triv} Consider $\mathbb{R}$ with the metric $d(x,y)=\frac{|x-y|}{1+|x-y|}$.
$\mathrm{Th}(\mathbb{R},+)$ is strongly minimal but does not interpret
any infinite discrete structure.
\end{ex}

   We will verify these statements in Theorems~\ref{thm:strong_min_and_cat_in_cont:2} and \ref{thm:group_char}.

   One can similarly show\editcom{Since the proof moved, the beginning of this sentence didn't make sense.} that the same is true of $(\mathbb{R}^n,+)$ for any $n<\omega$. The underlying additive group of a $p$-adic
field, $\mathbb{Q}_{p}$, with the appropriate metric, is likewise strongly
minimal, although it does have a discrete strongly minimal imaginary.


It turns out that, in the context of strongly minimal theories, this condition of being unable to interpret an infinite discrete structure has a very tight topological characterization in terms of models and is also very constraining in the special case of strongly minimal groups.

\begin{defn}
  A continuous theory $T$ is \emph{\esscont} if it does not interpret an infinite discrete structure.
\end{defn}



The goal of the rest of this subsection is to prove the following characterization of \esscont\ strongly minimal theories.

\begin{thm}\label{thm:strong_min_and_cat_in_cont:2}
  Let $T$ be a continuous strongly minimal theory. $T$ is \esscont\ if and only if some model of $T$ has a non-compact connected component.
\end{thm}

\begin{proof}[Proof of $\Rightarrow$]
  We will prove the contrapositive. Assume that every model of $T$ has compact connected components. (Although note for later that we only actually use the fact that connected components of generic elements are compact.) For any $\e >0 $, we will let $E_\e$ be the equivalence relation on models of $T$ induced by $a E_\e b$ if and only if there exists a sequence $c_0,\dots,c_n$ with $c_0=a$, $c_n=b$ and $d(c_i,c_{i+1})<\e $ for every $i<n$. We want to show the following.
  \begin{itemize}
  \item[$(\ast)$] There is a $\xi > 0$ such that the equivalence relation $E_\xi$ is definable (by a $\{0,1\}$-valued formula) and has infinitely many equivalence classes in any model of $T$.
  \end{itemize}

  From which it is immediate that $T$ is not \esscont\, as the imaginary quotient by this equivalence relation is infinite and discrete.




  Let $\frk{M}$ be some model of $T$ and let $p$ be the generic type over $\frk{M}$.

  \emph{Claim:} There is some $\e > 0$ such that for any $a\in \frk{C}$ with $a \models p$, the $E_\e$-equivalence class of $a$ is compact.

  \emph{Proof of claim.} Fix $a \in \frk{C}$ with $a\models p$. Let $F$ be the connected component of $a$. By assumption $F$ is compact, and by Lemma~\ref{lem:uniform-local-compactness}, there is some $\delta >0$ such that every closed $\delta$-ball in any model of this theory is compact. Let $b_0,\dots,b_{n-1}$ be elements of $F$ chosen so that $U = \bigcup_{i<n}B_{<\delta}(b_i)$ covers $F$. Since $\cl U\subseteq \bigcup_{i<n}B_{\leq\delta }(b_i)$ is compact and $F$ is a connected component of it (and in particular is closed), we have that there is some clopen-in-$\cl U$ set $Q \subseteq U$ such that $F \subseteq Q$. Since $Q$ is relatively closed in $\cl U$, it is closed, and since $Q$ is relatively open in $U$, it is open, so it is actually clopen.  Moreover, there is some $\e > 0$ such that $Q^{\leq \e}$ is disjoint from the complement of $Q$, since $\cl U$ is compact. Clearly we have that the $E_\e$-equivalence class of $a$ is contained in $Q$. It's not hard to see that an $E_\e $-equivalence class must be closed (for any $\e>0$), so we have that it is compact.\hfill $\qed_{\text{claim}}$ 


  Fix $\e > 0$ as in the claim. Now note that for any $b \models p$, the $E_\e$-equivalence class of $b$ must be isometric to the $E_\e$-equivalence class of $a$, since there is an automorphism of the monster model taking $b$ to $a$. By the compactness of the equivalence class, there exists a natural number $n$ such that if $aE_\e c$, then there exists a sequence $e_0,\dots,e_{n}$ such that $a=e_0$, $c=e_{n}$, and $d(e_i,e_{i+1})<\e$ for each $i<n$ (even if we only need a shorter chain, we can just use repetitions). This fact is uniformly true for any $b \models p$. By compactness, there is a $\xi >0$ with $\xi < \e$ such that this statement is still true and such that for realizations of $p$, the $E_\xi$-equivalence classes are the same as the $E_\e$-equivalence classes. By construction, we have that if $aE_\xi b$ and $\neg a E_\xi c$, then $d(b,c) \geq \e$. 

  Now consider the formulas
  \begin{align*}
    \lambda(x,y) &:=  \inf_{z_1,\dots,z_{n }} \max_{i\leq n}d(z_i,z_{i+1})\dotdiv \e, \\ 
\eta(x,y) &\coloneqq \sup_{z_1,\dots,z_{n}} \min_{i\leq n}\xi\dotdiv d(z_i,z_{i+1}),\text{ and} \\
    \chi(x) &\coloneqq \inf_y \max\{\lambda(x,y),\eta(x,y)\},
    \end{align*}
    where $z_0$ is understood to mean $x$ and $z_{n+1}$ is understood to mean $y$ (note the $+1$, but also note that the $+1$ is not actually essential, rather it just makes the formulas $\lambda$ and $\eta$ easier to write).

    To unpack what these mean, given $a$ and $b$, $\lambda(a,b) > 0$ implies that for any sequence $c_0,...,c_{n+1}$ with $c_0 = a$ and $c_{n+1}=b$, for some $i\leq n$, we have $d(c_i,c_{i+1}) > \e$ (although note that $\lambda(a,b) > 0$ is actually slightly stronger than this statement in insufficiently saturated models).  On the other hand, $\eta(a,b) > 0$ means that there is a chain $c_0,\dots,c_{n+1}$ with $c_0 =a$, $c_{n+1}=b$, and $d(c_i,c_{i+1})< \xi$ for every $i\leq n$. So, putting it together, $\chi(a) > 0$ implies that for every $b$, one of these two conditions holds. By the above statements regarding any $a$ satisfying $p$, we have that $\chi(a) > 0$.

    So we have that $\cset{\chi(x) > 0}$ is an open neighborhood of $p$ in $S_1(\frk{M})$, so in particular, $\cset{\chi(x) = 0}$ is compact (note, also, that $\chi$ always takes on non-negative values). 


    \emph{Claim:} There is a natural number $k$, such that for any $a$ and $b$, if $aE_\xi b$, then this is witnessed by a chain of length no more than $k$.

    \emph{Proof of claim.} Let $X$ be a finite subset of $\cset{\chi(\frk{C})=0}$ such that $\cset{\chi(\frk{C})=0}\subseteq X^{<\xi}$. For any $a \in \cset{\chi(\frk{C})=0}$, let $F_a\subseteq \cset{\chi(\frk{C})=0}$ be the $E_\xi$-equivalence class of $a$ computed inside $\cset{\chi(\frk{C})=0}$, i.e.\ the set of all elements $b$ of $\cset{\chi(\frk{C}=0)}$ satisfying $aE_\xi b$ with a witnessing chain contained entirely in $\cset{\chi(\frk{C})=0}$. Note that by construction, $\cset{\chi(\frk{C})=0}=\bigcup_{x\in X}F_x$. Also note that each $F_a$ is closed, and therefore compact (as a set of elements of $\frk{C}$).

    For each $x \in X$, we have that
    \[
      F_x \subseteq B_{<\xi}(x)\cup (B_{<\xi}(x))^{<\xi}\cup ((B_{<\xi}(x))^{<\xi})^{<\xi}\cup \dots.
    \]
    By the compactness of $F_x$, some finite initial segment of this is sufficient to cover $\cset{\chi(\frk{C})=0}$. Let $\ell$ be the longest length of one of these initial segments. Let $k = \ell + 1 + n$.

    To see that this $k$ is sufficient, we need to consider three cases.
    \begin{enumerate}
    \item $\chi(a) > 0$. In this case, by our previous discussion of $\chi(x)$, we know that there is a chain of length no more than $n+1$ witnessing $aE_\xi b$, which is no more than $k$.
    \item $\chi(a) = 0$ and the chain witnessing $aE_\xi b$ is entirely contained in $\cset{\chi(\frk{C})=0}$. In this case, we have that $b \in F_a$. There is some $x \in X$ with $d(a,x)< \xi$, so by the definition of $\ell$ we get a chain of length $1+\ell$ witnessing that $aE_\xi b$.
    \item $\chi(a)=0$ and the chain witnessing $aE_\xi b$ is not entirely contained in $\cset{\chi(\frk{C})=0}$. In this case, let $c_0,\dots,c_{m}$, with $c_0 = a$ and $c_m = b$, be the chain witnessing that $a E_\xi b$. Let $c_i$ be the first such that $\chi(c_i)=0$ but $\chi(c_{i+1})> 0$. By construction, we can find a chain of length no more than $\ell$ witnessing that $aE_\xi c_i$ and a chain of length no more than $n$ witnessing that $c_{i+1} E_\xi b$. Concatenating these chains gives a witnessing chain of length no more than $\ell+1+n$, which is $k$.\hfill $\qed_{\text{claim}}$
    \end{enumerate}

    Now, since this is true for all $a$ and $b$, by compactness, there must actually be some $\delta > 0$ with $\delta < \xi$ such that for any $a$ and $b$, with $aE_\xi b$, there is a chain $c_0,\dots,c_k$, with $c_0=a$, $c_k=b$, and $d(c_i,c_{i+1})<\delta$ for all $i<k$.   We have just established that for any $a$ and $b$, if $aE_\xi b$, then $\eta(a,b) > \xi - \delta$, but if $\neg a E_\xi b$, then $\eta(x,y) = 0$. So consider the formula
    \[
      E(x,y)\coloneqq\min\left\{\frac{(\xi-\delta)\dotdiv \eta(x,y)}{\xi-\delta},1\right\}.
      \]
      We now have that this $E(x,y)$ is $\{0,1\}$-valued and has $E(x,y)=0$ if and only if $xE_\xi y$, as required. To see that this equivalence relation has infinitely many equivalence classes in every model, note that if it had finitely many equivalence classes then that model would itself be compact, since its equivalence classes are compact, which is a contradiction.
\end{proof}
It is not hard to show that if some model of $T$ has non-compact connected components, then every non-prime model has non-compact connected components. The set $E$ in the example mentioned in the text after Example~\ref{ex:not-comp-nor-co-comp} (i.e.\ the definable subset of $\mathbb{R}$) shows that this can fail for the prime model.

In order to prove the other direction of Theorem~\ref{thm:strong_min_and_cat_in_cont:2} we will need to collect some lemmas.

\begin{lem}\label{lem:strong_min_and_cat_in_cont:1}
If $T$ is a strongly minimal theory, $\frk{M}\models T$, and for any $a$ realizing the generic type over $\frk{M}$, the connected component of $a$ (in the monster model) is compact, then no model of $T$ has a non-compact connected component.
\end{lem}
\begin{proof}
  Since the proof of Theorem~\ref{thm:strong_min_and_cat_in_cont:2} only uses that the connected components of generic elements are compact, we have that there is some $\xi > 0$ such that the equivalence relation $E_\xi$ (as defined in that proof) has compact equivalence classes. For any $a$ the connected component of $a$ is contained in the $E_\xi$-equivalence class of $a$, and so since connected components are always closed, we have that the connected component of $a$ is compact.
\end{proof}

The following is a (weak) analog of a well known fact in discrete logic that given two strongly minimal sets in an uncountably categorical theory, there is a definable finite-to-finite correspondence between them. 

\begin{lem}\label{lem:strong_min_and_cat_in_cont:2}
  Let $T$ be a strongly minimal theory. Suppose that $I$ is a strongly minimal imaginary. There exists a formula $\varphi(x,y)$ (using at most the parameters needed to define $I$), with $x$ a variable in the sort $I$ and $y$ a variable in the home sort, such that for any generic $a \in I$, the zeroset of $\varphi(a,y)$ is algebraic and has each realization generic over $\frk{M}$ and for any generic $b$ in the home sort, the zeroset of $\varphi(x,b)$ is algebraic and has each realization generic over $\frk{M}$.

\end{lem}
\begin{proof}
  Let $A$ be the set of parameters needed to define $I$ and let $\frk{M}$ be a separable model of $T$ containing $A$. Let $c$ be an element of $I(\frk{C})$ that is generic over $\frk{M}$, and let $\frk{N} \succ \frk{M}$ be a model of $T$ that is prime over $\frk{M}b$. 
  Because $\frk{N}$ is a proper elementary extension of $\frk{M}$ we must have that $H(\frk{N})$ is strictly larger than $H(\frk{M})$, where $H$ is the home sort. Let $b$ be some element of $H(\frk{N})\smallsetminus H(\frk{M})$. We necessarily have that $b$ is generic over $\frk{M}$, but $\tp(c/\frk{M}b)$ is $d$-atomic, which by strong minimality means that it must be algebraic. Let $\varphi_0(x,y)$ be a formula such that $\varphi_0(b,y)$ is the distance predicate of $\tp(c/\frk{M}b)$.

  A symmetric argument gives that $\tp(b/\frk{M}c)$ is also algebraic and therefore $d$-atomic, so let $\varphi_1(x,y)$ be a formula such that $\varphi_1(x,c)$ is the distance predicate of $\tp(b/\frk{M}c)$. The formula $\varphi(x,y)=\max\{\varphi_0(x,y),\varphi_1(x,y)\}$ has the required properties.
  %
%
\end{proof}

\begin{rem}\label{rem:char-essent-cont}
  As an aside, one might wonder whether or not two strongly minimal sets in an \insep\ categorical theory always have a definable compact-to-compact correspondence. The answer is yes. In Lemma~\ref{lem:strong_min_and_cat_in_cont:2} (which, of course, generalizes to arbitrary \insep\ categorical theories), the set $\cset{\varphi(x,y)=0}\cap F(x,y)$, where $F(x,y)$ is the partial type consisting of all types for which at least one of the projections is generic, is relatively definable in $F(x,y)$. By Proposition~\ref{prop:ext}, we can find a definable set $D(x,y)$ such that $D(x,y)\cap F(x,y)=\cset{\varphi(x,y)=0}\cap F(x,y)$.  By strong minimality, it is easy to show that $D(x,y)$ is a compact-to-compact relation between the two sets, but it may fail to be either total or surjective. By uniform local compactness, for sufficiently small $\e > 0$, $\cset{D(x,y)\leq \e}$ is still a compact-to-compact relation. By Lemma~\ref{lem:important-sep-lemma}, we may choose $\e$ so that $\cset{D(x,y)\leq \e}$ is definable. Since the interior of $\cset{D(x,y)\leq \e}$ contains the `rank $2$' type (type of a pair of mutually generic elements), we have (where $A$ and $B$ are the strongly minimal sets in question) that the sets $U=\{a\in A: \neg (\exists b \in B) D(a,b) \leq \e\}$ and $V = \{b \in B:\neg (\exists a \in A) D(a,b)\leq \e\}$ are both pre-compact, so we have that $\cset{D(x,y)\leq \e}\cup (\overline{A}\times \overline{B})$ is a definable compact-to-compact correspondence between $A$ and $B$.

  This is still true for any pair of non-orthogonal strongly minimal types in an arbitrary theory, but a minor technicality is that one would need to show that the type spaces corresponding to $A\times B$ are dictionaric.

  The original correspondence is actually slightly nicer than merely definable. It has the property that for any generic $a \in A$, the set $\cset{\varphi(a,y)=0}$ is definable and likewise for any generic $b\in B$ as well as the property that the distance predicates of these definable sets are uniformly definable, i.e.\ there is a formula $\psi(x,y)=0$ such that for any generic $a \in A$, $\psi(a,y)=0$ is the distance predicate of $\cset{\varphi(a,y)=0}$ (with a similar statement for $B$). This additional niceness is not always possible. Consider $\mathbb{R}$ as a metric space and the $\mathbb{R}$-definable subset $E = \{\pm\ln(1+n):n<\omega\}$ (which is very similar to the definable set mentioned after Example~\ref{ex:not-comp-nor-co-comp}). Although the strongly minimal types in the home sort $H$ and $E$ are non-orthogonal in that they are literally the same type, there is no compact-to-compact correspondence between $H$ and $E$ with this additional property. To see this, note that if a formula $\chi(a,y)$ is a distance predicate for any parameter $a$, then the function $a \mapsto \cset{\chi(a,y)=0}$ is continuous with regards to the Hausdorff metric. Assume that  $\varphi(x,y)=0$ is a definable compact-to-compact correspondence between $H$ and $E$, where $x$ is a variable of sort $H$ and $y$ is a variable of sort $E$. Then consider the set $\cset{\varphi(x,0)=0}$. This must contain some $a\in H$. By continuity, we have that $\varphi(a',0)$ must be $0$ for all $a^\prime$, but this implies that $0$ is related to a non-compact set of $a'$'s in $A$, which is a contradiction.

  In this example it can be arranged that $\cset{\varphi(x,b)=0}$ is uniformly definable as a function of $b$, but it is not too difficult to show that if we take $E\cup \mathbb{R}_{\geq 0}$ and $E \cup \mathbb{R}_{\leq 0}$ as our two strongly minimal sets, then neither direction can be taken to be uniformly definable.
\end{rem}

Now we can finish the proof of the characterization.

      \begin{proof}[Proof of Theorem~\ref{thm:strong_min_and_cat_in_cont:2} $\Leftarrow$] 

        Let $\frk{M}$ be the prime model of $T$. Assume that some model $\frk{N} \succeq \frk{M}$ of $T$ has a non-compact connected component $C$. If $C$ does not contain an element that is generic over $\frk{M}$, then by Lemma~\ref{lem:strong_min_and_cat_in_cont:1} there is a model of $T$ that contains a non-compact connected component $C^\prime$ which contains an element that is generic over $\frk{M}$, so we may assume without loss of generality that $C$ contains an element generic over $\frk{M}$. Let $a$ be some such generic element. 

        Now assume that $T$ is not \esscont. Let $I$ be an imaginary sort of $T$ that is infinite and discrete. Since $T$ is \insep\ categorical, there is a strongly minimal set $D \subseteq I$ (possibly over some larger set of parameters). Let $J$ be the imaginary sort corresponding to $D$. Let $\frk{N}\succeq\frk{M}$ be a separable elementary extension containing the parameters necessary to define $D$. By Lemma~\ref{lem:strong_min_and_cat_in_cont:2}, there is an $\frk{M}$-formula $\varphi(x,y)$, with $x$ a variable in $J$ and $y$ a variable in the home sort, such that 
        for any $a$ in $J(\frk{C})$, generic over $\frk{N}$, the set of $b$'s in the home sort for which $\varphi(a,b)=0$ holds is compact and consists only of elements generic over $\frk{N}$.

        
        Fix $b \in H(\mathfrak{C})$ generic, and note that the connected component of $b$ is non-compact. By compactness there must exist an $\e > 0$ such that for any generic $a$, if $\varphi(a,b)> 0$, then $\varphi(a,b)>\e$. So consider the formula $\psi(x,y)\coloneqq \min\{\frac{1}{\e}\varphi(x,y),1\}$. We have that for generic $a$, $\psi(a,b)\in\{0,1\}$. This is part of the type of $b$, so for any generic $b$ this must be true as well. By uniform continuity of $\psi(x,y)$, this implies that there is a $\delta>0$ such that if generic $b,b^\prime\in H(\frk{C})$ satisfy $d(b,b^\prime)<\delta$, then $\psi(a,b)=\psi(a,b^\prime)$, but this implies that for any $b^\prime$ in the connected component of $b$, $\psi(a,b^\prime)=\varphi(a,b^\prime)=0$, which is a contradiction. Therefore no such $I$ can exist, and $T$ is \esscont.
%
%
 %
%
      \end{proof}

      \begin{cor}
        For any strongly minimal theory $T$, the following are equivalent:
        \begin{itemize}
        \item $T$ is \esscont.
        \item Some model of $T$ has a non-compact connected component.
        \item Every generic element of a model of $T$ has a non-compact connected component.
        \item $T$ does not have a $\varnothing$-definable $\{0,1\}$-valued equivalence relation on its home sort with infinitely many equivalence classes.
        \end{itemize}
      \end{cor}

      The example given immediately after Example~\ref{ex:not-comp-nor-co-comp} (i.e.\ $\mathbb{R}$ with the appropriate metric), as well as some of the examples given in Remark~\ref{rem:char-essent-cont}, show that the prime model may fail to have a non-compact connected component in \anesscont\ strongly minimal theory.
      
      
      After all of this, it is natural to wonder about the kinds of (pre-)geometries that can occur in \esscont\ strongly minimal theories. There are a myriad of specific questions in this vein.      It is possible that characterizing \esscont\ strongly minimal theories might be easier than characterizing discrete strongly minimal theories. A negative answer to either of the following two questions would be an indication that \esscont\ strongly minimal sets admit less complexity than discrete strongly minimal sets.
      
      \begin{quest}\label{quest:stronk-min-thys:1}
        Is there \anesscont\ strongly minimal set with the geometry of an algebraically closed field?
        \end{quest}

      \begin{quest}
        Is there \anesscont\ strongly minimal set with non-locally modular, flat pregeometry?
      \end{quest}

      A difficult-to-rigorize follow-up question would be this: Can a Hrushovski construction build \anesscont\ strongly minimal set?
        
      \begin{quest}
        Is there \anesscont\ strongly minimal set whose geometry is not isomorphic to the geometry of any discrete strongly minimal set?

        Is there a continuous, but not \esscont, strongly minimal theory with the same?
      \end{quest}

This last question becomes trivial on cardinality grounds if we consider pregeometries, rather than geometries.

\subsubsection{Characterization of Strongly Minimal Groups}
\label{sec:ess-cont-stronk-min-grp}

In this section we will give a topological group theoretic characterization of strongly minimal groups in continuous logic and identify which among them are \esscont.


\begin{defn}
  A theory of groups is a theory $T$ in a language containing a binary function symbol $x\cdot y$ (which we will freely write as concatenation), a unary function symbol $x^{-1}$, and a constant symbol $e$ such that $T$ implies the following.
  \begin{itemize}
  \item $\sup_{xyz}d(x(yz),(xy)z)=0$
  \item $\sup_{x}d(ex,x)=0$
  \item $\sup_xd(x^{-1}x,e)=0$
  \end{itemize}
  A \emph{group structure} is a model of a theory of groups.
\end{defn}

Note that despite appearances, $\mathsf{IHS}$ is not an example of a theory of groups. In fact, it is possible to show that $\mathsf{IHS}$ does not even interpret a non-compact group. If we had taken the whole Hilbert space as the structure, rather than just the unit ball, and modified the metric accordingly, then the result theory would be a theory of groups.

The following was originally shown in \cite[Prop.\ 3.13]{benyaacov2010}.
We have given a full proof here for completeness.

\begin{prop}\label{prop:group-invar-met}
 Let $T$ be a theory of groups. There is a definable bi-invariant metric $\rho$ that is uniformly equivalent to $d$ (the original metric). 
\end{prop}
\begin{proof}
  Consider the formula $\rho(x,y)=\sup_{z,w}d(zxw,zyw)$. First note that the formula is clearly bi-invariant in the sense that $\rho(x,y)=\rho(uxv,uyv)$ for any $x,y,u,v$. Also note that it is clearly symmetric, non-negative, and satisfies $\rho(x,x)=0$.

  To verify the triangle inequality, consider $a,b,c$ and assume we are working a sufficiently saturated model (since the triangle inequality is first-order, this will imply the same for all models of the theory). For any $f,g$ we have that $d(fag,fcg)\leq d(fag,fbg)+d(fbg,fcg)\leq \rho(a,b)+\rho(b,c)$, so the triangle inequality holds.

  This establishes that $\rho(x,y)$ is a pseudo-metric. To establish that it is a metric we need to show that it is uniformly equivalent to $d$. Since it is a definable pseudo-metric, it is automatically uniformly dominated by $d$, so we only need to show that $d$ is uniformly dominated by it. By compactness, it is enough to show that in any model $\frk{M}$ if $\rho(a,b)=0$, then $d(a,b)=0$. Let $a$ and $b$ in $\frk{M}$ have $\rho(a,b)=0$. Pass to an elementary extension $\frk{N}\succeq\frk{M}$ containing $c$ and $f$ such that $d(caf,cbf)=0$, so in particular $caf=cbf$. By the group axioms we have that $a=c^{-1}caff^{-1}=c^{-1}cbff^{-1}=b$, i.e.\ $d(a,b)=0$. By elementarity, this is true in $\frk{M}$ as well, so we are done.
\end{proof}

Proposition~\ref{prop:group-invar-met} seems to be in conflict with the well known fact that there are metrizable groups that do not admit a bi-invariant metric. There is a hidden assumption here, which is that the group operations are uniformly continuous, and not merely continuous. 

Our proof that strongly minimal groups are Abelian in continuous logic is heavily based on the proof of the analogous statement in \cite[Cor.\ 3.5.5]{buechler_2017}, originally proven in \cite{Reineke1975}. 


One should think of the following condition as being a natural analog of having finite order.

\begin{defn}
  Given a topological group $G$, we say than an element $a$ has \emph{pre-compact orbit} if the orbit of $a$,  $a^{\mathbb{Z}}$, is pre-compact (i.e.\ has compact closure).\footnote{Elements of a topological group with this property are referred to in some literature as just `compact.'}
\end{defn}
\begin{lem}
  Let $G$ be a topological group with a topology induced by the metric $d$. If $a\in G$ has pre-compact orbit, then for every $\e > 0$ there are arbitrarily large $\ell$ such that $d(a^\ell,e)<\e$.
\end{lem}
\begin{proof}
  For any $\delta>0$, the set $\bigcup_{n\in\mathbb{Z}}B_{<\delta}(a^n)$ covers the closure of the orbit of $a$. By compactness we can find a finite set $X\subset \mathbb{Z}$ such that $\bigcup_{n\in X}B_{<\delta}(a^n)$ covers the closure of the orbit of $a$, and therefore the orbit of $a$. This implies that for any $m< \omega$ there exists $n \in X$ and $k > n+m$ such that $d(a^n,a^k)<\delta$. Since we can do this for any $\delta>0$ and by considering $a^ka^{-n}$, we have that there are arbitrarily large $\ell$ such that $d(a^\ell,e)<\e$.
\end{proof}


\begin{lem}\label{lem:group-2}
  Let $G$ be a metric group with bi-invariant metric $d$. If all elements of $G$ have pre-compact orbit and all elements of $G\smallsetminus \{e\}$ are conjugate, then $|G|\leq 2$. 
\end{lem}
\begin{proof}
  We may assume that $G$ has more than one element. Fix $a \in G \smallsetminus \{e\}$. We need to show that $a^{2}=e$. If $a^{2}=e$, then we are done, so assume that $a^2\neq e$. 

  For any $\e > 0$, find the smallest positive $n_\e$ such that $d(a^{n_\e},e)<\e$, which, note, is larger than $1$ whenever $\e \leq d(a,e)$. Note that by the conjugacy condition and the bi-invariance of the metric, $n_\e$ does not depend on the choice of $a$.

  \emph{Claim:} $n_\e$ is either $1$ or a prime number.

  \emph{Proof of claim.} Assume that $n_\e$ is not $1$ and is the composite number $mk$, with $m,k>1$. By assumption, $a^m\neq e$, so since $a^{n_\e}$ and $a^m$ are conjugate, we can find $b$ such that $b^{-1}a^{n_\e}b=a^m$. By the bi-invariance of the metric we have that $d(a^{n_\e},e)=d(b^{-1}a^{n_\e}b,b^{-1}eb)=d(a^m,e) < \e$, contradicting the minimality of $n_\e$. \hfill $\qed_{\text{claim}}$
  

  Suppose that $n_\e$ is not eventually $2$ as $\e \to 0$.  Since $a^{-1}\neq e$, we can find a $b \in G$ such that $b^{-1}ab=a^{-1}$. This implies that for any $k$, $b^{-k}ab^{k}=a^{(-1)^k}$. By our assumption, for any $\gamma > 0$, there is an $\e > 0$ with $\e < \delta$ such that $n_\e$ is odd and such that if $d(c,e) < \e$, then $d(c^{-1}ac,a) < \gamma$. Now we have that $b^{-n_\e}ab^{n_\e}=a^{-1}$, since $n_\e$ is odd, implying that $d(a,a^{-1})<\gamma$. Since we can do this for any $\gamma > 0$, we have that $a=a^{-1}$, which is a contradiction. Therefore $n_\e$ must be eventually $2$ as $\e \to 0$, but this implies that in fact $a^2 = e$.

  So, by the conjugacy condition, we have that every element a of $G$ satisfies $a^2=e$. It is not hard to show that this implies that $G$ is Abelian, but this implies that every element's conjugacy class consists solely of itself, and hence $|G|=2$, as required.
\end{proof}
\begin{lem}\label{lem:stronk-min-subgroup}  
 If $G$ is a strongly minimal group and $H$ is a type-definable, non-compact subgroup of $G$, then $H=G$.
\end{lem}
\begin{proof}
  If $H$ is a proper subgroup of $G$, then the cosets of $H$ are also non-compact type-definable subsets of $G$. Furthermore, $H$ is disjoint from its non-trivial cosets, which contradicts strong minimality. 
\end{proof}

\begin{prop}\label{prop:stronk-min-Abel}
Let $T$ be a strongly minimal theory of groups. $T$ is Abelian.
\end{prop}
\begin{proof}
  For the sake of this proof we will write $A\backslash B$ to mean the collection of right cosets of $A$ in the set $B$, i.e.\ $\{Ab:b \in B\}$, as opposed to the set theoretic relative complement, which is written $A\smallsetminus B$ elsewhere in this paper. The $A\smallsetminus B$ notation does not occur in this proof, and the $A\backslash B$ notation does not occur outside of this proof.  
  
  Assume that $T$ is not Abelian. By Proposition~\ref{prop:group-invar-met} we may assume that $T$ has a bi-invariant metric. Let $Z$ be the center of $G$ (i.e.\ the set of all elements $g$ satisfying $gh=hg$ for all $h \in G$). 
  For any $g \in G$, let $C(g)$ be the centralizer of $g$ (i.e.\ the set of all elements $h$ satisfying $h^{-1}gh=g$). Note that the orbit of $g$ is contained in $C(g)$. This is a type-definable group defined by the closed condition $d(h^{-1}gh,g)=0$. For any $g \in G$ with $g \notin Z$, $C(g)$ is a type-definable proper subgroup of $G$, so by Lemma~\ref{lem:stronk-min-subgroup} it is compact, hence $g$ has pre-compact orbit.

  \emph{Claim:} For any $g \in G$ with $g \notin Z$, the set of conjugates of $g$, $a^G\coloneqq \{h^{-1}gh:h\in G\}$, is not pre-compact (i.e.\ has compact closure).

  \emph{Proof of claim.} 
  We want to argue that there is a natural topological isomorphism between (the metric closure of) $C(g)\backslash G$ and (the metric closure of) $a^G$, where $C(g)\backslash G$ is made into a metric space by the Hausdorff metric on sets, which, by bi-invariance and the fact that $C(g)\backslash G$ is a set of cosets, is equal to $\rho(x,y) = \inf_{z \in C(g)}d(x,zy)$. 
  Then we will show that $C(g)\backslash G$ is not pre-compact.

  It is a basic algebraic fact that there is a natural bijection between $C(g)\backslash G$ and $g^G$. Explicitly, for any $a$ and $b$, we have that if $a^{-1}ga=b^{-1}gb$, then $gab^{-1}=ab^{-1}g$, so $ab^{-1}\in C(g)$, implying that $C(g)b=C(g)ab^{-1}b=C(g)a$. This gives a natural function from $g^G$ to $C(g)\backslash G$. To see that it is surjective, note that for any $C(g)a$, the element $a^{-1}ga$ maps to $C(g)a$. To see that it is injective, if $G(g)a=G(g)b$, then this implies that $a=cb$ for some $c \in C(g)$, so we have that $a^{-1}ga=b^{-1}c^{-1}gcb=b^{-1}gb$, as required. 

  All we need to do is argue that this bijection is metrically uniformly continuous with metrically uniformly continuous inverse. I claim that for any $\e > 0$, there is a $\delta > 0$ such that if $d(ag,ga)< \delta$, then $d(a,C(g))<\e$. This follows from compactness and the fact that $C(g)$ is an algebraic set. 
  For any given $\e > 0$, find such a $\delta > 0$ and consider $a$ and $b$ such that $d(a^{-1}ga,b^{-1}gb) < \delta$. By bi-invariance, this implies that $d(gab^{-1},ab^{-1}g) < \delta$, so we have that there is some $c \in C(g)$ with $d(ab^{-1},c) < \e$. This implies that $d(a,cb) < \e$, so we have that the distance between $C(g)a$ and $C(g)b$ in the $\rho$-metric is less than $\e$. Since we can do this for any $\e > 0$, we have that the map is uniformly continuous.

  Now we just need to show that the map has uniformly continuous inverse. For any $\e > 0$, find a $\delta > 0$ small enough that for any $c$ if $d(a,b) < \delta$, then $d(a^{-1}ca,b^{-1}cb)< \e$.  Assume that $C(g)a$ and $C(g)b$ have $\rho$-distance less than $\delta$. By bi-invariance, this implies that there is some $c \in C(g)$ such that $d(a,cb) < \delta$. This implies, by the choice of $\e$, that $d(a^{-1}ga,b^{-1}c^{-1}gcb) < \e$, and so $d(a^{-1}ga,b^{-1}gb) < \e$. Since we can do this for any $\e > 0$, we have that the inverse of the bijection is uniformly continuous.

  To show that $C(g)\backslash G$ is not pre-compact, first note that we have already established that $C(g)$ is compact. If $C(g)\backslash G$ were pre-compact, then this would imply that for any $\e > 0$ there is a finite set of elements $X \subseteq C(g)\backslash G$ such that $G \subseteq \bigcup_{x \in X}q^{-1}(x)^{<\e}$, where $q : G \rightarrow C(g)\backslash G$ is the natural quotient map. Since each pre-image $q^{-1}(x)$ is isometric to $C(g)$, this is enough to imply that $G$ is compact, which contradicts our assumptions. \hfill $\qed_{\text{claim}}$

  Now we have that for any $a$ and $b$ not in $Z$, the metric closure of $a^G$ is equal to the metric closure of $b^G$, because these sets are definable and non-compact, and so, by strong minimality, must have non-empty intersection in some elementary extension, which implies in the elementary extension that there is some $c$ such that $c^{-1}ac=b$. By passing to a sufficiently saturated elementary extension, we may assume that every pair $a$ and $b$ not in $Z$ are conjugate. This implies that in $G/Z$ any two non-identity elements are conjugate (because conjugacy commutes with group homomorphisms). We also still have that every element has pre-compact orbit (after endowing $G/Z$ with the Hausdorff metric on cosets, as we did with $C(g)\backslash G$), so by Lemma~\ref{lem:group-2} we have that $|G/Z|\leq 2$. This implies that $Z$ must be non-compact (otherwise $G$ would be compact), but then by Lemma~\ref{lem:stronk-min-subgroup} we have that $G=Z$, which contradicts our assumption that $G$ was not Abelian. 
%
 %
\end{proof}

The corresponding result in \cite{buechler_2017}---namely, Proposition 3.5.2---is the statement that any infinite $\omega$-stable group has an infinite definable Abelian subgroup. The proof as stated does not generalize to continuous $\omega$-stable groups, as we now have many different Morley ranks rather than a single one.

  \begin{quest}
    To what extent does \cite[Prop.\ 3.5.2]{buechler_2017} generalize to arbitrary $\omega$-stable groups in continuous logic? Is it true that every non-compact $\omega$-stable group has a definable non-compact Abelian subgroup? Does it follow if we assume that models of the theory are locally compact?
  \end{quest}

  It is known that type-definable subgroups in $\omega$-stable theories are definable \cite{benyaacov2010}, so it is not preposterous to hope that there may be nice definable subgroups in arbitrary $\omega$-stable continuous groups. 

  We would like to take a moment to negatively answer a question given in the introduction of \cite{benyaacov2010}, which to our knowledge does not have an answer at least in the published literature. In \cite{benyaacov2010}, Ben Yaacov asked whether or not type-definable groups in stable theories are always intersections of some family of definable groups. It turns out that, unlike in discrete logic, superstability is not enough to ensure that type-definable groups are the intersections of families of definable groups in continuous logic. It is not hard to show that the structure $(\mathbb{Q},+,\cos,\sin)$\footnote{One should think of this as $(\mathbb{Q},+)$ with a non-trivial homomorphism to the circle group $S^1$.} with the discrete metric has a weakly minimal and therefore superstable theory but also has the property that the type-definable group given by $\{\cos(x)=1,\sin(x)=0\}$ is not, in sufficiently saturated elementary extensions, the intersection of any family of definable subgroups. In fact, the only definable subgroups in this theory are the trivial subgroup and the group itself. This is in line with the general phenomenon, discussed in Section~\ref{sec:dict}, that, while $\omega$-stable theories are well behaved with regards to definable sets, even superstable theories are not. 

  For a full characterization of \esscont\ strongly minimal groups, we will need the following significant fact from the theory of topological groups.

  \begin{fact}[{\cite[Thm.\ 25]{Morris1977}}]\label{fact:lcHAg-char}
    Every locally compact Hausdorff Abelian group has an open subgroup topologically isomorphic to $\mathbb{R}^n\times K$ for some compact group $K$ and some non-negative integer $n$.
  \end{fact}

  We will also need the following well known algebraic fact regarding Abelian groups.

  \begin{fact}\label{fact:alg-well-known}
    An Abelian group $G$ is divisible if and only if it is an injective object in the category of Abelian groups, i.e.\ if and only if for any group $H$ with subgroup $F \subseteq H$ and any homomorphism $f:F\rightarrow G$, there exists a homomorphism $g: H \to G$ extending $f$.\editcom{Added second `if and only if.'}
  \end{fact}

  Facts~\ref{fact:lcHAg-char} and \ref{fact:alg-well-known} imply that every locally compact Hausdorff Abelian group can be written in the form $\mathbb{R}^n\times G$, with $n$ a non-negative integer and $G$ a totally disconnected locally compact Hausdorff group. (To see this, note that any homomorphism extending the inclusion map of $\mathbb{R}^n$ into $\mathbb{R}^n\times K$ to the entire group must be continuous, and $G$ can be taken to be the kernel of this extension.)

\begin{lem}\label{lem:strong_min_and_cat_in_cont:3} 
  If $G$ is a group structure with a bi-invariant metric and $H$ is a definable subgroup of $G$, then the imaginary sort given by quotienting by the definable pseudo-metric $\rho(x,y)=\inf_{z\in H}d(x,yz)$ corresponds to the set $G/H$ of left cosets of $H$. (And a similar metric gives the set of right cosets of $H$.)

  If $H$ is a normal subgroup, then the natural group structure on $G/H$ is definable on $G/\rho$.
\end{lem}
\begin{proof}
  The fact that $\rho$ is a pseudo-metric follows from the bi-invariance of the metric (really we only need right invariance). The easiest way to see this is that $\rho(x,y)$ is equal to $d_H(xH,yH)$, where $d_H$ is the Hausdorff metric on sets.

  We clearly have that if $a$ and $b$ are conjugate by an element of $H$, then $\rho(a,b)=0$. Conversely, assume that $\rho(a,b)=0$. This implies that for any $\e >0$ we can find a $c$ such that $d(a,bc)<\e$. By bi-invariance, this is equivalent to $d(e,a^{-1}bc)<\e$. So we have that for every $\e >0$, $\dinf (a^{-1}b,H)<\e$, which implies that $a^{-1}b\in H$, since $H$ is definable and therefore closed. 

  Now we just need to show that if $H$ is a normal subgroup, then the group structure on $G/H$ is definable on $G/\rho$. Let $q: G \rightarrow G/\rho$ be the quotient map. For any $a,b,c \in G$, we have that $\rho(q(a)q(b),q(c))$ is equal to $\rho(ab,c)$, since $H$ is a normal subgroup. This is therefore a $\rho$-invariant formula and corresponds to a formula on the imaginary sort.
\end{proof}

\begin{thm}\label{thm:group_char}
  A completely metrizable topological group $G$ can be given a metric under which it is a strongly minimal group if and only if it is non-compact and
  \begin{itemize}
  \item there is a prime $p$ such that every element of $G$ has order $p$ or
  \item $G$ is divisible and can be written as $\mathbb{Q}^\kappa\oplus\mathbb{R}^n\oplus H$, where $H$ has a compact subgroup $K$ such that $H/K = \bigoplus_p(\mathbb{Z}/p^\infty\mathbb{Z})^{\alpha_p}$, with $n$ and each $\alpha_p$ finite and $\kappa$ arbitrary, where $\mathbb{Z}/p^\infty \mathbb{Z}$ is the $p$-Pr\"ufer group.\footnote{The $p$-Pr\"ufer group is isomorphic to the group of $p^n$th roots of unity under multiplication.}
  \end{itemize}
  Furthermore, the resulting theory is \esscont\ if and only $G$ is of the second form with $n>0$.
\end{thm}
\begin{proof}
  $(\Rightarrow)$. Suppose that $G$ is a strongly minimal group structure. 
  By Proposition~\ref{prop:stronk-min-Abel}, $G$ is Abelian.

  For any prime $p$, let  $\overline{pG}$ denote the metric closure of $\{px:x\in G\}$, which is a subgroup. It is definable by $d(x,\overline{pG})=\inf_{y}d(x,py)$. Suppose that for some $p$, $\overline{pG}$ is a proper subgroup of $G$. By Lemma~\ref{lem:stronk-min-subgroup}, this implies that $\overline{pG}$ is compact. Consider the type-definable subgroup $\{x\in G: px=e\}$. If this is compact, then we have that $G$ is homeomorphic to a product of two compact sets, which is a contradiction, therefore $\{x\in G: px=e\}$ must be all of $G$, by Lemma~\ref{lem:stronk-min-subgroup}, and so $G$ falls under the first bullet point.

  Now assume that for every $p$, $\overline{pG}$ is all of $G$. Let $G^\prime$ be an $\omega$-saturated elementary extension of $G$. We now have that $G^\prime$ is a divisible group. By Fact~\ref{fact:lcHAg-char}, $G^\prime$ can be written as $\mathbb{R}^n\oplus H$, where $H$ has a compact subgroup $K$ such that $H/K$ is a discrete Abelian group. $n$ must be finite by local compactness. By replacing $K$ with $G\cap K$, we may assume that $K$ is a subgroup of $G$. Since $G^\prime$ is divisible, we must have that $H$ is divisible as well. By the characterization of divisible Abelian groups, $H/K$ can be written as $\mathbb{Q}^\kappa\oplus \bigoplus_p (\mathbb{Z}/p^\infty\mathbb{Z})^{\alpha_p}$. By Fact~\ref{fact:alg-well-known}, we have that $H$ can be written as $\mathbb{Q}^k\oplus L$, where $K$ is a subgroup of $L$ and $L/K$ is isomorphic to $\bigoplus_p (\mathbb{Z}/p^\infty\mathbb{Z})^{\alpha_p}$. For each prime $p$, the type-definable subgroup $\{x \in G^\prime: px=e\}$ cannot be all of $G^\prime$ and so by Lemma~\ref{lem:stronk-min-subgroup} is compact. This implies that $\alpha_p$ must be finite. Since we can do this for every prime $p$, we have that $G^\prime$ falls under the case in the second bullet point.

  $G$ is an open subgroup of $G^\prime$. This implies that $\mathbb{R}^n\oplus \{e\} \subseteq G^\prime$ is also a subgroup of $G$. This implies that we can write $G$ as $\mathbb{R}^n\oplus L$ with $K \subseteq L$ and $L/K$ a discrete group. Each subgroup $\{x \in G: px = e\}$ is algebraic, so we must have that $\{x \in G: px = e\} = \{x \in G^\prime : px = e\}$. Thinking of $L/K$ as a subgroup of $H/K$, this implies that each factor of the form $\mathbb{Z}/p^\infty\mathbb{Z}$ is contained in $L/K$. Therefore $L/K$ must be of the form $\mathbb{Q}^\lambda\oplus \bigoplus_p(\mathbb{Z}/p^\infty\mathbb{Z})^{\alpha_p}$ for some $\lambda \leq \kappa$, and $G$ falls under the case in the second bullet point.

  $(\Leftarrow)$. Let $G$ be a completely metrizable group such that every element of $G$ has order $p$. Let $K$ be a compact open subgroup of $G$ such that $G/K$ is discrete. By standard group theory, $G$ and $K$ are, algebraically speaking, vector spaces over $\mathbb{F}_p$. Let $d^K$ be an arbitrary bi-invariant metric on $K$ inducing the topology with diameter at most $\frac{1}{2}$. Let $d$ be a metric on all of $G$ defined by $d(a,b)=d^K(ab^{-1},e)$ if $ab^{-1}\in K$ and $d(a,b)=1$ otherwise. Because $d^K$ is bi-invariant, this is a metric. Under this metric,  $K$ is contained in the $\frac{2}{3}$-ball of $e$ and so is in the algebraic closure of $\varnothing$ and is in every model of $\mathrm{Th}(G)$. If $G^\prime$ is any model of $\mathrm{Th}(G)$ and $G^{\prime\prime}$ is any elementary extension of $G^\prime$, it is not hard to show that any two elements of $G^{\prime\prime}\smallsetminus G^{\prime}$ are automorphic, so $\mathrm{Th}(G)$ is strongly minimal.

  Let $G$ be a completely metrizable group falling under the case in the second bullet point. Let $K$ be the compact subgroup and let $d^K$ be an arbitrary bi-invariant metric giving the topology on $K$ with diameter at most $\frac{1}{2}$. Given $a,b \in G$, write them as $(a_0,a_1)$ and $(b_0,b_1)$, where $a_0,b_0 \in \mathbb{R}^n$ and $a_1,b_1 \in \mathbb{Q}^\kappa \oplus H$. Put a metric on $G$ by $d(a,b)=1$ if $a_1 b_1^{-1} \notin K$ and $d(a,b)=\max\left\{\frac{\left\lVert a_0-b_0 \right\rVert_\infty}{1+\left\lVert a_0-b_0 \right\rVert_\infty},d^K(a_1b_1^{-1},e)\right\}$ otherwise. This is a metric which induces the topology on $G$. Note that $\frac{\left\lVert a_0-b_0 \right\rVert_\infty}{1+\left\lVert a_0-b_0 \right\rVert_\infty}$ is a metric on $\mathbb{R}^n$, so $d$ can be written as the maximum of a metric on $\mathbb{R}^n$ and a metric on $\mathbb{Q}^\kappa\oplus H$ and therefore is a metric inducing the product topology. Note that since $G$ is non-compact, the diameter of $G$ with regards to $d$ is $1$.

  Let $G^\prime$ be an elementary extension of $(G,d)$. The theory of $(G,d)$ entails that the metric on $G^\prime$ is $[0,1]$ valued, but also entails that for any $\e >0$ with $\e < 1$, the closed $\e$-ball of any element is compact (and isometric to the corresponding closed ball around $e$).

  \emph{Claim:} For any prime $p$ and any positive integer $n$, the type-definable set $\{x:p^nx =e\}$ is definable and algebraic.

  \emph{Proof of claim.} Each of these sets is compact in $G$, so we only need to show that they are definable. Fix $\e > 0$ and find $\delta > 0$ small enough that for $a \in K$ if $d^K(p^na,e)< \delta$, then there exists a $b \in K$ with $d(a,b) < \e$ such that $p^n b= e$ (this always exists by compactness). We may take $\delta$ to be less than $\frac{1}{2}$ and less than $\e$.

  For $a \in G$, suppose that $d(p^na,e)< \delta$. If we write $a$ as $(a_0,a_1)$ with $a_0 \in \mathbb{R}^n$ and $a_1 \in \mathbb{Q}^\kappa \oplus H$, then this implies that $\frac{\left\lVert p^na_0 \right\rVert_\infty}{1+\left\lVert p^n a_0 \right\rVert_\infty} < p^n \e$ and also that $p^n a_1 \in K$ with $d(p^na_1,e) < \delta$. By our choice of $\delta$, this implies that there is some $b_1 \in K$ such that $d(a_1,b_1) < \e$ and $p^nb_1=0$. This implies that $d((a_0,a_1),(0,b_1)) < \max\{\e,\delta\} < \e$. Therefore, we have that the relevant set is definable. \hfill $\qed_{\text{claim}}$

  The claim implies that any $a \in G^\prime \smallsetminus G$ is divisible and torsion free. By the classification of divisible Abelian groups and the fact that $G^\prime/K$ is locally homeomorphic to $G/K$, we have that $G^\prime$ is topologically isomorphic to $G\oplus \mathbb{Q}^\lambda$ for some cardinal $\lambda$. Furthermore, we have that for any $a,b \in G^\prime$, if $d(a,b) < 1$, then $b-a$ is in $G$. This implies that $G^\prime$ can be realized as $G\oplus \mathbb{Q}^\lambda$ with the discrete metric on $\mathbb{Q}^\lambda$ and the max metric on the product. From this we get that any two elements of $G^\prime \smallsetminus G$ are automorphic, and the same argument will work for any elementary extension $G^{\prime\prime}$ of $G^\prime$, so we have that $T$ is strongly minimal.

  The `furthermore' statement follows directly from Theorem~\ref{thm:strong_min_and_cat_in_cont:2}.
\end{proof}

It is possible that $G$ in the statement of Theorem~\ref{thm:group_char} is not of the form $K \times G/K$. An example of this is the additive group of the $p$-adic numbers with the appropriate metric, as discussed after Example~\ref{exa:min-non-triv}. $K$ will be the subgroup of $p$-adic integers or some scaling of them, and the $p$-adic numbers do not decompose as a direct sum with any of these groups as a factor.

It is also possible to give a bi-invariant metric to a group of one of the forms given in Theorem~\ref{thm:group_char} which will make it fail to be strongly minimal. This is very easy when the group is discrete---$\mathbb{Q}$ with the metric $d(x,y)$ which is $1$ when $x-y \in \mathbb{Z} \smallsetminus \{0\}$ and $2$ when $x-y \notin \mathbb{Z}$---but it is also possible when the group is of a form that would result in \anesscont\ theory---$\mathbb{R}$ with the metric $d(x,y)=\min\{|x-y|,1\}+d(x,y+\mathbb{Z})$, which is a metric as it is the sum of a metric and a pseudo-metric.

Characterizing the metrics which make the groups identified in Theorem~\ref{thm:group_char} strongly minimal seems difficult. We were unable to answer what ought to be the easiest question related to this issue.

\begin{quest}
  If $(G,d,+)$ is a group structure such that $(G,d,+)$ is topologically one of the groups specified in Theorem~\ref{thm:group_char}, $d$ is a bi-invariant metric, and $(G,d)$ is a strongly minimal metric space, does it follow that $(G,d,+)$ is strongly minimal?
\end{quest}

\def\Th{\mathrm{Th}}

Another direction for future study would be to replicate the very tight characterization of transitive faithful $\omega$-stable group actions on strongly minimal sets:
\begin{fact}[{\cite[Thm.\ 3.5.2]{buechler_2017}}]\label{fact:stronk-min-thys:1}
  If $(G,X)$ is a (discrete) $\omega$-stable transitive faithful group action with $X$ strongly minimal, then $\mathrm{MR}(G)\leq 3$ and
  \begin{enumerate}
  \item if $\mathrm{MR}(G) = 1$, then $G$ has a definable finite index subgroup $H$ which acts regularly on $X$ and
  \item if $\mathrm{MR}(G) \geq 2$, then there is a field $K$ definable on $X$ or $X\setminus \{a\}$ for some point $a$.
  \end{enumerate}
\end{fact}
Recall that a group action is \emph{transitive} if the orbit of every element is the entire set, it is \emph{faithful} if $\mathrm{Stab}(X)$ is trivial, and it is \emph{regular} if for any $a \in X$, the function $g \mapsto ga$ is a bijection. Note that in the first case, $H$ is strongly minimal.

A completely reckless conjecture would be that only the first case of Fact~\ref{fact:stronk-min-thys:1} can occur in \anesscont\ theory (this would be related to a negative answer to Question~\ref{quest:stronk-min-thys:1}).

\begin{quest}
  Is it possible to have a locally compact metric group $G$ with bi-invariant metric acting faithfully and transitively on a metric space $X$ such that $\Th(G,X)$ is \esscont, $\omega$-stable, and has $X$ strongly minimal without $G$ having a definable strongly minimal group with compact index that acts regularly on $X$?
\end{quest}

Local compactness also adds a new parameter to these questions regarding these kinds of characterizations.

\begin{quest}
  If $G$ is a (not necessarily locally compact) metric group with bi-invariant metric acting faithfully and transitively on a metric space $X$ such that $\Th(G,X)$ is $\omega$-stable and has $X$ strongly minimal, is $G$ necessarily locally compact?
\end{quest}

\subsection{A Partial Baldwin-Lachlan Characterization}
\label{sec:Bald-Lach}

\subsubsection{Main Theorem}
\label{sec:main-thm-sec}

Assuming $T$ is a theory with strongly minimal sets, part of the
Baldwin-Lachlan characterization goes through exactly. This statement is analogous to the discrete statement `For a countable theory $T$, if $T$ has a prime model and a minimal set definable over it, then for any $\kappa \geq \aleph_1$, $T$ is $\kappa$-categorical if and only if $T$ has no Vaughtian pairs.' Our continuous generalization of this statement is made more complicated by a few factors. We strengthen the result by using the weakening of strongly minimal set given in Definition~\ref{defn:approx-stronk-min}. We also need one of two strengthenings of no Vaughtian pairs, either of which is sufficient. And, given the presence of certain counterexamples in continuous logic (such as Example~\ref{exa:MAIN-COUNTEREXAMPLE}), we would like to state the result both for definable sets in the home sort and for arbitrary imaginaries.

\begin{thm}\label{thm:main-thm}
  Let $T$ be a countable complete theory with non-compact models and let $\kappa$ be any uncountable cardinal.
  \begin{enumerate}[label=(\roman*)]
  \item \label{main-1} If $T$ has a prime model and an approximately minimal set definable over it, then the following are equivalent.
    \begin{enumerate}
    \item $T$ is $\kappa$-categorical.
    \item \label{main-1-1} $T$ is dictionaric and has no Vaughtian pairs.
    \item \label{main-1-2} $T$ has no open-in-definable Vaughtian pairs. 
    \end{enumerate}
  \item \label{main-2} If $T$ has a prime model and an approximately minimal imaginary definable over it, then the following are equivalent.
    \begin{enumerate}
    \item $T$ is $\kappa$-categorical. 
    \item $T$ is dictionaric and $T^{\mathrm{eq}}$ has no Vaughtian pairs. 
    \item $T^{\mathrm{eq}}$ has no open Vaughtian pairs.
    \end{enumerate}
  \end{enumerate}
\end{thm}

\begin{proof}
  Since $T$ has an approximately minimal set definable over its prime model, by Proposition~\ref{prop:strong_min_and_cat_in_cont:2} we have that there is a minimal set definable over its prime model.
  
  \ref*{main-1}. If $T$ is $\kappa$-categorical, then $T$ is $\omega$-stable and therefore dictionaric, and also has no open-in-definable Vaughtian pairs, and therefore no Vaughtian pairs.

  If either (b) or (c) is true, then by Proposition~\ref{prop:vaught-minimal}, the definable minimal set is strongly minimal. Let
$D$ be this strongly minimal set. Let $A$ be a basis in $D$ of
cardinality $\kappa$, and let $\mathfrak{A}$ be prime over $\mathrm{acl}(A)$.
If $\mathfrak{B}$ is a model of density character $\kappa$, then since
$T$ has no Vaughtian pairs, $D(\mathfrak{B})$ has density character
$\kappa$, and so we can find a basis $B$ of $D(\mathfrak{B})$ of
cardinality $\kappa$. Therefore we can find an isomorphism $D(\mathfrak{A})\cong D(\mathfrak{B})$.
Since $\mathfrak{A}$ is prime over $D(\mathfrak{A})$ we can extend
this isomorphism to an embedding $\mathfrak{A}\preceq\mathfrak{B}$,
but since $T$ has no Vaughtian pairs this must be an isomorphism.
Therefore all models of density character $\kappa$ are isomorphic
to $\mathfrak{A}$.

\ref*{main-2}. The proof here is the same as the proof of \ref*{main-1} with the following notes: If $T$ is dictionaric, then $T^\mathrm{eq}$ is dictionaric over models, which is enough to show that minimal sets are strongly minimal in case (b). In case (c), if $T$ has no open Vaughtian pairs in imaginaries, then it has no open-in-definable Vaughtian pairs in imaginaries, since definable subsets of imaginaries are imaginaries.
\end{proof}

Note that the conclusion of Theorem~\ref{thm:main-thm} also holds if we know that $T$ has an approximately minimal set (or imaginary) definable over every model, as then we can show that $T_A$ for some countable set of constants is $\kappa$-categorical, and therefore $\omega$-stable, implying that $T$ has a prime model with an approximately minimal set (or imaginary) definable over it.

Now we will explore some cases under which the assumptions of Theorem~\ref{thm:main-thm} are satisfied.

\subsubsection{Theories with a Locally Compact Model}

Here we give a Baldwin-Lachlan characterization for theories with a locally compact model. We should note that having a locally compact model is not equivalent to having every model locally compact.\editcom{Added paragraph.}

\label{sec:theor-with-locally}
\begin{prop}
\label{prop:loc-comp}
\leavevmode
\begin{enumerate}[label=(\roman*)]
\item If $T$ is a theory with a non-compact, locally compact model
$\mathfrak{M}$, and $S_{1}(\mathfrak{M})$ is CB-analyzable (in particular if $T$ is $\omega$-stable or $S_{1}(\mathfrak{M})$ is small), then
pre-minimal types are dense in $S_{1}(\mathfrak{M})\smallsetminus\mathfrak{M}$
(which is closed).

\item If $T$ is a totally transcendental theory with non-compact models such that every model is locally compact, then strongly minimal global types are dense
among non-algebraic global types.
\end{enumerate}
\end{prop}

\begin{proof}
\emph{(i).} First we need to see that $S_{1}(\mathfrak{M})\smallsetminus\mathfrak{M}$
is closed. For every $a\in\mathfrak{M}$ there is an $\varepsilon>0$
such that $B_{\leq2\varepsilon}(a)$ is compact. This implies that
$B_{\leq\varepsilon}(a)$ is algebraic over $a$ by Lemma \ref{lem:basic-alg}.
Therefore in particular $B_{<\varepsilon}(a)\cap\mathfrak{M}=B_{<\varepsilon}(a)\subseteq S_{1}(\mathfrak{M})$.
Therefore $\mathfrak{M}$ as a subset of $S_{1}(\mathfrak{M})$ is
a union of open sets and is itself open. Since $S_{1}(\mathfrak{M})$
is CB-analyzable, $d$-atomic-in-$S_{1}(\mathfrak{M})\smallsetminus\mathfrak{M}$
types are dense in $S_{1}(\mathfrak{M})\smallsetminus\mathfrak{M}$. Let
$p$ be any such type. Since $S_{1}(\mathfrak{M})$ is CB-analyzable,
it is dictionaric, so let $D\subseteq S_{1}(\mathfrak{M})$ be a definable
set such that $D\cap(S_{1}(\mathfrak{M})\smallsetminus\mathfrak{M})=\{p\}$.
Then $D$ is a minimal set: If $F,G\subseteq D$ are two $M$-zerosets, at most one of them can contain $p$, so at most one of
them can be non-algebraic. Since this is true of any type that is relatively 
$d$-atomic in $S_{1}(\mathfrak{M})\smallsetminus\mathfrak{M}$, we
have that pre-minimal types are dense in $S_{1}(\mathfrak{M})\smallsetminus\mathfrak{M}$.

\emph{(ii).} This is immediate from the fact that pre-minimal types over $\aleph_0$-saturated
structures are strongly minimal.
\end{proof}

\begin{cor}
  If $T$ is a theory with a locally compact model, then $T$ is inseparably categorical if and only if it is $\omega$-stable and has no Vaughtian pairs.
\end{cor}

\subsubsection{Ultrametric Theories and Theories with Totally Disconnected Type
Spaces}
\label{sec:ultr-theor-theor}
$p$-adic Banach spaces are a natural example of ultrametric structures (i.e.\ ultrametric metric structures). $\ell^{\infty}$ and $c_0$ spaces over $p$-adic fields are known to behave somewhat
analogously to Hilbert spaces. 
In particular they have a good notion of orthonormal bases \cite[Thm. 5.16]{nla.cat-vn984637} and as such they are also \insep\ categorical.
This is not surprising as the unit balls of these spaces are in a
precise sense the inverse limit of the sequence of structures $(\mathbb{Z}/p^{n}\mathbb{Z})^{\omega}$
as $n\rightarrow\omega$.
\begin{ex} \label{ex:ult-example}
An \insep\ categorical ultrametric theory with no strongly minimal
set. \label{ex:no-strong-min}
\end{ex}
\begin{proof} [Verification]
Let $\mathfrak{M}=c_{0}(\omega,\mathbb{Z}_{p})$ be the unit ball
of the $p$-adic Banach space $c_{0}(\kappa,\mathbb{Q}_{p})$, i.e.
$\mathfrak{M}$ consists of elements $a$ of $\mathbb{Z}_{p}^{\omega}$
satisfying $a(i)\rightarrow0$ as $i\rightarrow\omega$. The metric
on $\mathfrak{M}$ is given by $d(a,b)=\sup d^{\mathbb{Z}_{p}}(a(i),b(i))$.
The language consists only of $+$ (note that since $\mathbb{Z}$
is dense in $\mathbb{Z}_{p}$, we don't actually need to have explicit
scalar multiplication functions).

$(\ast)$ Clearly we have that the binary relation $\cset{d(x,y)\leq p^{-1}}$ is
an equivalence relation. The imaginary obtained by quotienting by
this equivalence relation is clearly the infinite dimensional vector
field over the finite field $\mathbb{F}_{p}$. Furthermore models
of $\mathrm{Th}(\mathfrak{M})$ are prime over this imaginary. To see
this assume that $\mathfrak{A}\prec\mathfrak{B}$ is a proper elementary
pair of models of this theory. Let $b\in\mathfrak{B}\smallsetminus\mathfrak{A}$.
Let $a$ be an element of $\mathfrak{A}$ such that $d(b,a)=d(b,\mathfrak{A})$
(this always exists because the set of possible distances is reverse
well-ordered). Then we have that $d(b-a,0)=d(b,a)$ and $b-a\notin\mathrm{\ensuremath{\mathfrak{A}}}$.
If this is $1$, then we are done, since $b$ is necessarily in its
own equivalence class in the imaginary that is not contained in $\mathfrak{A}$,
otherwise $d(b-a,0)=p^{-n}$ for some $n>0$. The theory knows that
if $d(c,0)=p^{-n}$, then there is a unique element $e$ satisfying
$p^{n}e=c$. Let $p^{n}c=b-a$. We have that $c$ must be an element
of $\mathfrak{B}$ satisfying $d(c,\mathfrak{A})=1$. Therefore the
imaginary must be bigger in $\mathfrak{B}$ than it is in $\mathfrak{A}$. 

So since models of $\mathrm{Th}(\mathfrak{M})$ are prime over a strongly
minimal imaginary we have that $\mathrm{Th}(\mathfrak{M})$ is \insep\
categorical.

To see that $\mathrm{Th}(\mathfrak{M})$ has no strongly minimal sets,
note that it is enough to show that $S_{1}(\mathfrak{M})$ has no
pre-minimal types, since $\mathfrak{M}$ is approximately $\aleph_0$-saturated.
If $p\in S_{1}(\mathfrak{M})$ is some non-algebraic type, then the
argument in paragraph $(\ast)$ gives an $\mathfrak{M}$-definable
bijection between the smallest ball centered on some element of $\mathfrak{M}$
that contains $p$ and the entire structure. Since the unique non-algebraic
type $q$ satisfying $d(q,\mathfrak{M})=1$ is not pre-minimal, this implies
that $p$ is not pre-minimal. Since we could do this for any non-algebraic $p$, this implies that no type is pre-minimal, and so no type is strongly minimal.
\end{proof}

The relationship between ultrametric theories and theories with totally disconnected type spaces is summarized in the following theorem.

\begin{thm}
\label{thm:zero-dim-type-spaces}Let $T$ be a countable theory, the
following are equivalent:

\begin{enumerate}[label=(\roman*)]
\item For every $n<\omega$ and every parameter set $A$, $S_{n}(A)$
is totally disconnected (where we take a single point to be totally disconnected).\editcom{Extended parenthetical comment.}

\item For every finite parameter set $\overline{a}$, $S_{1}(\overline{a})$
is totally disconnected.

\item For every $n<\omega$, $S_{n}(\varnothing)$ is totally disconnected.

\item The diagonal in $S_{2}(\varnothing)$ (i.e.\ $\cset{d(x,y)=0}$)
has a basis of clopen neighborhoods.\editcom{Clarified statement.}

\item $T$ is dictionaric and there is a definable metric $\rho$, uniformly
equivalent to $d$, such that the distance set, $\rho(T)=\{\rho(a,b)|a,b\in\mathfrak{M}\models T\}$,
contains no neighborhood of $0$.

\item $T$ is dictionaric and there is a definable ultrametric $\rho$,
uniformly equivalent to $d$, such that $\rho(T)\subseteq\{0\}\cup\{2^{-i}|i<\omega\}$.
\end{enumerate}
\end{thm}

\begin{proof}
\emph{$\text{(i)}\Rightarrow\text{(iii)}$.} This is immediate.

\emph{$\text{(iii)}\Rightarrow\text{(i)}$.} Assume that some $S_{n}(A)$ has a nondegenerate
continuum $C$. Let $p,q\in C$ be distinct types. By compactness
there must exist a restricted formula $\varphi(\overline{x};\overline{a})$
and a finite parameter set $\overline{a}\in A$ such that $p(\overline{x})\models\varphi(\overline{x};\overline{a})=0$
and $q(\overline{x})\models\varphi(\overline{x};\overline{a})=1$.
This implies that $p\upharpoonright\overline{a}$ and $q\upharpoonright\overline{a}$
are still distinct types, so since the natural projection $\pi:S_{n}(A)\rightarrow S_{n}(\overline{a})$
is continuous, $\pi(C)$ must be a nondegenerate continuum, thus $S_{n}(\overline{a})\subseteq S_{n+|\overline{a}|}(\varnothing)$
fails to be totally disconnected.

\emph{$\text{(iii)}\Rightarrow\text{(ii)}$.} Each $S_{1}(\overline{a})$ is a subspace
of $S_{1+|\overline{a}|}(\varnothing)$, so this is immediate as well.

\emph{$\text{(ii)}\Rightarrow\text{(iii)}$.} If $S_{1}(\varnothing)$ is not totally disconnected,
then we are done. 

Let $n$ be the first $n<\omega$ such that $S_{n}(\varnothing)$
is totally disconnected but $S_{n+1}(\varnothing)$ is not totally
disconnected. Consider the projection $\pi:S_{n+1}(\varnothing)\rightarrow S_{n}(\varnothing)$.
Since $S_{n}(\varnothing)$ is totally disconnected, any continuum
in $S_{n+1}(\varnothing)$ must be contained in a single fiber. This
fiber is isomorphic to $S_{1}(\overline{a})$ for some parameter set
$\overline{a}$, so $S_{1}(\overline{a})$ is not totally disconnected. 

\emph{$\text{(iii)}\Rightarrow\text{(iv)}$.} This is immediate.

\emph{$\text{(iv)}\Rightarrow\text{(vi)}$.} Let $\{\varphi_{i}\}_{i<\omega}$ be an enumeration
of $\{0,1\}$-valued restricted formulas of two variables corresponding
to indicator functions of clopen neighborhoods of the diagonal in
$S_{2}(\varnothing)$ (note that there are only countably many of
them since every $\{0,1\}$-valued formula is equal to
some $\{0,1\}$-valued restricted formula). Define $\rho$ by
\[
\rho(x,y)=\sup_{i<\omega}2^{-i}\underset{z}{\sup}|\varphi_{i}(x,z)-\varphi_{i}(y,z)|.
\]
This is a formula and so is continuous on $S_{2}(\varnothing)$.
It is an ultra-psuedo-metric since it is the supremum of a family
of ultra-pseudo-metrics. It also clearly only takes on values in the
set $\{0\}\cup\{2^{-i}|i<\omega\}$.

Since it vanishes on the diagonal, it must be uniformly dominated
by $d$. Let $p$ be a $2$-type not on the diagonal. There is some
clopen neighborhood $Q$ of the diagonal which does not contain $p$.
Let $\varphi_{i}$ be its indicator function. We then have $\rho^{p}\geq2^{-i}$,
so $\rho$ does not vanish anywhere besides the diagonal, so by compactness
it must be uniformly equivalent to $d$.

\emph{$\text{(vi)}\Rightarrow\text{(v)}$.} This is immediate.

\emph{$\text{(v)}\Rightarrow\text{(iii)}$.} For any $p\in U\subseteq S_{n}(\varnothing)$,
let $D$ be a definable set such that $p\in D\subseteq U$. Note that
a set definable relative to $d$ is definable relative to any metric
uniformly equivalent to $d$. By compactness there is an $\varepsilon>0$
small enough that $D^{\rho\leq\varepsilon}\subseteq U$. Since there
are arbitrarily small gaps in the distance set we can find $0<\delta<\varepsilon$
such that $\delta\notin\rho(T)$ and we have that $D^{\rho\leq\varepsilon}=D^{\rho<\varepsilon}$
is a clopen set. Therefore $S_{n}(\varnothing)$ has a basis of clopen
sets and is totally disconnected.
\end{proof}
\begin{lem}
If $X$ is a metrically separable ultrametric space, then its distance
set, $d(X)=\{d(x,y)|x,y\in X\}$, is countable. 
\end{lem}

\begin{proof}
Let $\{x_{i}\}_{i<\omega}$ be a countable dense subset of $X$. There
are only countably many distances $d(x_{i},x_{j})$. By the ultrametric
inequality if we choose $x_{i}$ close enough to $y$ and $x_{j}$
close enough to $z$, $d(y,z)=d(x_{i},x_{j})$.
\end{proof}
\begin{lem}
\label{lem:not-trivial-ultra}If $T$ is an ultrametric theory, then
the underlying metric of each $S_{n}(A)$ is an ultrametric. (Recall
that the metric on $n$-tuples is defined as the maximum of the componentwise
distances.)
\end{lem}

\begin{proof}
Let $p,q,r\in S_{n}(A)$. We can find a model $\mathfrak{M}$ containing
$n$-tuples $\overline{a}\models p,$ $\overline{b}\models q$, and
$\overline{c}\models r$ such that $d(\overline{a},\overline{b})=d(p,q)$
and $d(\overline{b},\overline{c})=d(q,r)$. Then $d(p,r)\leq d(\overline{a},\overline{c})\leq \imax {d(\overline{a},\overline{b})}{  d(\overline{b},\overline{c})}=\imax{d(p,q)}{ d(q,r)}$. 
\end{proof}
Note that Lemma \ref{lem:not-trivial-ultra} isn't entirely trivial
in the sense that not all metric properties of models of theories
transfer directly to the type spaces of that theory.
\begin{cor}
\label{cor:bad-ultrametric}If $T$ is an ultrametric theory whose
distance set, 
\[d(T) \coloneqq \{d(a,b)|a,b\in\mathfrak{M}\models T\},\]
\editcom{Added $\coloneqq$.} is somewhere
dense, then its type space $S_{2}(T)$ is not small (and in particular
$T$ is not $\omega$-stable).
\end{cor}

\begin{proof}
Assume that $d(T)$ is dense in some interval $(r,s)$. Assume without
loss that $0<r$. Since $d(T)$ is dense in $(r,s)$, for every $t\in(r,s)$
and every $\varepsilon>0$ we can find $a,b\in\mathfrak{M}$ for some
model such that $|t-d(a,b)|<\varepsilon$, so by compactness there
exists $c,e\in\mathfrak{N}$ such that $d(c,e)=t$. This means that
the set $\cset{|d(x,y)-t|=0}$ is non-empty in $S_{2}(\varnothing)$ for
every $t\in(r,s)$. Let $p$ and $q$ be $2$-types such that $d^{p}=u>v=d^{q}$.
Let $ab\models p$ and $ce\models q$ in some model. If $d(ab,ce)<v$,
then we have that $d(a,b)\leq \max\{d(a,c), d(c,e) , d(e,b)\}\leq v$,
which is a contradiction, so we have $d(ab,ce)\geq v$.

Since $r>0$, if we pick a type $p_{t}\in\cset{|d(x,y)-t|=0}$ for each
$t\in(r,s)$, the set $\{p_{r}|t\in(r,s)\}$ will be uncountable and
$({>}r)$-separated, so $S_{2}(T)$ is not small.
\end{proof}
It is possible for an ultrametric theory $T$ with somewhere dense
distance set to have $S_{1}(T)$ be a single point (even if we require that $T$ be superstable).\editcom{Changed second sentence into parenthetical comment.} 
\begin{cor}
Every totally transcendental ultrametric theory $T$ has totally disconnected
type spaces.
\end{cor}

\begin{proof}
Every countable reduct $T_{0}$ of $T$ is $\omega$-stable, so by
Corollary \ref{cor:bad-ultrametric}, $d(T_{0})$ is nowhere dense
in $[0,1]$, so Theorem \ref{thm:zero-dim-type-spaces} applies and
$T_{0}$ has totally disconnected type spaces. Since this is true
for every countable reduct, $\{0,1\}$-valued formulas are logically
complete in the full theory, and $T$ has totally disconnected type
spaces as well.
\end{proof}
As a word of warning, not all ultrametric theories have totally disconnected or
even dictionaric type spaces.
\begin{cor}\label{cor:strong_min_and_cat_in_cont:2}
Every theory $T$ with totally disconnected type spaces is bi-interpretable
with a many-sorted discrete theory $T_{\mathrm{dis}}$.
\end{cor}

\begin{proof}
Since $T$ has totally disconnected type spaces, it is dictionaric.
By Theorem \ref{thm:zero-dim-type-spaces}, we can find an ultrametric
$\rho$ uniformly equivalent to $d$ such that $\rho(T)\subseteq\{0\}\cup\{2^{-i}|i<\omega\}$.
This means that we can define a sequence of $\{0,1\}$-valued equivalence
relations $E_{k}(x,y)=2^{k+1}\imin{2^{-k}\dotdiv\rho(x,y)}{1}$. 

Each imaginary $T_{k}=T/E_{k}$ is metrically discrete and dictionaric
by corollary \ref{cor:dict-hered}, so for every parameter set $A$,
$S_{n}((T/E_{k})_{A})$ is totally disconnected and metrically discrete
(the distance set of a type space is always a subset of the distance
set of the theory). If we form a many-sorted theory $\bigsqcup_{k<\omega}T/E_{k}$
with all the definable relations between different imaginaries, all
of the type spaces will still be totally disconnected and metrically
discrete, because for any finite set $T/E_{k_{0}},...,T/E_{k_{n-1}}$,
the largest $k_{i}$ theory has surjectively defined maps to the others,
so the mixed-sort type spaces are discrete quotients of some $S_{n}(T/E_{k})$
and are thus totally disconnected and metrically discrete.

To see that this is a bi-interpretation note that the discrete theory
has an $\omega$-ary imaginary consisting of the direct limit $\underset{\rightarrow}{\lim}T/E_{k}$; specifically we can consider the imaginary sort $\prod_{k<\omega} T/E_{k}$, and then note that the set 
\[D=\left\{ \alpha \in \prod_{k<\omega} T/E_{k} \, : \, (\forall k < \omega) \alpha(k) E_k \alpha(k+1) \right\}\]
is definable. To see that $D$ is definable, note that for each $k<\omega$, the set 
\[D_k = \left\{ \alpha \in \prod_{k<\omega} T/E_{\ell} \, : \, (\forall \ell < k) \alpha(\ell) E_\ell \alpha(\ell+1) \right\}\]
is clopen and satisfies $D\subseteq D_k \subseteq D^{<2^{-k}}$. $D$ is identical to the home sort of the original theory. 
\end{proof}

\editcom{Added sentence.} Now that we know that $\omega$-stable ultrametric theories are discrete theories in disguise, a Baldwin-Lachlan characterization is immediate.

\begin{cor}
  If $T$ is an ultrametric theory or has totally disconnected type spaces, then $T$ is \insep\ categorical if and only if it is $\omega$-stable and has no imaginary Vaughtian pairs.
\end{cor}

\subsubsection{When Can We Find Strongly Minimal Imaginaries over the Prime Model?}
\label{sec:when-stronk-im}

We will establish in Section~\ref{sec:theory-with-strongly} that there are \insep\ categorical theories with strongly minimal sets only over models of sufficiently high dimension. Such theories fail the assumptions of the first part of Theorem~\ref{thm:main-thm}, but what is unclear at the moment is the possibility of an \insep\ categorical theory that has a strongly minimal imaginary over some models but not others. In this section we will present some partial progress on this question, and show that this cannot happen if the strongly minimal set is discrete.

\begin{prop}
If $T$ is a dictionaric theory with no imaginary Vaughtian pairs such that for some model $\mathfrak{M}$ there is an infinite discrete $\mathfrak{M}$-definable imaginary, then for every model $\mathfrak{N}$ there is an infinite discrete $\mathfrak{N}$-definable imaginary. 
\end{prop}

\begin{proof}
Assume that $T$ is countable. By Lemma~\ref{lem:imaginary-norm-form} we may assume that the infinite discrete imaginary over $\mathfrak{M}$ is a definable subset $D(x,\overline{a})$ of some $\varnothing$-definable imaginary $I$. Assume without loss of generality that the metric on $D(-,\overline{a})$ is $\{0,1\}$-valued.

Since $D(x,\overline{a})$ is a distance predicate, part of $\mathrm{tp}(\overline{a})$ says that 
\[\chi_0(\overline{a})=\sup_x \inf_y \imax{D(y,\overline{a})}{ |D(x,\overline{a})-d(x,y)|}\] 
and
\[\chi_1(\overline{a}) = \sup_x \left|D(x,\overline{a})-\inf_y \imin{D(y,\overline{a})+d(x,y)}{1}\right|\]
both vanish.\editcom{Added `both vanish.'} (These are the axioms for distance predicates given in \cite{MTFMS} right before Theorem 9.12.) $\tp(\bar{a})$ also says that\editcom{Restructured this sentence too.} 
\[\eta(\overline{a}) = \sup_{x,y} \imin{d(x,y)}{ 1 \dotdiv d(x,y)} \dotdiv 4(D(x,\overline{a})+D(y,\overline{a}))\]
vanishes. (This is just  $\sup_{x,y \in D(-,\overline{a})} \imin{d(x,\allowbreak y)}{ 1 \dotdiv d(x,y)}$ expanded out.)

Fix $\mathfrak{N}$ another model of $\mathfrak{M}$ and let $\{\overline{b}_i\}_{i < \omega}$ be a sequence of elements of $\mathfrak{N}$ such that $\mathrm{tp}(\overline{b}_i)\rightarrow \mathrm{tp}(\overline{a})$ (we can choose it to be a sequence rather than a net since $T$ is countable), such that in particular $\chi_{0}(\overline{b}_i),\chi_{1}(\overline{b}_i),\eta(\overline{b}_i)< 2^{-i}$ for every $i<\omega$. For each $i<\omega$ let $E_i$ be a $\overline{b}_i$-definable set such that $\cset{D(-,\overline{b}_i) < 2^{-i+1}}\supseteq  E_i \supseteq \tint  E_i \supseteq \cset{D(-,\overline{b}_i) \leq 2^{-i}}$.

For each $i<\omega$, we have that
\[\mathfrak{A} \models   \sup_{x,y \in E_i} \imin{d(x,y)}{ 1 \dotdiv d(x,y)} \dotdiv 8\cdot 2^{-i+1} = 0.\]
So in particular for sufficiently large $i<\omega$, for any $x,y \in E_i$, either $d(x,y) < \frac{1}{4}$ or $d(x,y)>\frac {3}{4}$, implying that we can define a $\{0,1\}$-valued equivalence relation on $E_i$ by $\rho(x,y) = \imin{2(d(x,y)\dotdiv \frac{1}{4})}{1}$ for sufficiently large $i<\omega$.

Now assume that for all $i<\omega$ for which $\rho$ defines an equivalence relation on $E_i$, there are\editcom{Changed to `there are' from `it has.'} finitely many $\rho$-equivalence classes in $E_i$.\editcom{Added `in $E_i$.'} Let $\mathfrak{A}\succ \mathfrak{N}$ be a proper elementary extension. Fix a non-principal ultrafilter $\mathcal{U}$ on $\omega$ and consider the ultraproduct $(\mathfrak{B},\mathfrak{N}^\prime,\overline{c}) = \prod_{i<\omega}(\mathfrak{A},\mathfrak{N},\overline{b}_i)/\mathcal{U}$, where $(\mathfrak{A},\mathfrak{N},\overline{b}_i)$ is a structure whose universe is $\mathfrak{A}$ with a distance predicate for $\mathfrak{N}$ and constants for $\overline{b}_i$. Note that $\mathfrak{B}$ is a proper elementary extension of $\mathfrak{N}^\prime$, because $\mathfrak{A}$ is uniformly a proper elementary extension of $\mathfrak{N}$ in the product. Clearly we have that $\mathrm{tp}(\overline{c})=\mathrm{tp}(\overline{a})$, so $D(-,\overline{c})$ is a definable set which is infinite and has a $\{0,1\}$-valued metric.

We want to argue that $\mathfrak{B} \succ \mathfrak{N}^\prime$ is a Vaughtian pair with regards to the set $D(-,\overline{c})$. Let $\alpha(i)$ be a sequence such that $\alpha(i)\in (\mathfrak{A},\mathfrak{N},\overline{b}_i)$ and such that the limit $\alpha / \mathcal{U}$ is in $D(-,\overline{c})$. This means that we must have $\lim_\mathcal{U} D(\alpha(i),\overline{b}_i)=0$. Pick $\varepsilon>0$. For each $\alpha(i)$ we have that there is a $\beta(i)$ such that $D(\beta(i),\overline{b}_i) < 2^{-i}$ and $|D(x,\overline{b}_i)-d(\alpha(i),\beta(i))|<2^{-i}$, so in particular $\beta(i) \in E_i(\mathfrak{A})$.  Since $E_i$ only has finitely many $\rho$-classes, there must be a $\gamma(i) \in  E_i(\mathfrak{N})$ such that $d(\beta(i),\gamma(i)) < \frac{1}{4}$, implying that $d(\alpha(i),\gamma(i)) < \frac{1}{4} + 2^{-i+1}$ and in particular $d(\alpha(i),\mathfrak{N})<\frac{1}{4} + 2^{-i+1}$. So then in the limit we have $d(\alpha/\mathcal{U},\mathfrak{N}^\prime) \leq \frac{1}{4}$, but since $D(-,\overline{c})$ has a $\{0,1\}$-valued metric this implies that $\alpha/\mathcal{U} \in D(\mathfrak{N}^\prime, \overline{c})$, so $(\mathfrak{B},\mathfrak{N}^\prime)$ is a Vaughtian pair.

Since $T$ has no imaginary Vaughtian pairs, this is a contradiction and some $E_i$ must have infinitely many $\rho$-classes, so then $E_i/\rho$ is an infinite discrete imaginary over $\mathfrak{N}$.
\end{proof}

\begin{cor}
If $T$ is an \insep\ categorical theory with a discrete strongly minimal imaginary over some set, then it has a discrete strongly minimal imaginary over the prime model.
\end{cor}
\begin{proof}
By the previous proposition there is an infinite discrete imaginary over the prime model, maybe not strongly minimal, but by $\omega$-stability we can find a minimal set in it which must by no imaginary Vaughtian pairs be strongly minimal.
\end{proof}

\subsection{The Number of Separable Models}

Regarding the question of the number of separable models, we immediately get the following, generalizing a result originally due to Morley in discrete logic \cite{DBLP:journals/jsyml/Morley70}.

\begin{prop}
If $T$ is an \insep\ categorical theory with a minimal set or imaginary definable over the prime model, then $T$ has at most countably many separable models.
\end{prop}
\begin{proof}
By Lemma~\ref{lem:imaginary-norm-form} we may assume that $D$ is a strongly minimal set (rather than a strongly minimal imaginary) by appending a $\varnothing$-definable imaginary sort.\editcom{Removed `So.'} Let $D(x;\overline{a})$ be a minimal set definable over the prime model. Since $T$ is dictionaric and has no Vaughtian pairs, this is a strongly minimal set.\editcom{Removed `Now.'} Let $\mathfrak{A}$ be any separable model of $T$. The prime model $\mathfrak{M}$ embeds into $\mathfrak{A}$ so we can find $\overline{b}\in A$ with $\overline{b} \equiv \overline{a}$ and we get that $D(x;\overline{b})$ is a strongly minimal set. It must be the case that $\mathrm{dim}(D(\mathfrak{A};\overline{b}))\leq \omega$, since $\mathfrak{A}$ is separable. Assume that we chose $\overline{b}$ so that $\mathrm{dim}(D(\mathfrak{A};\overline{b}))\leq \omega$ among parameters with the same type as $\overline{b}$. If $\mathfrak{B}$ is any other separable model we may find $\overline{c}\in \mathfrak{B}$ with $\overline{c}\equiv \overline{a}$ and $\mathrm{dim}(D(\mathfrak{B};\overline{c}))$ minimal. If $\mathrm{dim}(D(\mathfrak{A};\overline{b}))=\mathrm{dim}(D(\mathfrak{B};\overline{c}))$, then we get an elementary map $f:D(\mathfrak{A};\overline{b})\cup\overline{b}\rightarrow D(\mathfrak{B};\overline{c})\cup\overline{c}$, which extends to an isomorphism between $\mathfrak{A}$ and $\mathfrak{B}$, since $T$ has no (imaginary) Vaughtian pairs. Therefore since there are only countably many possible dimensions for separable models, we get that there are at most countably many separable models.
\end{proof}

\begin{lem}\label{Lem:Prime-over-approx}
If $T$ is an \insep\ categorical theory and $(D,P)$ is a $\varnothing$-definable approximately minimal pair, then every model $\mathfrak{M}$ of $T$ is prime over $\cset{P(\mathfrak{M})\brackconv}$.
\end{lem}
\begin{proof}
Assume that some model $\mathfrak{A}$ is not prime over $\cset{P(\mathfrak{A})\brackconv}$. Let $\mathfrak{B}\prec \mathfrak{A}$ be prime over $\cset{P(\mathfrak{A})\brackconv}$, so in particular $\cset{P(\mathfrak{A})\brackconv}=\cset{P(\mathfrak{B})\brackconv}$. Let $\mathfrak{M}$ be the prime model of $T$ and find some embedding $\mathfrak{M}\preceq \mathfrak{B}$. Since $D$ is $\varnothing$-definable, we can find a strongly minimal $E\subseteq M$ definable over $\mathfrak{M}$. Note that every element of $E(\mathfrak{C})\smallsetminus \cset{P(\mathfrak{C})\brackconv}$, where $\mathfrak{C}$ is the monster model, is algebraic over $\mathfrak{M}$. This implies that $E(\mathfrak{B})=E(\mathfrak{A})$ as well, but then this is a Vaughtian pair, which is a contradiction. Therefore every model $\mathfrak{M}$ of $T$ is prime over $\mathfrak{M}$.
\end{proof}

The difficulty in characterizing the number of separable models of an \insep\ categorical theory seems to be related to the phenomenon of non-$d$-finite types, identified by Ben Yaacov and Usvyatsov in \cite{Yaacov2007}. Finitary types in continuous logic can behave analogously to $\omega$-types in discrete logic. Ben Yaacov and Usvyatsov identified a class of types they call $d$-finite\editcom{Removed `types.'} which behave like discrete finitary types. In their paper they were able to prove that a superstable theory with `enough uniformly $d$-finite types' (where uniformly $d$-finite is a technical strengthening of $d$-finite) has either $1$ or infinitely many separable models, whereas in general an $\omega$-stable continuous theory can have any finite number of separable models, including $2$ (they also showed that in the presence of `enough $d$-finite types,' a continuous theory cannot have exactly $2$ separable models).

So in some easy cases we get the full Baldwin-Lachlan theorem on the number of separable models of an \insep\ categorical theory:

\begin{thm}
If $T$ is an \insep\ categorical theory and any of the following occur, then $T$ has either $1$ or $\omega$ separable models.
\begin{itemize}
\item
$T$ has a $\varnothing$-definable approximately minimal pair (possibly in an imaginary).
\item
$T$ is ultrametric or has totally disconnected type spaces.
\item
$T$ has enough uniformly $d$-finite types.
\end{itemize}
\end{thm}
\begin{proof}
The only difficult case is the first one, but this is covered by Lemma \ref{Lem:Prime-over-approx}. The second case follows from the fact that such theories are interdefinable with many-sorted discrete theories, and the third case is a direct corollary of Theorem 4.7 in \cite{Yaacov2007}.
\end{proof}

The following example shows that we cannot hope to show that every model of an \insep\ categorical theory is exactly homogeneous, although at the moment they seem to always be approximately homogeneous.

\begin{ex} \label{ex:no-dim}
A totally categorical theory $T$ with a strongly minimal set such that for any strongly minimal set $D(x,\overline{a})$ in the home sort of the unique separable model $\mathfrak{M}$, for any $n\leq \omega$, there is a $\overline{b}\equiv\overline{a}$ such that $\mathrm{dim}(D(\mathfrak{M},\overline{b}))=n$.

\end{ex}
\begin{proof} [Verification]
The construction of this example is very similar to the construction of $T_\omega$ in Theorem \ref{thm:bad-example}, so we will only sketch the differences. Make the following two changes to the construction:

\begin{itemize}
\item
Define the score of an element $\alpha \in \mathfrak{W}$ by
\[s(\alpha) = \min \left\{ g(A):A\in G,\sum \alpha(A) = 0 \right\}.\]
\item
Remove the predicates $C_v$ from the language.
\end{itemize}

A similar analysis gives that the theory of this structure has a $\varnothing$-definable strongly minimal imaginary $I$ whose dimension in the prime model is infinite. Furthermore types over the prime model $\mathfrak{M}$ in the home sort are strongly minimal if and only if they correspond to elements $\alpha$ where $\alpha(0),\alpha(2),\alpha(4),\dots$ enumerates a linearly independent set and only one $\alpha(2i+1)$ is not realized in $I(\mathfrak{M})$. 

If, for example, $\alpha(0)$ is not realized in $I(\mathfrak{M})$, then the we can choose $\beta(2)\beta(4)\dots \equiv \alpha(2)\alpha(4)\dots$ such that the subspace of $I(\mathfrak{M})$ spanned by $\beta(2)\beta(4)\dots$ has arbitrary codimension $\leq \omega$, giving the required property for the theory of this structure.
\end{proof}

The problem in Example \ref{ex:no-dim} is that the strongly minimal set requires non-$d$-finite parameters. We can get some traction by assuming that we have a strongly minimal set definable over a $d$-finite tuple of parameters, but adapting the proof of the Baldwin-Lachlan theorem any further than this is unclear at the moment.

\begin{prop}
Suppose that $p(x,\overline{a})$ is a strongly minimal type, $\mathrm{tp}(\overline{a})$
is $d$-finite, and $T$ is a dictionaric theory. If $\mathfrak{M}$
is a model containing $\overline{a}$ and there is $n<\omega$ such
that for every $\varepsilon>0$, there is $\overline{b}\in\mathfrak{M}$
with $d(\overline{a},\overline{b})<\varepsilon$ and $\mathrm{dim}(p(\mathfrak{M},\overline{b}))\geq n$,
then $\mathrm{\mathrm{dim}(}p(\mathfrak{M},\overline{a}))\geq n$.
\end{prop}

\begin{proof}
Let $q(\overline{y},\overline{a})$ be the type of an $n$-element
independent sequence in $p(x,\overline{a})$. Find $\varepsilon>0$
small enough that $B_{\leq2\varepsilon}(c)\cap p(\mathfrak{C},\overline{a})\subseteq\mathrm{acl}(c\overline{a})$
for every $c\models p(x,\overline{a})$. (Note that this works because
we can define an approximately strongly minimal pair over $\overline{a}$
pointing at $p(x,\overline{a})$; then, by examining the definition
of approximately strongly minimal pair we get that for any $c\models p(x,\overline{a})$
there is some $\varepsilon>0$ such that the inclusion $B_{\leq2\varepsilon}(c)\cap p(\mathfrak{C},\overline{a})\subseteq\mathrm{acl}(c\overline{a})$ holds, but then by automorphisms
of $\mathfrak{C}$ this works for every $c\models p(x,\overline{a})$.)
Fix an approximately strongly minimal pair $(D(x,\overline{a}),P(x,\overline{a}))$
pointing at $p(x,\overline{a})$ and find a $\gamma>0$ small enough
that if $\overline{a}\equiv\overline{b}$ and $d(\overline{a},\overline{b})<\gamma$,
then $d_{H}(D(\mathfrak{C},\overline{a}),D(\mathfrak{C},\overline{b}))<\varepsilon$.

By $d$-finiteness (of $\mathrm{tp}(\overline{b})$), we can find a $\delta>0$
such that for any $\overline{b}$ with $d(\overline{a},\overline{b})<\delta$
and any $\overline{c}\models q(\overline{y},\overline{b})$, we can
find $\overline{e}$ such that $\overline{a}\overline{e}\equiv\overline{b}\overline{c}$
and $d(\overline{e},\overline{c})<\varepsilon$.

Let $\overline{b}\in\mathfrak{M}$ be such that $\overline{a}\equiv\overline{b}$
and $d(\overline{a},\overline{b})<\imin{\delta}{\gamma}$. By construction
we have that $d_{H}(D(\mathfrak{M},\overline{a}),D(\mathfrak{M},\overline{b}))<\varepsilon$.
Let $\overline{c}$ be an independent tuple of length $n$ of realizations
of $p(x,\overline{b})$. Let $\overline{e}^{0}$ be a tuple of elements
of $D(\mathfrak{M},\overline{a})$ such that $d(\overline{c},\overline{e}^{0})<\varepsilon$.
By construction we also have that there is a tuple $\overline{e}$
(in the monster model) such that $\overline{a}\overline{e}\equiv\overline{b}\overline{c}$
and $d(\overline{e},\overline{c})<\varepsilon$. For each $i<n$,
we have that $d(e_{i}^{0},e_{i})\leq d(e_{i}^{0},c_{i})+d(c_{i},e_{i})<\varepsilon+\varepsilon$.
Therefore $e_{i}\in B_{\leq2\varepsilon}(e_{i}^{0})\cap p(\mathfrak{C},\overline{a})$
for each $i$, but this implies by construction that $e_{i}\in\mathrm{acl}(e_{i}^{0}\overline{a})$,
so in particular $e_{i}\in\mathfrak{M}$. Therefore $\overline{e}\in\mathfrak{M}$
and we have that $\mathrm{dim}(p(\mathfrak{M},\overline{a}))\geq n$.
\end{proof}

\begin{cor}
If $\mathrm{tp}(\overline{a})$ is $d$-finite, $p(x,\overline{a})$
is strongly minimal, and $T$ is dictionaric, then in any approximately
homogeneous model $\mathfrak{M}$, for any $\overline{a},\overline{b}\in\mathfrak{M}$
with $\overline{a}\equiv\overline{b}$, we have that $\mathrm{dim}(p(\mathfrak{M},\overline{a}))=\mathrm{dim}(p(\mathfrak{M},\overline{b}))$.\editcom{Added `we have that.'}
\end{cor}

For the following, recall that in continuous logic there is an $\omega$-stable theory with precisely $2$ separable models.\editcom{Added second sentence.} Also note that we do not know whether or not \insep\ categorical theories have enough $d$-finite types (in the technical sense of \cite{Yaacov2007}).

\begin{cor}
If $T$ is an \insep\ categorical theory with a strongly
minimal type definable over a $d$-finite tuple in the prime model, then $T$
does not have precisely $2$ separable models.
\end{cor}

\begin{proof}
Assume that $T$ is not $\aleph_0$-categorical. The prime model and the approximately $\aleph_0$-saturated model
are both approximately $\aleph_0$-homogeneous by Proposition \ref{Prop:Prime-Hom} and Fact 1.5 in \cite{Yaacov2007}, respectively, so if $D(x;\overline{a})$ is a strongly minimal set with $\mathrm{tp}(\overline{a})$ atomic and $d$-finite, then for any $\overline{b}\equiv \overline{a}$, $\mathrm{dim}(D(\mathfrak{M};\overline{a}))=\mathrm{dim}(D(\mathfrak{M};\overline{b}))$ where $\mathfrak{M}$ is either prime or approximately $\aleph_0$-saturated. We know that the approximately $\aleph_0$-saturated model must have $\mathrm{dim}(D(\mathfrak{M};\overline{a}))=\omega$, so we must have that $\mathrm{dim}(D(\mathfrak{N};\overline{a})) = n <\omega$ where $\mathfrak{N}$ is the prime model (otherwise $T$ would be $\aleph_0$-categorical).  

All we need to do is argue that there is a model $\mathfrak{A}\succ \mathfrak{N}$ which is neither prime nor approximately $\aleph_0$-saturated. Let $b$ realize the non-algebraic type over $\mathfrak{N}$ in $D(x;\overline{a})$ and let $\mathfrak{A}$ be atomic over $\mathfrak{N}b$. Clearly $\mathrm{dim}(D(\mathfrak{A};\overline{a}))>n$. We need to argue that it is $n+1$.

By Lemma 4.5 in \cite{Yaacov2007} if $\mathrm{dim}(D(\mathfrak{A};\overline{a}))>n+1$, then there is some $c$ realizing the non-algebraic type in $D(x;\overline{a})$ over $\mathfrak{A}b$, implying that $\mathrm{tp}(c/\mathfrak{A})$ forks over $\mathfrak{A}b$, but this can only happen if $c \in \mathrm{acl}(\mathfrak{A}b)$, contradicting that $c$ realizes the strongly minimal type over $\mathfrak{A}b$. Therefore $\mathrm{dim}(D(\mathfrak{A};\overline{a}))=n+1$ and $\mathfrak{A}$ is neither prime nor approximately $\aleph_0$-saturated and $T$ has at least $3$ separable models.
\end{proof}



\subsection{Counterexamples Involving Strongly Minimal Sets}
\label{sec:stronk-min-counterexa}

Here we present our two technical counterexamples relating to theories with strongly minimal sets.

\subsubsection{A Theory with Strongly Minimal Sets, but Only over Models of Dimension $\geq n$}
\label{sec:theory-with-strongly}

One might hope that if an \insep\ categorical theory has a strongly minimal set over some model, then it has one over its prime model, but this is not true.

\begin{thm} \label{thm:bad-example}
  For any $n\leq \omega$ there is an \insep\ categorical theory $T_n$ with a $\varnothing$-definable strongly minimal imaginary $I$ such that\editcom{Added bullets.}
  \begin{itemize}
  \item $T_n$ has models with $\mathrm{dim}(I(\mathfrak{M}))=k$ for each $k\leq \omega$ but
  \item $T_n$ has a strongly minimal set over $\mathfrak{M}$ in the home sort if and only if $\mathrm{dim}(I(\mathfrak{M}))\geq n$.
  \end{itemize}
\end{thm}
\begin{proof}

Fix $n\leq \omega$. Let $V$ be the countable vector space over $\mathbb{F}_2$.  Let $\mathfrak{W} = V^\omega$ have the standard string ultrametric (i.e.\ $d(\alpha,\beta)=2^{-\ell}$ where $\ell$ is the length of the longest common initial segment of $\alpha$ and $\beta$). Let $f:V \rightarrow \omega$ be a fixed bijection. Let $G \subset \mathcal{P}(\{0 , 2, 4, \dots, 2n \})$ be the set of all non-empty subsets $A$ of $\{0 , 2, 4, \dots, 2n \}$ if $n<\omega$, and let $G \subset \mathcal{P}_\mathrm{fin}(\{0,2,4,\dots\})$ be the set of all non-empty finite subsets of $\{0,2,4,\dots\}$ if $n=\omega$. Let $g:  G \rightarrow \omega$ be a fixed injection. Let $\left<-,-\right> : \omega \times \omega \rightarrow \omega$ be a fixed pairing function.

For any $\alpha \in \mathfrak{W}$ we assign a score
\[s (\alpha) =  \min\left\{ \left< f\left( \sum \alpha(A) \right),g(A) \right> : A \in G\right\}.\]
Now let $D \subset \mathfrak{W}$ be the set of all $\alpha$ such that $\alpha(2i)=0$ for all $i > n$ and $\alpha(2i+1) = 0$ for all $i \leq s(\alpha)$.

\textit{Claim:} $D$ is a $\mathfrak{W}$-definable set.

\textit{Proof of claim.} For each $k<\omega$ let $D_k \subseteq \mathfrak{W}^\omega$ be the set of all $\alpha$ satisfying:

\begin{itemize}
\item
For any $i$ with $i > n$ and $2i < k$, the formula $\alpha(2i) = 0$.
\item
For any $i < \omega$ with $2i + 1 < k$, the formula 
\[\alpha(2i+1) \neq 0 \rightarrow \bigvee_{\left<v , A  \right> \leq i} \sum(\alpha(A)) = v.\]
%
\end{itemize}

In the language of $\mathfrak{W}$ we can say $\alpha(j)=v$ for any $j<\omega$ and $v\in V$ with a $\{0,1\}$-valued formula. Likewise note that there are only finitely many pairs $v,A$ with $\left<v,A \right> \leq i$, so the disjunction in the second family of formulas is first-order. So we have that each $D_k$ is definable by a $\{0,1\}$-valued formula.

It is clear that $D_k \supseteq D_{k+1}$ and $D(\mathfrak{W}) = \bigcap_{k<\omega} D_k(\mathfrak{W})$, so we may regard $D$ as the $\mathfrak{W}$-zeroset $\bigcap_{k<\omega} D_k$. Fix $\ell <\omega$ and find $m<\omega$ such that for any $v \in V$ with $f(v) \leq \ell$ and any $A \in G$ with $\max A \leq \ell$ and $\left<v , A \right> \leq m$.  Fix $\alpha \in D_{2m+1}$. Now find $\beta$ such that $\beta \upharpoonright \ell = \alpha \upharpoonright \ell$ and $\beta(i) = 0$ for all $i\geq \ell$, so in particular $d(\alpha,\beta) \leq 2^{-\ell}$. Now we have that $\beta \in D(\mathfrak{W})$, so we know that $d(\alpha,D(\mathfrak{W})) \leq 2^{-\ell}$ for all $\alpha \in D_{2m+1}$.

Since we can do this for arbitrarily large $\ell<\omega$ we can define a distance predicate for $D$ and $D$ is a definable set. \hfill $\square_{\text{claim}}$

Now consider the following set of $\{0,1\}$-valued formulas:

\begin{itemize}
\item
For each $v\in V$, $C_v(\alpha) = 0$ if and only if $\alpha(0)=v$.
\item
$P(\alpha,\beta,\gamma) = 0$ if and only if $\alpha(0) = \beta(0) + \gamma(0)$.
\item
For any even $k < \omega$, $Q_k (\alpha,\beta) = 0$ if and only if $\alpha(0) = \beta(k)$.
\item
For any odd $k < \omega$, $P_k(\alpha,\beta,\gamma) = 0$ if and only if $d(\beta,\gamma) \leq 2^{-k}$ and $\alpha(0) = \beta(k) + \gamma(k)$.
\end{itemize}

Now let $\mathfrak{A}= \left< D(\mathfrak{W}),\{C_v\}_{v\in V}, P, \{Q_{2k}\}_{k<\omega},\{R_{2k+1}\}_{k<\omega} \right>$ and $T_n = \mathrm{Th}(\mathfrak{A})$. As a reduct of $\mathrm{Th}(\mathfrak{W})$, it's clear that $T_n$ is $\omega$-stable. Furthermore since $\mathrm{Th}(\mathfrak{W})$ is \insep\ categorical we can easily construct an $\aleph_1$-saturated model of $T_n$. Let $V^\prime \supset V$ be the $\mathbb{F}_2$ vector space with dimension $\aleph_1$, let $\mathfrak{W}^\prime = (V^\prime)^\omega$, and let $\mathfrak{B} \succ \mathfrak{A}$ be the corresponding model of $T_n$. It's clear that $\mathfrak{W}^\prime$ is $\aleph_1$-saturated over $\mathfrak{W}$, so $\mathfrak{B}$ is $\aleph_1$-saturated as a model of $T_n$.

We can extend the definition of $s(\alpha)$ to cover $\alpha \in \mathfrak{W}^\prime$ if we allow for $s(\alpha) = \infty = \min \varnothing$. Then we can check, by the definitions of $D_k$, that $D(\mathfrak{W}^\prime)$ is precisely the set of $\alpha \in \mathfrak{W}^\prime$ such that $\alpha(2i)=0$ for all $i > n$ and $\alpha(2i+1)=0$ for all $i\leq s(\alpha)$ (where $i \leq \infty$ for all  $i<\omega$).

\textit{Claim:} $T_n$ is \insep\ categorical with a $\varnothing$-definable strongly minimal imaginary.

\textit{Proof of claim.} If we let $E$ denote the $\{0,1\}$-valued equivalence relation given by $E(x,y) = 2(d(x,y) \dotdiv \frac{1}{2})$, then it's clear than the quotient $H/E$ of the home sort $H$ by $E$ is a $\varnothing$-definable strongly minimal imaginary which is equivalent to a $\mathbb{F}_2$ vector space with constants for elements of the prime model. Now we just need to show that $T_n$ has no Vaughtian pairs over $H/E$. Suppose that $\mathfrak{M} \prec \mathfrak{N}$ are models of $T_n$ such that $H(\mathfrak{M})/E = H(\mathfrak{N})/E$. Let $\alpha$ be in $\mathfrak{N}\smallsetminus \mathfrak{M}$ and assume that $d(\alpha,\mathfrak{M})=2^{-k}$. Find $\beta \in \mathfrak{M}$ such that $d(\alpha, \beta) = 2^{-k}$ (this exists because the distance set is reverse well-ordered).

We may assume that $\mathfrak{M}$ and $\mathfrak{N}$ are separated and furthermore that $\mathfrak{N} \prec \mathfrak{B}$, so we can regard the elements of $\mathfrak{M}$ and $\mathfrak{N}$ as elements of $(V^\prime)^\omega$. In particular we have that for all $\alpha \in \mathfrak{N}$, $\alpha(2i)=0$ for all $i > n$  and $\alpha(2i+1) = 0$ for all $i \leq s(\alpha)$. 
 
Now we can see that $k$ cannot be even. If $k=2m$ for some $m$, then we can find $\gamma \in \mathfrak{M}$ such that $d(\alpha,\gamma) \leq 2^{-k-1}$ by finding the element of $H(\mathfrak{M})/E$ corresponding to $\alpha(2k)$ and replacing $\beta$ with $\gamma$ satisfying $\gamma(i)=\beta(i)$ for all $i \neq 2k$ and $\gamma(2k)=\alpha(2k)$.

So assume that $k=2m+1$ for some $m$. There is some element $c$ of $H(\mathfrak{N})/E$ such that for any $\sigma$ with $\sigma(0)=c$, we have $\mathfrak{N}\models P_{2m+1}(\sigma, \beta, \alpha)=0$. Therefore we can find such a $\sigma$ in $\mathfrak{M}$. But now there must be some $\alpha^\prime$ with $\mathfrak{M} \models P_{2m+1}(\sigma,\beta,\alpha^\prime)=0$, which implies that $d(\alpha,\alpha^\prime) \leq 2^{-k-1}$, which is also a contradiction.

Therefore there are no Vaughtian pairs over $H/E$ and $T_n$ is \insep\ categorical. \hfill $\square_{\text{claim}}$

Now it's clear that the models of $T_n$ are uniquely determined by $\textrm{dim}(H/E)$ and that every dimension $\geq 0$ is possible. So for each $k \leq \omega$ let $V_k \subset V^\prime$ be a vector space such that $\mathrm{dim}(V_k / V) = k$ and let $\mathfrak{M}_k \prec \mathfrak{B}$ be the corresponding model of $T_n$.  We need to characterize the types in $S_1(\mathfrak{M}_k)$.  Note that every type in $S_1(\mathfrak{M}_k)$ is realized in $\mathfrak{B}$, since it is $\aleph_1$-saturated.

Fix $k \leq \omega$. For each $\tau \in V_k^{<\omega}$, let $\zeta_\tau \in \mathfrak{M}_k$ be a fixed element satisfying $\zeta_\tau \supset \tau$.  For any $\alpha \in \mathfrak{W}^\prime$, assign an `index' $h(\alpha) = \left<\eta(\alpha) , X(\alpha) \right>$ by:

\begin{itemize}
\item
Let $\eta(\alpha) \in V_k ^ {\leq \omega}$ be the longest initial segment of $\alpha$ such that $\eta(\alpha)(i) \in V_k$ for every $i < |\eta(\alpha)|$.
\item
If $|\eta(\alpha)|$ is infinite then $X(\alpha) = \varnothing$.
\item
If $|\eta(\alpha)|$ is finite and even, then $X(\alpha) \subseteq \mathcal{P}_\mathrm{fin}(\omega) \times V_k$ is the set of pairs $\left<A , v \right>$ such that $\sum \alpha(2A) = v$.
\item
If $|\eta(\alpha)|$ is finite and odd, then $X(\alpha) \subseteq \mathcal{P}_\mathrm{fin}(\omega\cup\{-1\}) \times V_k$ is the set of pairs $\left<A , v \right>$ such that $\sum \alpha(2A) = v$, where we set $\alpha(-2) = \zeta_{\alpha \upharpoonright |\eta(\alpha)|} (|\eta(\alpha)|) + \alpha(|\eta(\alpha)|)$. (Recall that $\alpha \upharpoonright m$ is the sequence $\alpha(0),\allowbreak\alpha(1),\allowbreak\dots,\allowbreak\alpha(m-1)$, with length $m$.)

\end{itemize}

With the following two claims we will show that $h(\alpha)$ exactly captures $\mathrm{tp}(\alpha/ \mathfrak{M}_k)$.

\textit{Claim 1:} For any $\alpha,\beta \in \mathfrak{B}$, if $h(\alpha) \neq h(\beta)$ then $\alpha \not \equiv_{\mathfrak{M}_k} \beta$.

\textit{Claim 2:} For any $\alpha, \beta \in \mathfrak{B}$, if $h(\alpha) = h(\beta)$ then for any $\varepsilon>0$ there is an automorphism $\sigma$ of $\mathfrak{B}$, fixing $\mathfrak{M}_k$, such that $d(\sigma(\alpha),\beta) \leq \varepsilon$.

\textit{Proof of claim 1:} If either of $\eta(\alpha)$ or $\eta(\beta)$ is infinitely long, then the corresponding element is an element of $\mathfrak{M}_k$, so then $\alpha$ and $\beta$ clearly have different types over $\mathfrak{M}_k$. So assume that both $\eta(\alpha)$ and $\eta(\beta)$ are finite.

If $\eta(\alpha) \neq \eta(\beta)$, then there are elements of $\mathfrak{M}_k$ with different distances to $\alpha$ and $\beta$, implying that they have different types over $\mathfrak{M}_k$. So assume that $\eta(\alpha)=\eta(\beta)$.

If $\eta(\alpha) = \eta(\beta)$, then $X(\alpha) \neq X(\beta)$, and so this clearly gives an $\mathfrak{M}_k$-formula satisfied by $\alpha$ and not satisfied by $\beta$, so we have that $\alpha \not \equiv_{\mathfrak{M}_k} \beta$.

So we have that $h(\alpha) \neq h(\beta) \Rightarrow \alpha \not \equiv_{\mathfrak{M}_k} \beta$, as required. \hfill $\square_{\text{claim 1}}$

\textit{Proof of claim 2:} Assume that $h(\alpha) = h(\beta)$. If $\eta(\alpha) = \eta(\beta)$ is infinitely long, then $\alpha = \beta \in \mathfrak{M}_k$ and there is nothing to prove, so assume that $\eta(\alpha)=\eta(\beta)$ is finitely long.

First we will prove that there is an automorphism $\sigma_0$ of $\mathfrak{B}$ fixing $\mathfrak{M}_k$ such that $\sigma_0 (\alpha) (2i) = \beta(2i)$ for every $i<\omega$ and $\sigma_0(\alpha)(|\eta(\alpha)|) = \beta(|\eta(\alpha)|)$ if $|\eta(\alpha)|$ is odd.

If $|\eta(\alpha)|=|\eta(\beta)|$ is finite and even, then $X(\alpha)=X(\beta)$  is precisely the statement that 
\[\alpha(0)\alpha(2)\alpha(4)\dots \equiv_{V_k} \beta(0)\beta(2)\beta(4)\dots\]
 in the structure $V^\prime$. So we easily get an automorphism of $V^\prime$ taking $\alpha(0)\alpha(2)\alpha(4)\dots$ to $\beta(0)\beta(2)\beta(4)\dots$ fixing $V_k$. This extends to an automorphism of all of $\mathfrak{W}^\prime$ fixing $V_k^\omega$ which then induces an automorphism of $\mathfrak{B}$ fixing $\mathfrak{M}_k$ with the required property.

If $|\eta(\alpha)|=|\eta(\beta)|$ is finite and odd, then $X(\alpha) = X(\beta)$ is precisely the statement that
\[(\zeta_{\alpha \upharpoonright |\eta(\alpha)|} (|\eta(\alpha)|) + \alpha(|\eta(\alpha)|))\alpha(0)\alpha(2)\alpha(4)\dots\]
\[\equiv_{V_k} (\zeta_{\alpha \upharpoonright |\eta(\alpha)|} (|\eta(\alpha)|) + \beta(|\eta(\alpha)|))\beta(0)\beta(2)\beta(4)\dots\]
in the structure $V^\prime$. So we easily get an automorphism of $V^\prime$ taking
\[(\zeta_{\alpha \upharpoonright |\eta(\alpha)|} (|\eta(\alpha)|) + \alpha(|\eta(\alpha)|)),\alpha(0),\alpha(2),\alpha(4),\dots\]
 to
\[(\zeta_{\alpha \upharpoonright |\eta(\alpha)|} (|\eta(\alpha)|) + \beta(|\eta(\alpha)|)),\beta(0),\beta(2),\beta(4),\dots\]
fixing $V_k$. This extends to an automorphism of all of $\mathfrak{W}^\prime$ fixing $V_k^\omega$ which then induces an automorphism of $\mathfrak{B}$ fixing $\mathfrak{M}_k$ with the required property.

Now we need to argue that $h(\sigma_0(\alpha)) = h(\beta)$. If $|\eta(\alpha)|=|\eta(\beta)|$ is even, then there is nothing to prove. If $|\eta(\alpha)|=|\eta(\beta)|$ is odd, then the only thing to worry about is that we might have moved $\zeta_{\alpha \upharpoonright |\eta(\alpha)|} (|\eta(\alpha)|)$, but this is determined by $\alpha \upharpoonright |\eta(\alpha)| = \beta \upharpoonright |\eta(\beta)|$ which is unchanged and equal to $\sigma_0(\alpha) \upharpoonright |\eta(\sigma_0(\alpha))|$. So we have $h(\sigma_0(\alpha)) = h(\beta)$.

So now let $\gamma \in (V^\prime)^{\leq \omega}$ be the longest common initial segment of $\sigma_0(\alpha)$ and $\beta$, so that in particular $d(\sigma_0(\alpha),\beta)=2^{-|\gamma|}$. If $\gamma$ is infinitely long then we are done, so assume that $\gamma$ is finitely long. It must be the case that $|\gamma|$ is odd, since we ensured that $\sigma_0(\alpha)(2i) = \beta(2i)$ for every $i$. If $|\eta(\beta)|$ is odd, then it must be the case that $|\gamma|> |\eta(\beta)|$, since we ensured that $\sigma_0(\alpha)(|\eta(\beta)|)=\beta(|\eta(\beta)|)$. In any case we always have $|\gamma| > |\eta(\beta)|\geq 0$, so in particular we always have that the last element of $\gamma$, $\gamma(|\gamma|-1)$, is not in $V_k$.

Now define a map $\sigma_1 : \mathfrak{B} \rightarrow \mathfrak{B}$ by:
\begin{itemize}
\item
$\sigma_1(\chi) = \chi$ if $d(\chi,\beta) \geq 2^{-|\gamma|+1}$.

\item 
If $d(\chi,\beta) \leq 2^{-|\gamma|}$, so in particular $\chi(|\gamma|-1)=\beta(|\gamma|-1)$, then $\sigma_1(\chi)(j)=\chi(j)$, if $j \neq |\gamma|$, and $\sigma_1(\chi)(|\gamma|)=\chi(|\gamma|) + \beta(|\gamma|) + \sigma_0(\alpha(|\gamma|))$.

\end{itemize}
By checking the definition of $D$ we can see that $\sigma_1$ is a bijection on $\mathfrak{B}$. It's also clearly an isometric map. By checking the predicates in the language of $\mathfrak{B}$ we can see that $\sigma_1$ is an automorphism of $\mathfrak{B}$. Furthermore, clearly $\sigma_1(\alpha)(2i+1)=\beta(2i+1)$, so we have that $d(\sigma_1(\sigma_0(\alpha)),\beta) < d(\alpha,\beta)$, as required.

Now the only thing left to verify is that $\sigma_1$ fixes $\mathfrak{M}_k$. If $\chi$ is moved by $\sigma_1$, then $d(\gamma, \alpha) \leq 2^{-|\gamma|}$, so in particular $\chi(|\gamma|-1)=\beta(|\gamma|-1) \notin V_k$, so $\chi \notin \mathfrak{M}_k$. Therefore everything moved by $\sigma_1$ is not in $\mathfrak{M}_k$ and we have that $\sigma_1$ fixes $\mathfrak{M}_k$.

So by iterating the construction of $\sigma_1$ we can get the required automorphisms bringing $\alpha$ arbitrarily close to $\beta$. \hfill $\square_{\text{claim 2}}$

So we see that $\alpha \equiv_{\mathfrak{M}_k} \beta$ if and only if $h(\alpha) = h(\beta)$.

Now we need to determine the values of $h(\alpha)$ that correspond to a strongly minimal type. 

\textit{Claim:} If $\alpha(2i+1) \neq 0$ for some $i<\omega$, then $\mathrm{tp}(\alpha/\mathfrak{M}_k)$ is not strongly minimal.

\textit{Proof of claim.} If $\alpha \in \mathfrak{M}_k$, then $\mathrm{tp}(\alpha/\mathfrak{M}_k)$ is clearly not strongly minimal, so assume that $\alpha \notin \mathfrak{M}_k$. Since $\alpha \notin \mathfrak{M}_k$, we have that $\eta(\alpha)$ has finite length. Since $\alpha(2i+1) \neq 0$ for some $i<\omega$, we must have $s(\alpha) < \infty$ (this is determined by $\mathrm{tp}(\alpha/ \mathfrak{A})$ so \textit{a fortiori} it is determined by $\mathrm{tp}(\alpha/\mathfrak{M}_k)$). This implies that if $\beta$ has $\alpha(i) = \beta(i)$ for all $i \leq \imax{|\eta(\alpha)|} {(2s(\alpha)+1)}$ and $\alpha(2i)=\beta(2i)$ for all $i<\omega$, then $\alpha \equiv_{\mathfrak{M}_k} \beta$, implying that $\mathrm{\alpha/\mathfrak{M}_k}$ has many non-algebraic global extensions and is not strongly minimal. \hfill $\square_{\text{claim}}$

As a corollary of this we get that if $\alpha(0),\alpha(2),\dots,\alpha(2n)$ (where `$\alpha(0),\allowbreak\alpha(2),\dots,\alpha(2\omega)$' is understood to mean the infinite tuple $\alpha(0),\alpha(2),\dots$) are not linearly independent over $V$, then $\mathrm{tp}(\alpha/\mathfrak{M}_k)$ is not strongly minimal. Conversely we have that if $\alpha(0), \alpha(2), \dots, \alpha(2n)$ are linearly independent over $V$, then $\alpha(2i+1) = 0$ necessarily for every $i<\omega$, by the definition of $D$.

\textit{Claim:} If $\alpha(0), \alpha(2), \dots,\alpha(2n)$ are linearly independent over $V$ and
\[\mathrm{dim}(\{\alpha(0), \alpha(2), \dots,\alpha(2n)\} / \mathfrak{M}_k) > 1,\]
then $\mathrm{tp}(\alpha / \mathfrak{M}_k)$ is not strongly minimal.

\textit{Proof of claim.} If $\mathrm{dim}(\{\alpha(0), \alpha(2), \dots,\alpha(2n)\} / \mathfrak{M}_k) > 1$ then $\mathrm{tp}(\alpha / \mathfrak{M}_k)$ has more than one global non-algebraic extension and so is not strongly minimal. \hfill $\square_{\text{claim}}$

Clearly if $\mathrm{dim}(\{\alpha(0), \alpha(2), \dots,\alpha(2n)\} / \mathfrak{M}_k) = 0$, then we have that $\mathrm{tp}(\alpha/ \mathfrak{M}_k)$ is atomic and therefore not strongly minimal. So the only way for $\mathrm{tp}(\alpha / \mathfrak{M}_k)$ to be strongly minimal is if $\mathrm{dim}(\{\alpha(0), \alpha(2), \dots,\alpha(2n)\} / \mathfrak{M}_k)  = 1.$ 

As it turns out this is a precise characterization.

\textit{Claim:} $\mathrm{tp}(\alpha / \mathfrak{M}_k)$ is strongly minimal if and only if 
\[\mathrm{dim}(\{\alpha(0), \alpha(2), \dots,\alpha(2n)\} / \mathfrak{M}_k) = 1.\]
\textit{Proof of claim.} We have already shown the $\Rightarrow$ direction, so we just need to show the converse, but this is easy since $\mathrm{tp}(\alpha/ \mathfrak{M}_k)$ is clearly non-algebraic but also has a unique non-algebraic extension over any parameter set.\editcom{Changed wording.} \hfill $\square_{\text{claim}}$

Now finally we see that 
\[\mathrm{dim}(\{\alpha(0), \alpha(2), \dots,\alpha(2n)\} / \mathfrak{M}_k) = 1\] 
with
\[\mathrm{dim}(\{\alpha(0), \alpha(2), \dots,\alpha(2n)\} / V) = n+1\]
is possible if and only if $k \geq n$, and so $T_n$ has the required property.\editcom{Added `so.'}
\end{proof}

\subsubsection{A Counterexample to the Classical Characterization}

One might hope that somehow the condition that a theory
be $\omega$-stable and have no Vaughtian pairs might be strong enough
to ensure that a theory is \insep\ categorical, but it is not
so. The full Baldwin-Lachlan characterization fails in continuous
logic, even after strengthening the no Vaughtian pairs condition. 

In order to fully state the extent of the counterexample we are going to present in this section, we need a broad generalization of open sets and definable sets.
\begin{defn}
\leavevmode
\begin{enumerate}[label=(\roman*)]
\item A set  $X\subseteq S_n(A)$ is \emph{locatable} if $X \subseteq \tint  X^{<\varepsilon}$ for every $\varepsilon>0$. (Definable, open, and open-in-definable sets are all locatable.)

\item Given a countable parameter set $A$, a pair of models $\mathfrak{N}\succ \mathfrak{M} \supseteq A$, where $\mathfrak{N}$ is a proper elementary extension of $\mathfrak{M}$, is a \emph{locatable Vaughtian pair} if there is a locatable set $X \subseteq S_1(A)$ such that $X$ contains a non-algebraic type and $X(\mathfrak{M})=X(\mathfrak{N})$.
\end{enumerate}
\end{defn}

The most exotic kind of locatable sets that have explicitly appeared in this paper are open-in-definable sets. At the moment it is unknown whether or not \insep\ categorical theories can have locatable Vaughtian pairs, but it seems unlikely.

\begin{ex}
\label{exa:MAIN-COUNTEREXAMPLE} A countable $\omega$-stable theory with no locatable Vaughtian pairs that is not \insep\ categorical. 
\end{ex}

\begin{proof} [Verification]
Let $V$ be a countable vector space over a finite field $\mathbb{F}_{p}$.
Let $\mathfrak{M}$ be the structure whose universe is $V^{\omega}$
with the standard string ultrametric. For each $n,m<\omega$ with
$n\equiv m\,\mathrm{ (mod }\,2)$, let $P_{n,m}$ be a $\{0,1\}$-valued
quaternary relation such that $P_{n,m}(a,b,c,e)=0$ if and only if
$d(a,b)\leq2^{-n+1}$, $d(c,e)\leq2^{-m+1}$, and $a(n)-b(n)=c(m)-e(m)$.
For any fixed $n,m<\omega$ we have that if $d(a_{0},a_{1}),d(b_{0},b_{1})<2^{-n-1}$
and $d(c_{0},c_{1}),d(e_{0},e_{1})<2^{-m-1}$, then $P_{n,m}(a_{0},b_{0},c_{0},e_{0})=P_{n,m}(a_{1},b_{1},c_{1},e_{1})$,
so each $P_{n,m}$ is uniformly continuous.

We will show that $T$ is $\aleph_0$-categorical. To see that $T=\mathrm{Th}(\mathfrak{M})$ is $\omega$-stable and $\aleph_0$-categorical,
note that it is a reduct of an imaginary of $(V,+)$, which is $\omega$-stable
and $\aleph_0$-categorical. It is also a reduct of Example \ref{ex:no-strong-min}.

To see that $T$ is not \insep\ categorical, notice that if $W$
is the unique elementary extension of $V$ of cardinality $\aleph_1$,
then $(V\times W)^{\omega}$, $(W\times V)^{\omega}$, and $(W\times W)^{\omega}$
are the universes of three non-isomorphic models of $T$.

Finally to see that $T$ has no locatable Vaughtian pairs, we need
to analyze the structure of its type spaces more carefully. It is
sufficient to prove this statement\editcom{Changed `it' to `this statement.'} considering locatable subsets of $S_{1}(\mathfrak{M})$,
where $\mathfrak{M}$ is the unique separable model, since any countable
set of parameters can be found inside $\mathfrak{M}$. 

\textit{Claim:} Every non-algebraic type $p\in S_1(\mathfrak{M})$ is uniquely determined by the $k<\omega$ for which $d(p,\mathfrak{M})=2^{-k}$ and $\alpha\upharpoonright k$ for any $\alpha \in \mathfrak{M}$ with $d(\alpha,p)=2^{-k}$.

\emph{Proof of claim.} To see that this is true let $a,b$ be two elements of some $\mathfrak{N}\succ \mathfrak{M}$ with $d(a,\mathfrak{M})=d(b,\mathfrak{M})=2^{-k}$ and suppose that there are $\alpha, \beta \in \mathfrak{M}$ with $d(a,\alpha)=2^{-k}$ and $d(b,\beta)=2^{-k}$ and $\alpha\upharpoonright k = \beta \upharpoonright k$. The statement $\alpha\upharpoonright k = \beta \upharpoonright k$ is the same as saying that $d(\alpha,\beta)\leq 2^{-k}$, so since $T$ is ultrametric we have that $d(b,\alpha)\leq \imax{2^{-k}} { 2^{-k}} = 2^{-k}$, so $d(b,\alpha)=2^{-k}$. Therefore also $d(a,b)\leq 2^{-k}$. \hfill $\qed_{\text{claim}}$

$\mathfrak{N}$ can be written in the form $(U\times W)^\omega$ for some $\mathbb{F}_{p}$-vector spaces $U,W\succeq V$.

Assume that $d(a,b)=2^{-k}$. Either $a(k),b(k) \in U\smallsetminus V$ or $a(k),b(k) \in W \smallsetminus V$ (depending on whether $k$ is even or odd), and $a(k)\neq b(k)$, since $d(a,b)=2^{-k}$, so we can find an explicit automorphism of $U$ or $W$ fixing $V$ and taking $a(k)$ to $b(k)$. This extends to an automorphism $f$ of all of $\mathfrak{N}$ fixing $\mathfrak{M}$. So by replacing $a$ with $f(a)$, we may assume that $d(a,b)<2^{-k}$.

Assume that $d(a,b)<2^{-k}$. This implies that $a(k)=b(k)\notin V$. Let $\ell > k$ be the first such that $a(\ell) \neq b(\ell)$. Let $g:\mathfrak{N}\rightarrow \mathfrak{N}$ defined by the following: $g(c)=c$ if $c\upharpoonright \ell \neq a \upharpoonright \ell$, otherwise $g(c)(m)=c(m)$ for $m \neq \ell$ and $g(c)(\ell)=c(\ell) + b(\ell) - a(\ell)$. One can check that this is an automorphism of $\mathfrak{N}$ with the property that $d(g(a),b)<d(a,b)$. Note that $g$ fixes $\mathfrak{M}$.

By iterating, for any $k<\omega$ we can find an automorphism $h$ of $\mathfrak{N}$, fixing $\mathfrak{M}$. such that $d(h(a),b)\leq 2^{-k}$. Therefore $a\equiv_\mathfrak{M} b$, as required.

Let each non-algebraic type in $S_1(\mathfrak{M})$ be denoted by an element of $V^{<\omega}$.

Note that each non-algebraic type is metrically isolated. To see this, consider $\sigma_0, \sigma_1 \in V^{<\omega}$ with $\sigma_0 \neq \sigma_1$. Let $\tau$ be the longest common initial segment of $\sigma_0$ and $\sigma_1$. We have that $d(\sigma_0,\sigma_1)=2^{-|\tau|}$. In particular this implies that other types have distance at least $2^{-|\sigma|}$ to the type associated to some $\sigma \in V^{<\omega}$.

Now also note that for any $\sigma \in V^{<\omega}$, the sequence of types $\{\sigma \frown v\}_{v\in V}$ limits to $\sigma$, because the limiting type must be in the $(\leq 2^{-|\sigma|})$-ball whose center starts with $\sigma$, but it must have distance greater than $2^{-|\sigma|-1}$ from $\mathfrak{M}$, so $\sigma$ is the unique type that it can be.

Let $L\subseteq S_1(\mathfrak{M})$ be a locatable set containing a non-algebraic type $\sigma$. The claim is that for some $v\in V$, $L$ must contain $\sigma \frown v$ as well. To see this, note that $L^{<2^{-|\sigma|-2}}$ must be a neighborhood of $\sigma$, so it must contain infinitely many types of the form $\sigma \frown v$.\editcom{Replaced `of' with `of the form.'} Let $\sigma \frown v$ be in $L^{<2^{-|\sigma|-2}}$ and let $\tau \in L$ be such that $d(\sigma\frown v, \tau) < 2^{-|\sigma|-2}$. But since $\sigma$ is $(\geq 2^{-|\sigma|-1})$-metrically isolated, this implies that $\sigma \frown v = \tau$.

Now note that in any proper elementary extension of $\mathfrak{M} = (V\times V)^\omega$, one of the two copies of $V$ must grow, but this implies that one of the types $\sigma$ or $\sigma \frown v$ must be realized in the extension, so there cannot be a Vaughtian pair over $L$.
\end{proof}

Of course the theory of this structure has an imaginary Vaughtian pair.

As mentioned before,\editcom{Removed `note that.'} in discrete logic a superstable theory has no Vaughtian
pairs if and only if it has no imaginary Vaughtian pairs (although
the same is not true of strictly stable theories \cite{BOUSCAREN1989129}).\editcom{Removed `but.'}
This example shows that the same does not even hold for $\omega$-stable
theories in continuous logic.

The construction used in Theorem~\ref{thm:bad-example} and Example~\ref{exa:MAIN-COUNTEREXAMPLE} clearly rely very heavily on the ability to associate an entire copy of a strongly minimal set to a single element of another strongly minimal set. This seems like something that can only be done with discrete structures, which raises the following question.\editcom{Added paragraph.}

\begin{quest}
Does there exist a countable $\omega$-stable theory $T$
with no Vaughtian pairs that is not \insep\ categorical and does
not interpret a discrete strongly minimal set? A strongly minimal
set at all? In particular is there something like Example \ref{exa:MAIN-COUNTEREXAMPLE}
whose `underlying pregeometries' are not discrete? 
\end{quest}

\section{\Insep\ Categorical Expansions of Banach Space}
\label{sec:ind-sub-abs}

In this section we will consider expansions of Banach spaces.  We introduce the notion of an indiscernible subspace. An indiscernible subspace is a subspace in which types of tuples of elements only depend on their quantifier-free types in the reduct consisting of only the metric and the constant $\mathbf{0}$. Similarly to indiscernible sequences, indiscernible subspaces are always consistent with a Banach theory (with no stability assumption, see Theorem \ref{thm:exist}) but are not always present in every model. We will show that an indiscernible subspace always takes the form of an isometrically embedded real Hilbert space wherein the type of any tuple only depends on its quantifier-free type in the Hilbert space. The notion of an indiscernible subspace is of independent interest in the model theory of Banach and Hilbert structures, and in particular here we use it to improve the results of Shelah and Usvyatsov in the context of types in the full language (as opposed to $\Delta$-types). Specifically, in this context we give a shorter proof of Shelah and Usvyatsov's main result \cite[Prop.\ 4.13]{SHELAH2019106738}, we improve their result on the strong uniqueness of Morley sequences in minimal wide types \cite[Prop.\ 4.12]{SHELAH2019106738}, and we expand on their commentary on the ``induced structure'' of the span of a Morley sequence in a minimal wide type \cite[Rem.\ 5.6]{SHELAH2019106738}. This more restricted case is what is relevant to \insep\ categorical Banach theories, so our work is applicable to the problem of characterizing such theories. 

Finally, we present some relevant counterexamples and in particular we resolve (in the negative) the question of Shelah and Usvyatsov presented at the end of Section 5 of \cite{SHELAH2019106738}, in which they ask whether or not the span of a Morley sequence in a minimal wide type is always a type-definable set.

\subsection{Banach Theory Background}

For $K \in \{\mathbb{R}, \mathbb{C}\}$, we think of a $K$-Banach space $X$ as being a metric structure $\mathfrak{X}$ whose underlying set is the closed unit ball $B(X)$ of $X$ with metric $d(x,y) = \left\lVert x - y \right\rVert$.\footnote{For another equivalent approach, see \cite{MTFMS}, which encodes Banach structures as many-sorted metric structures with balls of various radii as different sorts.}  This structure is taken to have for each tuple $\bar{a} \in K$  an $|\bar{a}|$-ary predicate $s_{\bar{a}}(\bar{x}) = \left\lVert \sum_{i<|\bar{a}|} a_i x_i \right\rVert$, although we will always write this in the more standard form. 
Note that we evaluate this in $X$ even if $\sum_{i<|\bar{a}|} a_i x_i$ is not actually an element of the structure $\mathfrak{X}$.  For convenience, we will also have a constant for the zero vector, $\mathbf{0}$, and an $n$-ary function $\sigma_{\bar{a}}(\bar{x})$ such that $\sigma_{\bar{a}}(\bar{x}) = \sum_{i<|\bar{a}|} a_i x_i$ if it is in $B(X)$ and  $\sigma_{\bar{a}}(\bar{x}) = \frac{\sum_{i<|\bar{a}|} a_i x_i}{\left\lVert \sum_{i<|\bar{a}|} a_i x_i \right\rVert}$ otherwise. If $|a|\leq 1$, we will write $ax$ for $\sigma_{a}(x)$. Note that while this is an uncountable language, it is interdefinable with a countable reduct of it (restricting attention to rational elements of $K$). These structures capture the typical meaning of the ultraproduct of Banach spaces.  
We will often conflate $X$ and the metric structure $\mathfrak{X}$ in which we have encoded $X$.

\begin{defn}
A \emph{Banach (or Hilbert) structure} is a metric structure which is the expansion of a Banach (or Hilbert) space. A \emph{Banach (or Hilbert) theory} is the theory of such a structure. The adjectives \emph{real} and \emph{complex} refer to the scalar field $K$.
\end{defn}

$C^{\ast}$- and other Banach algebras are commonly studied examples of Banach structures that are not just Banach spaces.

A central problem in continuous logic is the characterization of \insep\ categorical countable theories, that is to say countable theories with a unique model in each uncountable density character. The analog of Morley's theorem was shown in continuous logic via related formalisms \cite{ben-yaacov_2005, Shelah2011}, but no satisfactory analog of the Baldwin-Lachlan theorem or its precise structural characterization of uncountably categorical discrete theories in terms of strongly minimal sets is known. Some progress in the specific case of Banach theories has been made in \cite{SHELAH2019106738}, in which Shelah and Usvyatsov introduce the notion of a wide type and the notion of a minimal wide type, which they argue is the correct analog of strongly minimal types in the context of \insep\ categorical Banach theories.

\begin{defn}
A type $p$ in a Banach theory is \emph{wide} if its set of realizations consistently contain the unit sphere of an infinite dimensional real subspace.

A type is \emph{minimal wide} if it is wide and has a unique wide extension to every set of parameters.
\end{defn}
In \cite{SHELAH2019106738}, Shelah and Usvyatsov were able to show that every Banach theory has wide complete types using the following classical concentration of measure results of Dvoretzky and Milman, which Shelah and Usvyatsov refer to as the Dvoretzky-Milman theorem.

\begin{fact}[Dvoretzky-Milman theorem] \label{fact:DM-thm}
Let $(X,\left\lVert \cdot \right\rVert)$ be an infinite dimensional real Banach space with unit sphere $S$ and let $f:S \rightarrow \mathbb{R}$ be a uniformly continuous function. For any $k<\omega$ and $\e > 0$, there exists a $k$-dimensional subspace $Y \subset X$ and a Euclidean norm\footnote{A norm $\vertiii{\cdot}$  is \emph{Euclidean} if it satisfies the parallelogram law, $2\vertiii{a}^2 + 2 \vertiii{b}^2 = \vertiii{a+b}^2 + \vertiii{a-b}^2,$ or, equivalently, if it is induced by an inner product.} $\vertiii{\cdot}$ on $Y$ such that for any $a,b \in S\cap Y$, we have $\vertiii{a} \leq \left\lVert a\right\rVert \leq (1 + \e)\vertiii{a}$ and $|f(a) - f(b)| < \e$.\footnote{Fact \ref{fact:DM-thm} without $f$ is (a form of) Dvoretzky's theorem.} 
\end{fact}


Shelah and Usvyatsov showed that in a stable Banach theory every wide type has a minimal wide extension (possibly over a larger set of parameters) and that every Morley sequence in a minimal wide type is an orthonormal basis of a subspace isometric to a real Hilbert space. Furthermore, they showed that in an \insep\ categorical Banach theory, every inseparable model is prime over a countable set of parameters and a Morley sequence in some minimal wide type, analogously to how a model of a discrete uncountably categorical theory is always prime over some finite set of parameters and a Morley sequence in some strongly minimal type.

The key ingredient to our present work is the following result, due to Milman. 
 It extends the Dvoretzky-Milman theorem in a manner analogous to the extension of the pigeonhole principle by Ramsey's theorem.\footnote{The original Dvoretzky-Milman result is often compared to Ramsey's theorem, such as when Gromov coined the term \emph{the Ramsey-Dvoretzky-Milman phenomenon} \cite{gromov1983}, but in the context of Fact \ref{fact:main} it is hard not to think of the $n=1$ case as being analogous to the pigeonhole principle and the $n>1$ cases as being analogous to Ramsey's theorem.}

\begin{defn}\label{defn:main-defn}
Let $(X,\left\lVert \cdot \right\rVert)$ be a Banach space. If $a_0,a_1,\dots,a_{n-1}$ and $b_0,b_1,\dots,\allowbreak b_{n-1}$ are ordered $n$-tuples of elements of $X$, we say that $\bar{a}$ and $\bar{b}$ are \emph{congruent} if $\left\lVert a_i - a_j\right\rVert=\left\lVert b_i - b_j \right\rVert$ for all $ i,j \leq n$, where we take $a_{n}=b_{n}=\mathbf{0}$. We will write this as $\bar{a} \cong \bar{b}$.
\end{defn}

\begin{fact}[\cite{zbMATH03376472}, Thm.\ 3] \label{fact:main}
Let $S^\infty$ be the unit sphere of a separable infinite dimensional real Hilbert space ${H}$ and let $f:(S^\infty)^n \rightarrow \mathbb{R}$ be a uniformly continuous function. For any $\varepsilon>0$ and any $k<\omega$ there exists a $k$-dimensional subspace $V$ of $H$ such that for any $a_0,a_1,\dots,a_{n-1},b_0,b_1,\dots,b_{n-1}\in S^\infty$ with $\bar{a} \cong \bar{b}$, $|f(\bar{a})-f(\bar{b})| < \varepsilon$.
\end{fact}

Note that the analogous result for inseparable Hilbert spaces follows immediately, by restricting attention to a separable infinite dimensional subspace. Also note that by using Dvoretzky's theorem and an easy compactness argument, Fact~\ref{fact:main} can be generalized to arbitrary infinite dimensional Banach spaces. Also note that while Fact~\ref{fact:DM-thm} and Fact~\ref{fact:main} are stated for real Banach spaces, analogous statements for complex Banach spaces can be easily derived from them.


 A modern proof of Fact \ref{fact:main} would go through the extreme amenability of the unitary group of an infinite dimensional Hilbert space endowed with the strong operator topology, or in other words the fact that any continuous action of this group on a compact Hausdorff space has a fixed point, which was originally shown in \cite{10.2307/2374298}.  This connection is unsurprising. It is well known that the extreme amenability of $\mathrm{Aut}(\mathbb{Q})$ (endowed with the topology of pointwise convergence) can be understood as a restatement of Ramsey's theorem. It is possible to use this to give a high brow proof of the existence of indiscernible sequences in any first-order theory $T$:
 
 \begin{proof}
Fix a first-order theory $T$. Let $Q$ be a family of variables indexed by the rational numbers. The natural action of $\mathrm{Aut}(\mathbb{Q})$ on $S_Q(T)$, the Stone space of types over $T$ in the variables $Q$, is continuous and so by extreme amenability has a fixed point. A fixed point of this action is precisely the same thing as the type of a $\mathbb{Q}$-indexed indiscernible sequence over $T$, and so we get that there are models of $T$ with indiscernible sequences.
 \end{proof}

 A similar proof of the existence of indiscernible subspaces in Banach theories (Theorem \ref{thm:exist}) is possible, but requires an argument that the analog of $S_Q(T)$ is non-empty (which follows from Dvoretzky's theorem) and also requires more delicate bookkeeping to define the analog of $S_Q(T)$ and to show that the action of the unitary group of a separable Hilbert space is continuous. In the end this is more technical than a proof using Fact \ref{fact:main} directly. 

\subsection{Asymptotically Hilbertian Spaces Do Not Interpret a Strongly Minimal Set}

For an overview of the properties of asymptotically Hilbertian spaces in the context of continuous logic, see \cite{HensonRaynaud}.

Compare the following with the fact that every discrete theory interprets a strongly minimal set.
\begin{prop}
If $T$ is the theory of an asymptotically Hilbertian space then it
does not have any non-algebraic locally compact imaginaries. In particular
it does not interpret a strongly minimal set.
\end{prop}

\begin{proof}
Assume that $T$ has a non-algebraic locally compact imaginary. By the same argument as in 
Proposition \ref{prop:loc-comp} this implies that it has a strongly minimal
imaginary. We may assume that this is definable over the unique approximately
$\aleph_0$-saturated model $\mathfrak{M}$. Since $T$ is $\aleph_1$-categorical
this implies that every model is prime over this strongly minimal
imaginary. Let $\mathfrak{A}\succ\mathfrak{B}\succ\mathfrak{M}$ be
a pair of proper elementary extensions such that in each extension
the dimension of the strongly minimal set increases by $1$. Since
$\mathfrak{B}\succ\mathfrak{M}$ is a minimal extension it must be
the case that the vector space dimension of the home sort increases
by precisely $1$. Likewise since $\mathfrak{A}\succ\mathfrak{B}$
is a minimal extension it must be the case that vector space dimension
of the home sort increases by precisely $1$. Let $b\in\mathfrak{B}\smallsetminus\mathfrak{M}$
realize the unique type of an element of norm $1$ orthogonal to $\mathfrak{M}$.
Likewise let $a\in\mathfrak{A}\smallsetminus\mathfrak{B}$ realize the
unique type of an element of norm $1$ orthogonal to $\mathfrak{B}$.
It must be the case that $\mathfrak{A}=\mathfrak{M}\oplus V$ where
$V$ is a $2$-dimensional Hilbert space generated by the orthogonal
basis $\{a,b\}$. Find $\varepsilon>0$ small enough that $\inf\{d(x,y):x\in I(\mathfrak{M}),y\in I(\mathfrak{B})\smallsetminus I(\mathfrak{M})\}>\varepsilon$
and $\inf\{d(x,y):x\in I(\mathfrak{B}),y\in I(\mathfrak{A})\smallsetminus I(\mathfrak{B})\}>\varepsilon$,
where $I$ is the strongly minimal imaginary. Such an $\varepsilon$
must exist since $I$ is strongly minimal. Now find $\delta>0$ small
enough that if $\sigma$ is a home sort automorphism fixing $\mathfrak{M}$ and 
 satisfying $d(x,\sigma(x))<\delta$ for every $x$,
then the induced automorphism of $I$, $\sigma_{I}$, satisfies $d(x,\sigma_{I}(x))<\varepsilon$
for every $x$ (such a $\delta$ must exist). Now let $\sigma$ be
the automorphism of $\mathfrak{A}$ that fixes $\mathfrak{M}$ and
rotates $V$ by an angle $\frac{\pi}{2n}$ small enough that $d(x,\sigma(x))<\delta$
for every $x\in\mathfrak{A}$. Assume without loss of generality that $\sigma^{n}(a)=b$
and $\sigma^{n}(b)=-a$. Now if we look at $\sigma_{I}^{n}$, by construction
it must be the case that $\sigma_{I}^{n}(I(\mathfrak{B}))\subseteq I(\mathfrak{B})$
and $\sigma_{I}^{n}(I(\mathfrak{B}))\not\subseteq I(\mathfrak{M})$.
Therefore if we look at the structure $\mathfrak{N}=\sigma_{I}^{n}(\mathfrak{B})$,
it must be the case that $a\in\mathfrak{N}$ and $I(\mathfrak{\mathfrak{B}})=I(\mathfrak{N})$,
but since $T$ has no imaginary Vaughtian pairs this implies that
$\mathfrak{N}=\mathfrak{A}$, which is a contradiction. 
\end{proof}

\subsection{Indiscernible Subspaces} \label{sec:ind-subsp}

\begin{defn}
 Let $T$ be a Banach {theory}. Let $\mathfrak{M}\models T$ and let $A\subseteq \mathfrak{M}$ be some set of parameters. 
 An \emph{indiscernible subspace over $A$} is a real subspace $V$ of $\mathfrak{M}$ such that for any $n<\omega$ and any $n$-tuples $\bar{b},\bar{c} \in V$, $\bar{b} \equiv_A \bar{c}$ if and only if $\bar{b} \cong \bar{c}$.

{\sloppy If $p$ is a type over $A$, then $V$ is an \emph{indiscernible subspace in $p$ (over $A$)} if it is an indiscernible subspace over $A$ and $b\models p$ for all $b\in V$ with $\left\lVert b \right\rVert = 1$.}


\end{defn}

Note that, as we have defined it, an indiscernible subspace is a real subspace even if $T$ is a complex Banach theory. Also note that an indiscernible subspace in $p$ is not literally contained in the realizations of $p$, but rather has its unit sphere contained in the realizations of $p$. It might be more accurate to talk about ``indiscernible spheres,'' but we find the subspace terminology more familiar.

Indiscernible subspaces are very metrically regular.

\begin{prop}
Suppose $V$ is an indiscernible subspace in some Banach structure. Then $V$ is isometric to a real Hilbert space.

In particular, a real subspace $V$ of a Banach structure is indiscernible over $A$ if and only if it is isometric to a real Hilbert space and for every $n<\omega$ and every pair of $n$-tuples $\bar{b},\bar{c}\in V$, $\bar{b}\equiv_A\bar{c}$ if and only if for all $i,j<n$, $\left<b_i,b_j\right> = \left<c_i,c_j\right>$. 
\end{prop}
\begin{proof}
For any real Banach space $W$, if $\dim W \leq 1$, then $W$ is necessarily isometric to a real Hilbert space. If $\dim V \geq 2$, let $V_0$ be a $2$-dimensional subspace of $V$. A subspace of an indiscernible subspace is automatically an indiscernible subspace, so $V_0$ is indiscernible. For any two distinct unit vectors $a$ and $b$, indiscernibility implies that for any $r,s\in \mathbb{R}$, $\left\lVert r a + s b\right\rVert = \left\lVert s a + r b\right\rVert$, hence the unique linear map that switches $a$ and $b$ fixes $\left\lVert \cdot \right \rVert$. This implies that the automorphism group of $(V_0, \left\lVert \cdot \right \rVert)$ is transitive on the $\left\lVert \cdot \right\rVert$-unit circle. By John's theorem on maximal ellipsoids \cite{MR0030135}, the unit ball of $\left\lVert \cdot \right \rVert$ must be an ellipse, so $\left\lVert \cdot \right \rVert$ is a Euclidean norm.

Thus every $2$-dimensional real subspace of $V$ is Euclidean and so $(V,\left\lVert \cdot \right \rVert)$ satisfies the parallelogram law and is therefore a real Hilbert space.

The `in particular' statement follows from the fact that in a real Hilbert subspace of a Banach space, the polarization identity \cite[Prop.\ 14.1.2]{Blanchard2002} defines the inner product  in terms of a particular quantifier-free formula:
\begin{equation*}
\left<x, y\right> = \frac{1}{4}\left( \left\lVert x + y \right\rVert ^2 - \left\lVert x - y \right\rVert^2 \right).\footnotemark \qedhere
\end{equation*}
\end{proof}
\footnotetext{There is also a polarization identity for the complex inner product: $${\left<x, y\right>_{\mathbb{C}} = \frac{1}{4}\left( \left\lVert x + y \right\rVert ^2 - \left\lVert x - y \right\rVert^2  + i\left\lVert x - iy \right\rVert^2 - i \left\lVert x + iy \right\rVert^2  \right).}$$}

As mentioned in \cite[Cor.\ 3.9]{SHELAH2019106738}, it follows from Dvoretzky's theorem that if $p$ is a wide type and $\mathfrak{M}$ is a sufficiently saturated model, then $p(\mathfrak{M})$ contains the unit sphere of an infinite dimensional subspace isometric to a Hilbert space. We refine this by showing that, in fact, an indiscernible subspace can be found.  

\begin{thm} \label{thm:exist} 
Let $A$ be a set of parameters in a Banach {theory} $T$ and let $p$ be a wide type over $A$. For any $\kappa$, there is $\mathfrak{M} \models T$ and a subspace $V\subseteq \mathfrak{M}$ of dimension $\kappa$ such that $V$ is an indiscernible subspace in $p$ over $A$. In particular, any $\aleph_0 + \kappa+|A|$-saturated $\mathfrak{M}$ will have such a subspace.
\end{thm}
\begin{proof}
For any set $\Delta$ of $A$-formulas, call a subspace $V$ of a model $\mathfrak{N}$ of $T_A$ \emph{$\Delta$-indiscernible in $p$} if every unit vector in $V$ models $p$ and for any $n<\omega$ and any formula $\varphi \in \Delta$ of arity $n$ and any $n$-tuples $\bar{b},\bar{c} \in V$ with $\bar{b} \cong \bar{c}$, we have $\mathfrak{N}\models \varphi(\bar{b}) = \varphi(\bar{c})$.

Since $p$ is wide, there is a model $\mathfrak{N}\models T$ containing an infinite dimensional subspace $W$ isometric to a real Hilbert space such that for all $b\in W$ with $\left\lVert b \right\rVert = 1$, $b\models p$. This is an infinite dimensional $\varnothing$-indiscernible subspace in $p$.

Now for any finite set of $A$-formulas $\Delta$ and formula $\varphi$, assume that we've shown that there is a model $\mathfrak{N}\models T$ containing an infinite dimensional $\Delta$-indiscernible subspace $V$ in $p$ over $A$. We want to show that there is a $\Delta \cup \{\phi\}$-indiscernible subspace in $V$. By Fact \ref{fact:main}, for every $k<\omega$ there is a $k$-dimensional subspace $W_{k}\subseteq V$ such that for any unit vectors $b_0,\dots,b_{\ell -1},c_0,\dots,c_{\ell-1}$ in $W_{k}$ with $\bar{b}\cong\bar{c}$, we have that $|\varphi^{\mathfrak{N}}(\bar{b})-\varphi^{\mathfrak{N}}(\bar{c})| < 2^{-k}$. If we let $\mathfrak{N}_k = (\mathfrak{N}_k,W_k)$ where we've expanded the language by a fresh predicate symbol $D$ such that $D^{\mathfrak{N}_k}(x)=d(x,W_k)$, then an ultraproduct of the sequence $\mathfrak{N}_k$ will be a structure $(\mathfrak{N}_\omega,W_\omega)$ in which $W_\omega$ is an infinite dimensional Hilbert space.

\emph{Claim:} $W_\omega$ is $\Delta\cup\{\varphi\}$-indiscernible in $p$. 

\emph{Proof of claim.} Fix an $m$-ary formula $\psi \in \Delta \cup \{\varphi\}$ and let $f(k)=0$ if $\psi \in \Delta$ and $f(k)=2^{-k}$ if $\psi = \varphi$. For any $k \geq 2m$ and $b_0,\dots,b_{m-1},c_0,\dots,c_{m-1}$ in the unit ball of $W_k$, there is a $2m$ dimensional subspace $W^\prime \subseteq W_k$ containing $\bar{b},\bar{c}$. By compactness of $B(W^\prime)^m$ (where $B(X)$ is the unit ball of $X$), we have that for any $\e > 0$ there is a $\delta(\e) > 0$ such that if $|\left<b_i,b_j \right> - \left<c_i,c_j \right>| < \delta(\e)$ for all $i,j < m$, then $|\psi^{\mathfrak{N}}(\bar{b})-\psi^{\mathfrak{N}}(\bar{c})| \leq f(k) + \e$. Note that we can take the function $\delta$ to only depend on $\psi$, specifically its arity and modulus of continuity, and not on $k$, since $B(W^\prime)^m$ is always isometric to $B(\mathbb{R}^{2m})^m$. Therefore, in the ultraproduct we will have $(\forall i,j<m)|\left<b_i,b_j \right> - \left<c_i,c_j \right>| < \delta(\e) \Rightarrow |\psi^{\mathfrak{N}}(\bar{b})-\psi^{\mathfrak{N}}(\bar{c})| \leq \e$ and thus  $\bar{b}\cong \bar{c} \Rightarrow \psi^{\mathfrak{N}_\omega}(\bar{b}) = \psi^{\mathfrak{N}_\omega}(\bar{c})$, as required. \hfill \hfill $\square_{\textit{Claim}}$

Now for each finite set of $A$-formulas we've shown that there's a structure $(\mathfrak{M}_\Delta,V_\Delta)$ (where, again, $V_\Delta$ is the set defined by the new predicate symbol $D$) such that $\mathfrak{M}_\Delta \models T_A$ and $V_\Delta$ is an infinite dimensional $\Delta$-indiscernible subspace in $p$. By taking an ultraproduct with an appropriate ultrafilter we get a structure $(\mathfrak{M},V)$ where $\mathfrak{M}\models T_A$ and $V$ is an infinite dimensional subspace. $V$ is an indiscernible subspace in $p$ over $A$ by the same argument as in the claim.

Finally note that by compactness we can take $V$ to have arbitrarily large dimension and that any subspace of an indiscernible subspace in $p$ over $A$ is an indiscernible subspace in $p$ over $A$, so we get the required result.
\end{proof}

Together with the fact that wide types always exist in Banach theories  with infinite dimensional models \cite[Thm.\ 3.7]{SHELAH2019106738}, we get a corollary. 

\begin{cor} \label{cor:ind-subsp}
Every Banach {theory} with infinite dimensional models has an infinite dimensional indiscernible subspace in some model. In particular, every such theory has an infinite indiscernible set, namely any orthonormal basis of an infinite dimensional indiscernible subspace.
\end{cor}


\subsection{Minimal{ }Wide Types}

\subsubsection{Characterization of Morley Sequences in Terms of Indiscernible Subspaces}
\label{sec:mor-seq-char-ind-sub}



Compare the following Theorem \ref{thm:main} with this fact in discrete logic: If $p$ is a minimal type (i.e.\ $p$ has a unique global non-algebraic extension), then an infinite sequence of realizations of $p$ is a Morley sequence in $p$ if and only if it is an indiscernible sequence.

Here we are using the definition of Morley sequence for (possibly unstable) $A$-invariant types: Let $p$ be a global $A$-invariant type, and let $B\supseteq A$ be some set of parameters. A sequence $\{c_i\}_{i< \kappa}$ is a \emph{Morley sequence in $p$ over $B$} if for all $i< \kappa$, $\mathrm{tp}(c_i/Bc_{<i}) = p \upharpoonright Bc_{<i}$. Note that this definition of Morley sequence agrees with the standard definition for types that are stable in the sense of Lascar and Poizat (as described in \cite[Def.\ 4.1]{SHELAH2019106738}).

\begin{thm} \label{thm:main}
Let $p$ be a minimal{ }wide type over the set $A$. For $\kappa\geq \aleph_0$, a set of realizations $\{b_i\}_{i<\kappa}$ of $p$ is a Morley sequence in (the unique global minimal wide extension of) $p$ if and only if it is an orthonormal basis of an indiscernible subspace in $p$ over $A$. 
\end{thm}
\begin{proof}
All we need to show is that an orthonormal basis of an indiscernible subspace in $p$ over $A$ is a Morley sequence in $p$. The converse will follow from the fact that all Morley sequences in a fixed invariant type of the same length have the same type along with the fact that minimal wide types have a unique global wide extension, which is therefore invariant.

Let $V$ be an indiscernible subspace in $p$ over $A$. Let $\{e_i\}_{i<\kappa}$ be an orthonormal basis of $V$. By construction, $\mathrm{tp}(e_0/A) = p$. Let $q$ be the global minimal wide extension of $p$. Assume that for some $j<\kappa$ we've shown for all $i<j$ that $\mathrm{tp}(e_i/Ae_{<i}) = q \upharpoonright Ae_{<i}$. Let $W = \cspan(e_{\geq j})$. Since $V$ is an indiscernible subspace over $A$, for all unit norm $b,c\in W$, $b\equiv_{Ae_{<j}} c$, so in particular $\mathrm{tp}(b/Ae_{<j})$ is wide. Since $p$ is minimal{ }wide we must have $\mathrm{tp}(b/Ae_{<j}) = q\upharpoonright Ae_{<j}$. Therefore $\{e_i\}_{i<\kappa}$ is a Morley sequence.
\end{proof}

What is unclear at the moment is the answer to this question: 

\begin{quest}
If $p$ is a minimal wide type over the set $A$, is it stable in the sense of \cite[Def.\ 4.1]{SHELAH2019106738}? In other words, is every type $q$ extending $p$ over a model $\mathfrak{M}\supseteq A$  a definable type?
\end{quest}

\subsubsection{Strongly Minimal Wide Types} 
\label{sec:stronk-min-wide-typ}
At the moment the contents of this section are little more than an observation, but hopefully in the future it may be a fruitful one.



In \cite{SHELAH2019106738}, Shelah and Usvyatsov construct minimal wide types in an arbitrary stable theory.\footnote{More specifically, they showed that any stable wide type has a minimal wide extension.} This is analogous to the construction of minimal types in discrete stable theories (i.e.\ fork until you do not have a non-algebraic forking extension), but just as with that construction, the method in \cite{SHELAH2019106738} does not give precise control over the resulting type.

There is a natural analog of strongly minimal types in the context of wide types. The relevant notion to generalize is Definition \ref{def:strongly-minim-types}. This gives the following.

\begin{defn}
  In a Banach theory $T$, a global type $p \in S_1(\frk{C})$ is \emph{strongly minimal wide} if it is $d$-atomic in the set of wide global types.

  An arbitrary type is \emph{strongly minimal wide} if it has a unique wide global extension and that extension is strongly minimal wide. 
\end{defn}

Just as with minimal and strongly minimal types in discrete logic, in a stable, non-$\omega$-stable Banach theory (with infinite dimensional models) there may fail to be any strongly minimal wide types, although there are always minimal wide types. Even in an $\omega$-stable theory, there may be minimal wide types that are not strongly minimal wide. Nevertheless, we do have following.
\begin{prop}
  Let $T$ be an $\omega$-stable Banach theory. Any open subset $U$ of $S_1(\frk{C})$ (the space of global types) containing a wide type contains a strongly minimal wide type.
\end{prop}
\begin{proof}
  It is not hard to see that the stated proposition is equivalent to saying that strongly minimal wide types are dense in the set of wide types. This follows immediately from Proposition 3.7 in \cite{BenYaacov2008}. 
\end{proof}

Given a strongly minimal wide type $p$,  by Proposition \ref{prop:ext} we can find a definable set $D \subseteq S_1(A)$ such that $D\cap W = \{p\}$, where $W$ is the (closed) set of wide types in $S_1(A)$. Since the set of norm $1$ types is always definable (by the formula $1-\lVert x \rVert$), we can require that $D(x) \models \lVert x \rVert = 1$ as well.

The issue with continuing the analogy with strongly minimal types is that while there is an easy characterization of definable sets containing a unique non-algebraic type which happens to be strongly minimal, there is not a clear analogous characterization of definable sets containing a unique wide type which happens to be strongly minimal wide.

Without the requirement that $D(x)\models \|x\| =1$, we have a counterexample.

\begin{prop}
  Let $T$ be a Banach theory. For any zeroset $F(x) \models \|x\| = 1$, there is a definable set $D(x)$ such that $D(x) \cap \cset{\|x\| = 1} = F(x)$.
\end{prop}
\begin{proof}
  
Let $\varphi(x)$ be a $[0,1]$-valued formula such that $\cset{\varphi(x) = 1} = F(x)$.

We need to show that the set
\[
  D \coloneqq \{\mathbf{0}\}\cup\left\{p\in S_1(A): p(x)\models 0 < \lVert  x \rVert \leq   \varphi\left(\frac{x}{\lVert x \rVert}\right)\right\}
\]
is definable.\footnote{Note that while $\varphi\left( \frac{x}{\lVert  x \rVert} \right)$ is not technically a formula, $D$ is nevertheless well defined.}

 First we need to show that $D$ is closed. Let $\{q_i\}_{i \in I}$ be a net limiting to some type in $D$. There are two cases. Either $\lim_{i \in I}\lVert q_i \rVert$ is $0$ or it is strictly positive. The first case is covered by the fact that $\mathbf{0} \in D$, so assume that the second case holds. There must be some $\e >0$ and some $i\in I$ such that for all $j \geq i$, $\lVert  q_j \rVert > \e $. So now we have that $\lim_{i \in I}q_i \in D$ if and only if $\lim_{i \in I}\lVert q_i \rVert \leq \lim_{i \in I}   \varphi\left( \frac{x}{\lVert x \rVert} \right)$ by the continuity of $\varphi\left( \frac{x}{\lVert x \rVert} \right)$ away from $\mathbf{0}$.

To establish that $D$ is definable we need to show that if $\{q_i\}_{i\in I}$ limits to a type in $D$, then $\lim_{i\in I}d(q_i,D)=0$. If $\{q_i\}_{i\in I}$ limits to $\mathbf{0}$, then $d(q_i,D)\leq d(q_i,\mathbf{0})=\lVert  q_i \rVert \rightarrow 0$ as well. If $\{q_i\}$ is limiting to a point other than $\mathbf{0}$, then $\{\lVert q_i \rVert \}_{i\in I}$ is eventually uniformly positive. So there is an $i\in I$ such that for all $j \geq i$,
$$d(q_j,D)\leq d\left(q_j,q_j\min\left\{1,\lVert q_j \rVert^{-1}\varphi\left( \frac{q_j}{\lVert q_j \rVert} \right)\right\} \right)  = \lVert q_j \rVert \dotdiv   \varphi\left( \frac{q_j}{\lVert q_j \rVert} \right) . $$
We have already established that this quantity must go to $0$ if $\{q_i\}_{i\in I}$ is limiting to $D\smallsetminus \{\mathbf{0}\}$, so we have that $D$ is definable.

Now by construction we have that $D(x)\cap \cset{\varphi(x)=0}=F(x)$.
\end{proof}

This implies that given a minimal wide type $p$ we can always find a definable set $D$ such that $D\cap \cset{\|x\|=1}=\{p\}$.

A natural attempt at a characterization would be definable sets which are wide but for which
\begin{itemize}
\item[] for every formula $\varphi(x)$ and every $\e>0$, there is an $n<\omega$ and a $\delta>0$ such that if a finite dimensional subspace $V$ with dimension at least $n$ has $\dinf(a,D) < \delta$ for all $a\in V$ of norm $1$, then for any $a,b \in V$ with norm $1$, we have that $|\varphi(a)-\varphi(b)|\leq \e$.
\end{itemize}
But it is only clear that this establishes that $D$ contains a unique wide type (which is therefore minimal wide).

The fundamental problem is that in general if $D$ is a definable set and $F$ is a closed set, then $D\cap F$ may fail to be relatively definable in $F$, in the sense of Definition~\ref{defn:strong_min_and_cat_in_cont:1}. This is of course in opposition to the behavior of relative definability in discrete logic. 

Even beyond this, what we really want is $D$ to be a nice definable subspace containing a unique wide type that happens to be strongly minimal wide, but it is still entirely unclear that this is always possible.

This all leaves the following questions.
\begin{quest}
  If $D(x)$ is a definable set such that $D(x)\models \|x \| = 1$ and $D(x)$ contains a unique global wide type $p$, is $p$ strongly minimal wide?
\end{quest}
\begin{quest}
  Is there a nice characterization of those definable sets which contain a unique wide type which happens to be strongly minimal wide?

\end{quest}
\begin{quest}
  Is every strongly minimal wide type contained in a definable subspace in which it is the unique wide type?
\end{quest}

\subsection{Banach Theory Counterexamples} \label{sec:count}
Here we collect some counterexamples that may be relevant to any model theoretic development of the ideas presented in this paper.

\subsubsection{No Infinitary Ramsey-Dvoretzky-Milman Phenomena in General}

Unfortunately some elements of the analogy between the Ramsey-Dvoretzky-Milman phenomenon and discrete Ramsey theory do not work. In particular, there is no extension of Dvoretzky's theorem, and therefore Fact \ref{fact:DM-thm}, to $k \geq \omega$, even for a fixed $\e>0$. 
Recall that a linear map $T:X\rightarrow Y$ between Banach spaces is an \emph{isomorphism} if it is a continuous bijection. This is enough to imply that $T$ is invertible and that both $T$ and $T^{-1}$ are Lipschitz. An analog of Dvoretzky's theorem for $k \geq \omega$ would imply that every sufficiently large Banach space has an infinite dimensional subspace isomorphic to Hilbert space, which is known to be false.  Here we will see a specific example of this.

 The following is a well known result in Banach space {theory} (for a proof see the comment after Proposition 2.a.2 in \cite{Lindenstrauss1996}).

\begin{fact} \label{fact:no-no}
For any distinct $X,Y \in \{\ell_p: 1\leq p < \infty\} \cup \{c_0\}$, no subspace of $X$ is isomorphic to $Y$.
\end{fact}

Note that, whereas Corollary \ref{cor:ind-subsp} says that every Banach theory is consistent with the partial type of an indiscernible subspace, the following corollary says that this type can sometimes be omitted in arbitrarily large models (contrast this with the fact that the existence of an Erd\"os cardinal implies that you can find indiscernible sequences in any sufficiently large structure in a countable language \cite[Thm.\ 9.3]{Kanamori2003}).

\begin{cor} \label{cor:no-no-cor}
For $p \in [1,\infty) \smallsetminus \{2\}$, there are arbitrarily large models of $\mathrm{Th}(\ell_p)$ that do not contain any infinite dimensional subspaces isomorphic to a Hilbert space. 
\end{cor}
\begin{proof}
Fix $p \in [1,\infty) \smallsetminus \{2\}$ and $\kappa \geq \aleph_0$.
  Let $\ell_p(\kappa)$ be the Banach space of functions $f:\kappa \rightarrow \mathbb{R}$ such that $\sum_{i<\kappa} |f(i)|^p < \infty$. Note that $\ell_p(\kappa) \equiv \ell_p$.\footnote{To see this, we can find an elementary sub-structure of $\ell_p(\kappa)$ that is isomorphic to $\ell_p$: Let $\mathfrak{L}_0$ be a separable elementary sub-structure of $\ell_p(\kappa)$. For each $i<\omega$, given $\mathfrak{L}_i$, let $B_i$ be the set of all $f \in \ell_p(\kappa)$ that are the indicator function of a singleton $\{i\}$ for some $i$ in the support of some element of $\mathfrak{L}_i$. $B_i$ is countable. Let $\mathfrak{L}_{i+1}$ be a separable elementary sub-structure of $\ell_p(\kappa)$ containing $\mathfrak{L}_i\cup B_i$. $\overline{\bigcup_{i<\omega}\mathfrak{L}_{i+1}}$ is equal to the span of $\bigcup_{i<\omega} B_i$ and so is a separable elementary sub-structure of $\ell_p(\kappa)$ isomorphic to $\ell_p$.} 
 Pick a subspace $V \subseteq \ell_p(\kappa)$. If $V$ is isomorphic to a Hilbert
space, then any separable $V_0 \subseteq V$ will also be isomorphic to a Hilbert
space. There exists a countable set $A \subseteq \kappa$ such that $V_0
\subseteq \ell_p(A) \subseteq \ell_p(\kappa)$. By Fact \ref{fact:no-no}, $V_0$
is not isomorphic to a Hilbert space, which is a contradiction. Thus no such $V$
can exist.
\end{proof}



Even assuming we start with a Hilbert space we do not get an analog of the infinitary pigeonhole principle (i.e.\ a generalization of Fact \ref{fact:DM-thm}). The discussion by H\'ajek and Novotn\'y in \cite[after Thm.\ 1]{Hajek2018} of a result of Maurey \cite{Maurey1995} implies that there is a Hilbert theory $T$ with a unary predicate $P$ such that for some $\e>0$ there are arbitrarily large models $\mathfrak{M}$ of $T$ such that for any infinite dimensional subspace $V \subseteq \mathfrak{M}$ there are unit vectors $a,b\in V$ with $|P^{\mathfrak{M}}(a)-P^{\mathfrak{M}}(b)| \geq \e$. 

Stability of a theory often has the effect of making Ramsey phenomena more prevalent in its models, so there is a natural question as to whether anything similar will happen here. Recall that a function $f:S(X)\rightarrow \mathbb{R}$ on the unit sphere $S(X)$ of a Banach space $X$ is \emph{oscillation stable} if for every infinite dimensional subspace $Y \subseteq X$ and every $\e>0$ there is an infinite dimensional subspace $Z \subseteq Y$ such that for any $a,b\in S(Z)$, $|f(a)-f(b)|\leq \e$.

\begin{quest} 
Does (model theoretic) stability imply oscillation stability? That is to say, if $T$ is a stable Banach theory, is every unary formula oscillation stable on models of $T$?
\end{quest}

\subsubsection{(Type-)Definability of Indiscernible Subspaces and Complex Banach Structures} \label{subsec:comp}

A central question in the study of \insep\ categorical Banach space theories is the degree of definability of the `minimal Hilbert space' that controls a given inseparable model of the theory. Results of Henson and Raynaud in \cite{HensonRaynaud} imply that in general the Hilbert space may not be definable.  In \cite{SHELAH2019106738}, Shelah and Usvyatsov ask whether or not the Hilbert space can be taken to be type-definable or a zeroset. In Example \ref{ex:no-def} we present a simple, but hopefully clarifying, example showing that this is slightly too much to ask.

It is somewhat uncomfortable that even in complex Hilbert structures we are only thinking about \emph{real} indiscernible subspaces rather than \emph{complex} indiscernible subspaces.
In particular, our existing Definition~\ref{defn:main-defn} is incompatible with complex structure:

\begin{prop} \label{prop:no-comp}
Let $T$ be a complex Banach theory. Let $V$ be an indiscernible subspace in some model of $T$. For any non-zero $a\in V$ and $\lambda \in \mathbb{C} \smallsetminus \{0\}$, if $\lambda a \in V$, then $\lambda \in \mathbb{R}$.
\end{prop}
\begin{proof}
Assume that for some non-zero vector $a$, both $a$ and $ia$ are in $V$. We have that $(a,ia)\equiv(ia,a)$, but $(a,ia)\models d(ix,y)=0$ and $(ia,a)\not\models d(ix,y)=0$, which contradicts indiscernibility. Therefore we cannot have that both $a$ and $ia$ are in $V$. The same statement for $a$ and $\lambda a$ with $\lambda \in \mathbb{C}\smallsetminus \mathbb{R}$ follows immediately, since $a,\lambda a \in V \Rightarrow ia \in V$.  
\end{proof}

We could define a notion of \emph{complex indiscernible subspaces} in which types are uniquely determined by (complex-valued) inner products, and that may be appropriate for complex Banach theories, as evidenced by the following. In the case of complex Hilbert space and other Hilbert spaces with a unitary Lie group action, Proposition~\ref{prop:no-comp} is the reason that indiscernible subspaces can fail to be type-definable. We will explicitly give the simplest example of this.



\begin{ex} \label{ex:no-def}
Let $T$ be the theory of an infinite dimensional complex Hilbert space and let $\mathfrak{C}$ be the monster model of $T$. $T$ is \insep\ categorical, but for any partial type $\Sigma$ over any small set of parameters $A$, $\Sigma(\mathfrak{C})$ is not an infinite dimensional indiscernible (real) subspace (over $\varnothing$).

\end{ex}
\begin{proof} [Verification]
$T$ is clearly \insep\ categorical by the same reasoning that the theory of real infinite dimensional Hilbert spaces is \insep\ categorical (being an infinite dimensional complex Hilbert space is first-order and there is a unique infinite dimensional complex Hilbert space of each infinite density character).  

If $\Sigma(\mathfrak{C})$ is not an infinite dimensional subspace of $\mathfrak{C}$, then we are done, so assume that $\Sigma(\mathfrak{C})$ is an infinite dimensional subspace of $\mathfrak{C}$. Let $\mathfrak{N}$ be a small model containing $A$. Since $\mathfrak{N}$ is a subspace of $\mathfrak{C}$, $\Sigma(\mathfrak{N}) = \Sigma(\mathfrak{C})\cap \mathfrak{N}$ is a subspace of $\mathfrak{N}$. Let $v \in \Sigma(\mathfrak{C})\smallsetminus \Sigma(\mathfrak{N})$. This implies that $v\in \mathfrak{C} \smallsetminus \mathfrak{N}$, so we can write $v$ as $v_\parallel+ v_\perp$, where $v_\parallel$ is the orthogonal projection of $v$ onto $\mathfrak{N}$ and $v_\perp$ is complex orthogonal to $\mathfrak{N}$. Necessarily we have that $v_\perp \neq 0$. Let $\mathfrak{N}^\perp$ be the orthocomplement of $\mathfrak{N}$ in $\mathfrak{C}$. If we write elements of $\mathfrak{C}$ as $(x,y)$ with $x\in \mathfrak{N}$ and $y\in \mathfrak{N}^\perp$, then the maps  $(x,y)\mapsto (x,-y)$, $(x,y)\mapsto (x,iy)$, and $(x,y)\mapsto(x,-iy)$ are automorphisms of  $\mathfrak{C}$ fixing $\mathfrak{N}$. Therefore $(v_\parallel + v_\perp) \equiv_{\mathfrak{N}} (v_\parallel - v_\perp) \equiv_{\mathfrak{N}} (v_\parallel + iv_\perp) \equiv_{\mathfrak{N}} (v_\parallel -iv_\perp)$, so we must have that $(v_\parallel - v_\perp),(v_\parallel + iv_\perp),( v_\parallel - iv_\perp) \in \Sigma(\mathfrak{C})$ as well. Since $\Sigma(\mathfrak{C})$ is a subspace, we have that $b_\perp \in \Sigma(\mathfrak{C})$ and $ib_\perp \in \Sigma(\mathfrak{C})$. Thus by Proposition \ref{prop:no-comp}, $\Sigma(\mathfrak{C})$ is not an indiscernible subspace over $\varnothing$.
\end{proof}

This example is a special case of this more general construction: If $G$ is a compact Lie group with an irreducible  unitary representation on $\mathbb{R}^n$ for some $n$ (i.e.\ the group action is transitive on the unit sphere), then we can extend this action to $\ell_2$ by taking the Hilbert space direct sum of countably many copies of the irreducible unitary representation of $G$, and we can think of this as a structure by adding function symbols for the elements of $G$. 
The theory of this structure will be totally categorical and satisfy the conclusion of Example \ref{ex:no-def}. 

Example \ref{ex:no-def} is analogous to the fact that in many strongly minimal theories the set of generic elements in a model is not itself a basis/Morley sequence. The immediate response would be to ask the question of whether or not the unit sphere of the complex linear span (or more generally the `$G$-linear span,' i.e.\ the linear span of $G\cdot V$) of the indiscernible subspace in a minimal{ }wide type agrees with the set of realizations of that minimal{ }wide type, but this can overshoot:

\begin{ex} \label{ex:bad-comp}
Consider the structure whose universe is (the unit ball of) $\ell_2 \oplus \ell_2$ (where we are taking $\ell_2$ as a real Hilbert space), with a complex action $(x,y)\mapsto (-y,x)$ and orthogonal projections $P_0$ and $P_1$ for the sets $\ell_2 \oplus \{\mathbf{0}\}$ and $\{\mathbf{0}\} \oplus \ell_2$, respectively. Let $T$ be the theory of this structure. This is a totally categorical complex Hilbert structure, but for any complete type $p$ and $\mathfrak{M}\models T $, $p(\mathfrak{M})$ does not contain the unit sphere of a non-trivial complex subspace.
\end{ex}
\begin{proof} [Verification]
$T$ is bi-interpretable with a real Hilbert space, so it is totally categorical. For any complete type $p$, there are unique values of $\left\lVert P_0(x) \right\rVert$ and $\left\lVert P_1(x) \right\rVert$ that are consistent with $p$, so the set of realizations of $p$ in any model cannot contain $\{\lambda a\}_{\lambda \in \mathrm{U}(1)}$ for $a$, a unit vector, and $\mathrm{U}(1) \subset \mathbb{C}$, the set of unit complex numbers. 
\end{proof}
The issue, of course, being that, while we declared by fiat that this is a complex Hilbert structure, the expanded structure does not respect the complex structure.

So, on the one hand, Example \ref{ex:bad-comp} shows that in general the unit sphere of the complex span won't be contained in the minimal{ }wide type. On the other hand, a priori the set of realizations of the minimal{ }wide type could contain more than just the unit sphere of the complex span, such as if we have an $\mathrm{SU}(n)$ action. The complex (or $G$-linear) span of a set is of course part of the algebraic closure of the set in question, so this suggests a small refinement of the original question of Shelah and Usvyatsov:

\begin{quest}
If $T$ is an \insep\ categorical Banach {theory}, $p$ is a  minimal{ }wide type, and $\mathfrak{M}$ is a model of $T$ which is prime over an indiscernible subspace $V$ in $p$, does it follow that $p(\mathfrak{M})$ is the unit sphere of a subspace contained in the algebraic closure of $V$? 
\end{quest}

This would be analogous to the statement that if $p$ is a strongly minimal type in an uncountably categorical discrete theory and $\mathfrak{M}$ is a model prime over a Morley sequence $I$ in $p$, then $p(\mathfrak{M})\subseteq \mathrm{acl}(I)$.

\subsubsection{Non-minimal Wide Types}

The following example shows, unsurprisingly, that Theorem \ref{thm:main} does not hold for non-minimal{ }wide types. 

\begin{ex}
Let $T$ be the theory of (the unit ball of) the infinite Hilbert space sum $\ell_2 \oplus \ell_2 \oplus \dots$, where we add a predicate $D$ that is the distance to $S^\infty \sqcup S^\infty \sqcup \dots$, where $S^\infty$ is the unit sphere of the corresponding copy of $\ell_2$. This theory is $\omega$-stable. The partial type $\{D = 0\}$ has a unique global non-forking extension $p$ that is wide, but the unit sphere of the linear span of any Morley sequence in $p$ is not contained in $p(\mathfrak{C})$.
\end{ex}
\begin{proof} [Verification]
This follows from the fact that on $D$ the equivalence relation `$x$ and $y$ are contained in a common unit sphere' is definable by a formula, namely \[E(x,y) = \inf_{z,w \in D}(d(x,z)\dotdiv 1) + (d(z,w)\dotdiv 1) + (d(w,y)\dotdiv 1),\]
where $a \dotdiv b = \max\{a-b,0\}$. If $x,y$ are in the same sphere, then let $S$ be a great circle passing through $x$ and $y$ and choose $z$ and $w$ evenly spaced along the shorter path of $S$. It will always hold that $d(x,z),d(z,w),d(w,y) \leq 1$, so we will have $E(x,y)=0$. On the other hand, if $x$ and $y$ are in different spheres, then $E(x,y)= \sqrt{2} -1$.

Therefore a Morley sequence in $p$ is just any sequence of elements of $D$ which are pairwise non-$E$-equivalent and the unit sphere of the span of any such set is clearly not contained in $D$.
\end{proof}

\subsubsection{Minimal Wide Types That are Not Strongly Minimal Wide}

The following two examples are analogous to the simplest examples of weakly minimal theories in discrete logic, namely the superstable theory of a countable sequence of independent unary predicates, each infinite and co-infinite, and the $\omega$-stable theory of a countable sequence of disjoint infinite and co-infinite unary predicates.


\label{sec:stab-with-no-stronk-min-wide}
\begin{ex}[Stable Theory with No Strongly Minimal Wide Types]\label{ex:stab-with-no-stronk-min-wide}
  Let $\frk{A}$ be (the unit ball of) the complex Hilbert space $L_2[0,1]$ (i.e.\ square integrable functions on $[0,1]$ with the standard Lebesgue measure) together with the linear operator $X$ defined by $(Xf)(x)=xf(x)$.\footnote{$\frk{A}$ can be either a real or a complex Hilbert space, but accounts of the spectral theorem are easier to find for complex Hilbert spaces.} Let $T$ be the theory of $\frk{A}$.

  The theory $T$ is superstable, but not $\omega$-stable, and has no strongly minimal wide types. 
\end{ex}
\begin{proof} [Verification]
  $X^{\frk{A}}$ is clearly a self-adjoint bounded operator whose spectrum is $[0,1]$ (as a subset of $\mathbb{C}$). The self-adjunction of $X$ is clearly first-order. Recall that an operator is invertible if and only if it is bounded below and has dense image. In a Hilbert space, if an operator fails to have dense image then this is witnessed by some unit vector that is orthogonal to the image of the operator (any element of the orthocomplement of the image). For any complex number $\lambda$, consider the sentence
  \begin{align*}
  \varphi_\lambda= \min\left\{\vphantom{\sum_{i=1}^{\infty}}\right. & \inf_x \left\lVert Xx - \lambda x \right \rVert + (1 \dotdiv \left\lVert x  \right \rVert),\\
    &\left. \sum_{i=1}^\infty 2^{-i}\inf_x\sup_y1\dotdiv d(x,s_i(Xy))\right\}.
  \end{align*}
Semantically, if $\frk{M}$ is a Hilbert structure in which $X$ is a bounded operator, $\frk{M} \models \varphi_\lambda = 0$ if and only if either $X-\lambda I$ fails to be bounded below (the first line) or $X-\lambda I$ fails to have dense image (the second line). So for any fixed $\lambda$, we have that $\frk{B} \models \varphi_\lambda=0$ if and only if $\lambda $ is in the spectrum of $X^{\frk{B}}$. So if $\frk{B} \models T$, then $X^{\frk{B}}$ must have $[0,1]$ as its spectrum. 

We want to characterize the type space $S_1(C)$ for any set of parameters $C$ in the monster model $\frk{C}$ of $T$. We may assume without loss of generality that $C$ is a linear subspace that is closed under $X$ (this is all contained in $\mathrm{dcl}(C)$).

Note that to characterize the types in $S_1(C)$, it is enough to characterize types $p(x)$ which entail $\langle  x,c \rangle=0$ for every $c \in C$ as well as $\lVert  x \rVert = 1$. Call such types \emph{orthonormal}.\footnote{Note that wide types are necessarily orthonormal, but the converse does not hold.} This is because every type $p \in S_1(C)$ has a unique $c \in C$ for which $d(p,c)$ is minimal (this, as well as everything else in this paragraph, is true in any Hilbert structure and over any set of parameters which is a subspace). If $a$ is some realization of $p$, then, assuming $p$ is not realized in $C$, we can consider the type of the vector $\frac{a-c}{\lVert  a-c \rVert}$. Let $q$ be the type of this vector. We have that $p$ and $q$ are interdefinable over $C$. This means that once we have a characterization of the orthonormal types, we get that every type p in $S_1(C)$ that is not realized in $C$ is uniquely determined by a triple $(c,\alpha,q)$ with $\lVert  c \rVert^2 + |\alpha|^2 \leq 1$, $\alpha\neq 0$, and $q$ an orthonormal type. Realizations of $p$ precisely correspond to vectors of the form $c+\alpha f$, with $f$ a realization of $q$. Note that if $O(C)$ is the set of orthonormal types in $S_1(C)$ (which is closed), then this argument establishes that the metric density character of $S_1(C)$ (with regards to the $d$-metric) is no greater than $\dc C + \dc O(C)$, where $\dc X$ is the density character of $X$. We will show that $\dc O(C) \leq 2^{\aleph_0}$, regardless of the choice of $C$, which will establish that $T$ is superstable.

Any vector $f \in \frk{C}$ induces a linear map $g\mapsto \langle f,g(X)f \rangle$ on the space of complex polynomials $g$. By the spectral theorem we know that this map extends uniquely to a map on the space of continuous functions from $[0,1]$ to $\mathbb{C}$. Furthermore, we know that this map takes the constant function $1$ to $1$, has operator norm $1$, and is positive semi-definite, so by the Riesz representation theorem there is a unique Borel probability measure $\mu_f$ on $[0,1]$ such that for any polynomial $g$, $\langle f,g(X)f \rangle=\int gd\mu_f$. Clearly the type of $f$ fixes $\mu_f$. We want to show that for any Borel probability measure (B.p.m.)\footnote{Note that a Borel probability measure on a metric space is automatically regular.} $\nu$ on $[0,1]$ there is an orthonormal type $p$ such that $\mu_f = \nu$ for any realization $f$ of $p$ and that an orthonormal type $p$ is determined by $\mu_f$ for any $f$ realizing $p$.

To show that such types exist for any B.p.m.\ $\nu$ on $[0,1]$, let $\frk{B}$ be a model containing $\frk{A}$ and $C$, fix a non-principal ultrafilter $\mathcal{U}$ on $\omega$, and consider the ultrapower $\frk{B}^{\mathcal{U}}$.

Construct a sequence of finite sets of intervals $\{I_n\}_{n<\omega}$ (where we require that intervals have positive length) with the following properties.
\begin{enumerate}
\item For each $n<\omega$, and each interval of the form $[i2^{-n},(i+1)2^{-n})$ for some $i<2^n$, there is precisely one interval $J\in I_n$ satisfying $J\subseteq[i2^{-n},(i+1)2^{-n}]$.
\item These are the only intervals in each $I_n$.
\item For $n\neq m$, $J\in I_n$, and $K \in I_m$, $J\cap K = \varnothing$. 
\end{enumerate}

It is not hard to construct such a sequence. Let $I_n^i$ be the interval in $I_n$ contained in $[i2^{-n},(i+1)2^{-n})$. For each $n<\omega$, let $f_n(x) = \sum_{i<2^{n}} \nu([i2^{-n},(i+1)2^{-n}))\chi_{I_n^i}(x)$, where $\chi_J(x)$ is the indicator function of $J$. 

It is clear that for each $n<\omega$, $f_n$ is a non-negative element of $L_1[0,1]$ with norm $1$. Furthermore, for any distinct $n,m<\omega$, $f_nf_m=0$. It is also not hard to show that for any fixed polynomial $p(x)$, $\int p(x)f_n(x)d\lambda \rightarrow \int p(x)d\nu$ as $n\to \infty $, where $\lambda$ is the Lebesgue measure on $[0,1]$.

Let $g$ be the element of $\frk{B}^{\mathcal{U}}$ corresponding to the sequence $\{\sqrt{f_n}\}_{n<\omega}$. By construction we have that $\langle g, p(X)g \rangle = \int p(x)d\nu$ for every polynomial $p(x)$. This implies that $\mu_g=\nu$. Also note that $\lVert g \rVert= 1$.

Now to show that $\tp(g/C)$ is orthonormal, for each $n<\omega$, let $h_n = \sqrt{f_n}$ (thought of as an element of $\frk{B}$). Let $H = \cspan\{h_n\}_{n<\omega}$. For any $c \in C$, let $c_\parallel$ be the orthogonal projection of $c$ onto $h$. Since the sequence $\langle h_n, c \rangle = \langle h_n, c_{\parallel} \rangle$ is square summable, it limits to $0$. Therefore we have that $\langle g, c \rangle=0$. Since we can do this for any $c \in C$, $g$ realizes an orthonormal type over $C$. (It actually realizes an orthonormal type over all of $\frk{B}$. Also note that in this paragraph we really only used the fact that $h_n$ is an orthonormal sequence.)

So we have that there is an orthonormal type for every B.p.m.\ on $[0,1]$.

If $p \in S_1(C)$ is an orthonormal type, we'll write $\mu_p$ for $\mu_f$ for some realization $f$ of $p$ ($\mu_p$ doesn't depend on the choice of $f$). We now need to show that for orthonormal types, $p$ is uniquely determined by $\mu_p$.

Fix orthonormal $p,q \in S_1(C)$ satisfying $\mu_p = \mu_q$. Let $\frk{B}_0$ be an elementary extension of $\frk{B}$ realizing both $p$ and $q$. Let $f$ and $g$ be these realizations. We are going to construct an elementary chain starting with $\frk{B}_0$.

Recall that there are $2^{\aleph_0}$ many B.p.m.s on $[0,1]$. Let $\{\nu_i\}_{i<2^{\aleph_0}}$ be an enumeration of the B.p.m.s on $[0,1]$. Let $\kappa = (2^{\aleph_0}\cdot \dc \frk{B}_0)^+$. By basic cardinal arithmetic, $\kappa = 2^{\aleph_0}\cdot \kappa$, so we may regard ordinals $i<\kappa$ as an ordered pair $(j,k)$, with $j<2^{\aleph_0}$ and $k<\kappa$. Given $\frk{B}_i$, construct $\frk{B}_{i+1}$ as an elementary extension of $\frk{B}_i$ that satisfies the following conditions.
\begin{itemize}
\item $\frk{B}_{i+1}$ has density character at most $\kappa$.
\item $\frk{B}_{i+1}$ realizes an orthonormal type $p\in S_1(\frk{B}_i)$ satisfying $\mu_p = \nu_j$, where $j$ is the first element of the ordered pair $(j,k)$ corresponding to $i$.
\end{itemize}

For $i<\kappa$ a limit ordinal, let $\frk{B}_i = \overline{\bigcup_{j<i}\frk{B}_j}$. Note that since $\kappa$ is a regular cardinal, $\dc \frk{B}_i\leq \kappa$.

Finally let $\frk{B}_{\kappa} = \bigcup_{i<\kappa}\frk{B}_i$. Since $\kappa$ is regular and uncountable, $\frk{B}_{\kappa}$ is complete. We are going to construct an automorphism of $\frk{B}_{\kappa}$ taking $f$ to $g$. Let $\{b_i\}_{i<\kappa}$ be an enumeration of a dense subset of $\frk{B}_\kappa$.

Let $f_0 = f$ and $g_0 = g$. We will proceed with a back-and-forth construction of length $\kappa$.

For any set of vectors $A$, let $\cspan_XA$ be the smallest closed subspace containing $A$ and closed under $X$. It is clear that the dimension of $\cspan_XA$ is no greater than $\aleph_0 + \dc A$. For any $i\leq\kappa$, let $f_{<i}$ represent the sequence $\{f_j\}_{j<i}$ and likewise for $g_{<i}$.

At even stage $i\cdot 2 < \kappa$ with $i>0$, given $f_{<i}$ and $g_{<i}$, find the first element of $\{b_i\}_{i<\kappa}$ not contained in  $\cspan_X(\frk{B}_0 f_{<i})$. Call this element $b_k$. Find a unit vector $f_i$, orthogonal to $\cspan_X(\frk{B}_0 f_{<i})$, such that $b_k \in \cspan\left( \{f_i\}\allowbreak\cup\cspan_X(\frk{B}  f_{<i}) \right)$. This is always possible by construction. (Also note that, up to multiplication by a unit norm complex number, $f_i$ is unique.) The odd stages proceed analogously.

Given the full sequences $f_{<\kappa}$ and $g_{<\kappa}$, by construction we have ensured that $\cspan_Xf_{<\kappa}=\cspan_Xf_{<\kappa}=\frk{B}_{\kappa}$. Let $\sigma_0$ be the partial function from $\frk{B}_{\kappa}$ to $\frk{B}_{\kappa}$ mapping $f_i$ to $g_i$ for each $i<\kappa$ as well as $h$ to itself for each $h\in \frk{B}_{0}$. By orthogonality, this clearly extends to a linear map on $\cspan(\frk{B}_0f_{<\kappa})$. Furthermore, the theory $T$ says that $X$ is a self-adjoint operator, which implies that for any vector $h$ in a model of $T$ and any $n,m<\omega$, $\langle X^nh,X^mh \rangle=\langle h,X^{n+m}h \rangle=\int x^{n+m}d\mu_h(x)$. So since these inner products only depend on $\mu_{f_i}$, we have that  $\sigma_0$ extends to a linear map that respects $X$ on all of $\cspan_X(\frk{B}_0f_{<\kappa})$, where by respecting $X$ we mean that $\sigma(Xh)=X\sigma(h)$. Let $\sigma$ be this extension. By construction, the domain of $\sigma$ is all of  $\frk{B}_\kappa$ and the range of $\sigma$ is also $\frk{B}_\kappa$. Since $\sigma$ is a linear map that respects $X$, it is an automorphism of $\frk{B}_\kappa$. Furthermore, we have that $\sigma(f_0)=g_0$, and so we have that $\tp(f_0/\frk{B}_0)=\tp(g_0/\frk{B}_0)$, and so in particular $\tp(f/C)=\tp(g/C)$, as required.

We have now established that $O(C)$, the set of orthonormal types over $C$, does not depend on the choice of $C$, as it always corresponds to the set of B.p.m.s on $[0,1]$. This set has cardinality $2^{\aleph_0}$ and so in particular has $\dc O(C) \leq 2^{\aleph_0}$. Thus $T$ is superstable.

To show that $T$ is not $\omega$-stable we need to show that there is an $\e > 0$ and an uncountable set of types in $O(C)$ which are $({>}\e)$-separated for any set of parameters $C$. For each $r \in [0,1]$, let $\delta_r$ be the Dirac measure centered at $r$. If we let $p_r$ be the type such that $\mu_{p_r}=\delta_r$, then we have that realizations of $p_r$ are eigenvectors of $X$ with eigenvalue $r$. Since $X$ is a self-adjoint operator, eigenvectors with distinct eigenvalues are necessarily orthogonal, therefore we have that $d(p_r,p_s)=\sqrt{2}$ for any distinct $r,s\in [0,1]$. Therefore no type space $S_1(C)$ has countable metric density character, and $T$ is not $\omega$-stable. 

Now, finally, we need to show that $T$ has no strongly minimal wide types. To do this it is sufficient to show that no $p \in O(C)$ is relatively $d$-atomic. Given $p \in O(C)$, we will construct a sequence $q_i \in O(C)$ limiting to $p$ topologically but for which there is some $\e > 0$ such that $d(p,q_i)>\e $ for all $i<\omega$. 

\emph{Claim:} The logical topology on $O(C)$ agrees with the topology of weak convergence of measure (i.e.\ the coarsest topology for which $\mu \mapsto \int f d\mu$ is continuous for each continuous function $f:[0,1]\rightarrow \mathbb{R}$), thinking of the types as their corresponding B.p.m.s.

\emph{Proof of claim.} It is clear that the logical topology on $O(C)$ is finer than the topology of weak convergence, since for any polynomial $f(x)$, the function $O(C)\rightarrow \mathbb{R}$ defined by $p\mapsto \int f(x)d\mu_p$ is continuous, and these functions generate the topology of weak convergence. Since polynomials are dense in the space of continuous functions on $[0,1]$ under the uniform norm, we have that $p\mapsto \int f(x)d\mu_p$ is continuous for each continuous $f:[0,1]\to \mathbb{R}$. Now the fact that the topologies agree follows from the fact that if $\tau_0$ and $\tau _1$ are two compact Hausdorff topologies on the same set such that $\tau_0 \subseteq \tau _1$, then in fact $\tau_0 = \tau_1$.\hfill \hfill $\square_{\text{claim}}$

So now given $p\in O(C)$ we just need to construct a sequence of B.p.m.s $\{\nu_i\}_{i<\omega}$ converging weakly to $\mu_p$, but which are all singular with regards to it, which will guarantee that $\lVert \nu_i - \mu_p \rVert_{\mathrm{tv}}=2$ for all $i<\omega$, where $\lVert  \cdot \rVert_{\mathrm{tv}}$ is the total variation norm. Then we will use this to give a lower bound on $d(p,q_i)$, where $q_i$ satisfies $\mu_{q_i}=\nu_i$, for all $i<\omega$ (in particular we will show that any realization of $p$ and any realization of $q_i$ must be orthogonal).

Let $A$ be the set of all $r \in [0,1]$ such that $\mu_p(\{r\})> 0$. Since $\mu_p$ is a probability measure, $A$ can be at most countable. For each $i<\omega$, let $\nu_i$ be a finite sum of Dirac measures, $\sum_{k<2^i}\mu_p([k2^{-i},(k+1)2^{-i})) \delta_{r_k^i}$, with each $r_k^i$ chosen so that $r_k^i \in [k2^{-i},(k+1)2^{-i})\smallsetminus A$. This is always possible since $A$ is at most countable.

Clearly we have by construction that $\nu_i$ limits to $\mu_p$ weakly, so now we just need to show that that realizations of $p$ and $q_i$ must actually be orthogonal. Let $a$ be a realization of $p$ and $b$ a realization of $q_i$. For any $\e > 0$, find an open set $U\subset [0,1]$ not containing any $r_k^i$ for any $i<2^{k}$ such that $\mu_p(U)> 1 - \e $. Find a real polynomial $f(x)$ such that for any $x \in U$, $|f(x)-1|<\e$; for each $i<2^{n}$, $|f(r_k^i)|<\e$; and for each $x \in [0,1]$, $-\e < f(x) < 1+\e$.

Consider $\langle a,f(X)b \rangle$. Note that since $f$ is a real polynomial, $f(X)$ is a self-adjoint operator, so we have that $\langle a,f(X)b \rangle=\langle f(X)a,b \rangle$. $|\langle  f(X)b,f(X)b \rangle | = |\langle  b, f(X)^2b \rangle| = \left| \int f^2 d\nu_i \right| < \e^2$, so $\lVert  f(x)b \rVert < \e $. Since $a$ is a unit vector, this implies that $|\langle  a, f(X)b \rangle|= |\langle f(X)a,b \rangle| < \e $ as well. Now consider
\begin{align*}
  \langle  a - f(X)a, a- f(X)a \rangle &= \langle  a,a \rangle - \langle  a, f(X)a \rangle - \langle  f(X)a,a \rangle + \langle f(X)a,f(X)a \rangle \\
  &= 1 - 2 \langle  a,f(X)a \rangle + \langle a,f(X)^2a \rangle\\
                                       &= 1 - 2 \int f d\mu_p + \int f^2 d\mu_p\\
  &= 1 + \int f^2 - 2f d\mu_p.
\end{align*}
Considering the $U$ and $[0,1]\smallsetminus U$ parts of the integral, we get that $|\langle a - f(X)a, a-f(X)a \rangle| < (\e + \e^2) + \e(2+2\e + 1 + 2\e + \e^2)=4\e+5\e^2 + \e^3$.

So now, since $b$ is a unit vector, we have that
\begin{align*}
  |\langle a,b \rangle| &\leq |\langle a-f(X)a,b \rangle| + |\langle f(X)a,b \rangle|\\
  & < 4\e + 5\e^2+\e^3 + \e = 5\e + 5\e^2 + \e^3.
\end{align*}

Since we can do this for any $\e >0$, we have that $\langle  a,b \rangle= 0$ for any $a\models p$ and $b \models q_i$, as required.

Since this is true of each $q_i$, we have that the sequence $\{q_i\}_{i<\omega}$ does not limit to $p$ in the $d$-metric and thus $p$ is not relatively $d$-atomic in $O(C)$.

Since this is true for any $p \in O(C)$ and since all wide types in $S_1(C)$ are contained in $O(C)$, $T$ has no strongly minimal wide types.
\end{proof}

It is also possible to construct an example of an $\omega$-stable theory with a minimal wide type that is not strongly minimal wide.

\begin{ex}
  Let $\frk{A}$ be (the unit ball of) the infinite Hilbert space sum $\ell_2 \oplus \ell_2 \oplus \dots$ together with a linear operator $M$ defined so that for any vector $a$ confined to the $n$th summand, $Ma = 2^{-n}a$. In other words, $M$ is a linear functional on a infinite dimensional Hilbert space whose eigenvalues are of the form $2^{-n}$ and each have an infinite dimensional eigenspace.

  Let $T$ be the theory of $\frk{A}$. $T$ is $\omega$-stable and has a unique type $p \in S_1(T)$ with the property that any realization $a$ of $p$ has norm $1$ and satisfies $Ma=0$. This type is minimal wide but not strongly minimal wide. 
\end{ex}
 \begin{proof} [Verification]
  
  It is not hard to verify that if $\frk{B}$ is any model of $T$, then $\frk{B}$ is a Hilbert space that decomposes into a direct sum of eigenspaces $\bigoplus_{\lambda \in \{0,2^{-0},2^{-1},2^{-2},\dots\}} V_\lambda$, where $V_\lambda$ is the eigenspace of $M$ corresponding to the eigenvalue $\lambda$.\footnote{Note that this relies on the fact that the spectrum of $M$ is countable. The analogous statement is not true for operators such as $X$ in Example \ref{ex:stab-with-no-stronk-min-wide}, as witnessed by the model $\frk{A}$.} It is also not hard to verify that for any such model each space $V_\lambda$ for $\lambda > 0$ must be infinite dimensional and that models of $T$ are precisely of this form. The approximately $\omega$-saturated separable model of $T$ is the one in which each $V_\lambda$ has countably infinite dimension.\footnote{It is actually $\omega$-saturated, not just approximately $\omega$-saturated.} By looking at the automorphisms of this structure, we can conclude that elements of $S_1(T)$ are in a one-to-one correspondence with square summable functions $f:\{0,2^{-0},2^{-1},2^{-2},\dots\}\to[0,1]$ and that the $d$-metric corresponds to the $\ell_2$ metric on this set. We have that the type corresponding to the function $g$ defined by $g(0)=1$ and $g(2^{-n})=0$ for all $n<\omega$ is axiomatized by $Mx=\mathbf{0}$ (i.e.\ $d(Mx,\mathbf{0})=0$). Let $p(x)$ be this type. It is clear that this type is wide.

  The space of global types $S_1(\frk{C})$ has a one-to-one correspondence with pairs $(c,f)$, where $c \in \frk{C}$ and $f$ is a function as before, satisfying $\left\lVert c \right\rVert^2_2 + \left\lVert  f \right\rVert^2_2 \leq 1$. Furthermore, we have that if $q$ and $r$ correspond to $(c,f)$ and $(e,h)$, respectively, then $d(q,r)= \sqrt{\left\lVert  c-e \right\rVert^2_2 + \left\lVert f-h  \right\rVert^2_2}$. Furthermore, we have that if $q$ and $r$ correspond to $(c,f)$ and $(e,h)$, respectively, then $d(q,r)= \sqrt{\left\lVert  c-e \right\rVert^2_2 + \left\lVert f-h  \right\rVert^2_2}$. (This characterization is true of the space of types over any model, not just $\frk{C}$, which allows us to show that $T$ is $\omega$-stable.)

  No type corresponding to a pair $(c,f)$ with $c \neq \mathbf{0}$ can be wide. This implies that the only wide global extension of $p$ is the type corresponding to $(0,g)$, and hence that $p$ is minimal wide. We also have that the types corresponding to pairs of the form $(\mathbf{0},g)$, with $g(\lambda)=1$ if $\lambda = 2^{-n}$ (for some particular $n$) and $g(\lambda) = 0$ otherwise, are wide for any $n<\omega$. Call the corresponding type $q_n$. For any such type we have that $d(p,q_n) = 1$, but nevertheless the sequence $\{q_n\}_{n<\omega}$ limits to $p$ topologically, therefore $p$ is not $d$-atomic in the set of wide global types, and is therefore not strongly minimal wide.
\end{proof}
This example is of course in some sense the Hilbert space analog of the theory of the discrete structure $\frk{M}$ in a language consisting of a countable sequence of unary predicates $\{U_i\}_{i<\omega}$ which partition $\frk{M}$ into infinite sets, with the unique minimal type not contained in any $U_i$ corresponding to $p(x)$. Whereas in discrete logic it is known that any minimal type in an uncountably categorical theory is actually strongly minimal, the analogous statements for minimal and minimal wide types in continuous logic are unknown.  
\begin{quest}
  If $T$ is an \insep\ categorical Banach theory, is every minimal wide type strongly minimal wide? 
\end{quest}

  




\bibliographystyle{plain}
\bibliography{../ref}

%


\end{document}